\documentclass[a4paper,11pt]{article}

\usepackage{amsmath}
\usepackage{color}
\usepackage[latin1]{inputenc}
\usepackage[T1]{fontenc}
\usepackage[normalem]{ulem}
\usepackage[english]{babel}
\usepackage{verbatim}
\usepackage{alphabeta}
\usepackage{mathtools}
\usepackage{graphicx}
\usepackage{enumerate,enumitem}
\usepackage{amsmath,amssymb,amsfonts,amsthm,mathrsfs}
\usepackage{rotating,relsize}
\usepackage{float}
\usepackage{array}
\usepackage{nicefrac}
\usepackage{MnSymbol}
\usepackage[all,cmtip]{xy}
 \usepackage[hidelinks]{hyperref}
\xyoption{all}
\setlist[enumerate]{leftmargin=.7cm,label=\roman*)}
\usepackage{appendix}

\newtheorem{theorem}{Theorem}[section]
\newtheorem{theoremA}{Theorem}

\newtheorem{corA}[theoremA]{Corollary}
\newtheorem{lemma}[theorem]{Lemma}
\newtheorem{prop}[theorem]{Proposition}
\newtheorem{cor}[theorem]{Corollary}

\theoremstyle{definition}
\newtheorem*{remark*}{Remark}
\newtheorem{defn}[theorem]{Definition}
\newtheorem{rem}[theorem]{Remark}

\DeclareMathOperator{\id}{id}

\DeclareMathOperator{\tr}{tr}

\DeclareMathOperator{\THH}{THH}
\DeclareMathOperator{\THR}{THR}
\DeclareMathOperator{\TCR}{TCR}
\DeclareMathOperator{\TC}{TC}
\DeclareMathOperator{\TRR}{TRR}

\DeclareMathOperator{\TR}{TR}
\DeclareMathOperator{\GW}{GW}
\DeclareMathOperator{\GL}{GL}

\DeclareMathOperator{\res}{res}
\DeclareMathOperator{\KR}{KR}
\DeclareMathOperator{\K}{K}

\DeclareMathOperator{\Z}{\mathbb{Z}}
\DeclareMathOperator{\F}{\mathbb{F}}
\DeclareMathOperator{\sd}{sd}
\DeclareMathOperator{\tran}{tran}

\DeclareMathOperator{\Lt}{L}
\DeclareMathOperator{\Mod}{Mod}

\DeclareMathOperator{\EM}{H}
\DeclareMathOperator{\M}{M}
\DeclareMathOperator{\W}{W}

\begin{document}
\begin{center}\LARGE{An analogue of the Milnor conjecture for the de Rham-Witt complex in characteristic 2}
\end{center}

\begin{center}\Large{Emanuele Dotto}
\end{center}

\let\thefootnote\relax

\vspace{.05cm}

\abstract{
We describe the modulo $2$ de Rham-Witt complex of a field of characteristic $2$, in terms of the powers of the augmentation ideal of the $\Z/2$-geometric fixed points of real topological restriction homology $\TRR$. This is analogous to the conjecture of Milnor, proved in \cite{Kato} for fields of characteristic $2$, which describes the  modulo $2$ Milnor K-theory in terms of the powers of the augmentation ideal of the Witt group of symmetric forms. Our proof provides a somewhat explicit description of these objects, as well as a calculation of the homotopy groups of the geometric fixed points of $\TRR$ and of real topological cyclic homology, for all fields.
}

\vspace{.05cm}

\section*{Introduction}

Let $k$ be a field. Let us recall that Milnor conjectured, in \cite{Milnor}, \cite{Milnor1}, that a certain canonical map of graded rings
\[
\K^M_\ast(k)/2\longrightarrow I^{\ast}/I^{\ast +1},
\]
should be an isomorphism.
Here $\K^M_\ast(k)$ is the Milnor K-theory of $k$, and $I:=\ker(rk\colon \W^s(k)\to\Z/2)$ is the augmentation ideal of the Witt group $\W^s(k)$ of symmetric forms over $k$. This conjecture was proved in \cite{Kato} when $k$ has characteristic $2$, and subsequently in \cite{OVV, VoeMilnor, Morel} in characteristic different from $2$.

The starting point of our paper is the following, somewhat overloaded observation. On one side of this isomorphism, we have a ``symbolic version'' $\K^M_\ast(k)$ of the algebraic K-theory spectrum $\K(k)$ of $k$. On the other side, we have the Witt group $\W^s(k)$, which is $\pi_0$ of the $\Z/2$-geometric fixed-points spectrum of a certain canonical $\Z/2$-equivariant refinement $\KR(k)$ of $\K(k)$ (see \cite{9II,9III}). One may then wonder if a similar relationship holds for other functors closely related to algebraic K-theory.
One does not need to look far for another such example, which is already provided in Kato's proof of Milnor's conjecture in characteristic $2$: The de Rham complex $\Omega^\ast_k$ is a ``symbolic version'' of the topological Hochschild homology spectrum $\THH(k)$, and $\THH(k)$ also admits a canonical $\Z/2$-equivariant refinement $\THR(k)$. The calculation of \cite[Corollary 5.2]{THRmodels} provides an isomorphism between $\pi_0$ of the $\Z/2$-geometric fixed points of this spectrum and $(k\otimes_S k)/2$, where $S\leq k$ is the subfield generated by the squares. The result analogous to Milnor's conjecture is then an isomorphism
\[
\Omega^\ast_k/2\stackrel{\cong}{\longrightarrow}J^{\ast}/J^{\ast +1},
\]
where $J$ is the kernel of the multiplication map $\mu\colon (k\otimes_S k)/2\to k/2$. Let us point out that, if $2$ is a unit in $k$, the source and target of this map are clearly zero, so that this statement has content only when the characteristic of $k$ is $2$.
It seems to be a standard result that this map is an isomorphism, and it plays an important role in the proof of \cite[Lemma 7(3)]{Kato} (see \cite{Arason} for a proof, which we recast in Lemma \ref{lemma:trunc0}). 
The main goal of our paper is to establish an analogous result for topological restriction homology, whose ``symbolic version'' is the de Rham-Witt complex of Bloch, Deligne and Illusie \cite{Ill}. 

Let $\W_{\langle2^\bullet\rangle}\Omega^\ast_k$ be the $2$-typical de Rham-Witt complex of $k$. We will take the definition of \cite{Costeanu}, as the initial object in the category of $2$-typical Witt complexes over $k$ (see also \cite{IbLarsDeRhamMixed}). 
For all integer $n\geq 0$, let $\TR^{n+1}(k;2)$ be the $2$-typical $(n+1)$-truncated  topological restriction homology of \cite{BHM} (see also \cite{AN}). Similarly to the relation between Milnor K-theory and algebraic K-theory, $\W_{\langle2^\bullet\rangle}\Omega^\ast_k$ and the homotopy groups of $\TR^{n+1}(k;2)$ agree in low degrees, and the former should be consider the symbolic version of the latter (see \cite{HdRWZp} and \cite[\S4]{GeisserHesselholt}). We recall that the spectrum $\TR^{n+1}(k;2)$ is defined as the $C_{2^n}$-fixed points of a $C_{2^n}$-equivariant structure on $\THH(k)$, where $C_{2^n}$ is the cyclic group of order $2^n$. 
This admits a $\Z/2$-equivariant refinement  $\TRR^{n+1}(k;2)$, constructed by extending the $C_{2^n}$-equivariant structure on $\THH(k)$ to an equivariant spectrum $\THR(k)$ for the dihedral group $D_{2^n}$ of order $2^{n+1}$. The fixed-point spectrum
\[
\TRR^{n+1}(k;2):=\THR(k)^{C_{2^n}}
\]
then inherits the structure of a $\Z/2$-spectrum, since $\Z/2$ is the Weyl group of $C_{2^n}$ in $D_{2^n}$. This construction is carried out in \cite{Amalie} and \cite[\S1]{geomTCR}, and we review it in \S\ref{prelim}.
There is then a canonical ring homomorphism
\[
\res^{D_{2^n}}_{C_{2^n}}\colon \pi_0\TRR^{n+1}(k;2)^{\phi \Z/2}\longrightarrow \pi_0\TR^{n+1}(k;2)/2,
\]
where $(-)^{\phi \Z/2}$ denotes the geometric fixed-points functor, and we let $J_{\langle 2^n \rangle}$ be its kernel.
There are operators between these spectra
\[
\xymatrix{
\TRR^{n+1}(k;2)^{\phi \Z/2}
\ar@<1ex>[r]^-{R}
\ar@<-1ex>[r]_-F
&
 \TRR^{n}(k;2)^{\phi \Z/2}
 \ar[l]|-V
}
\]
which correspond to the usual respective maps $R$, $V$ and $F$ on $\TR^{n+1}(k;2)$ under the restriction map above. There is also a map 
\[\sigma\colon \TRR^{n+1}(k;2)^{\phi \Z/2}\longrightarrow \TRR^{n+1}(k;2)^{\phi \Z/2}\]
of order $2$, which is induced by the action of the Weyl group of $\Z/2$ in the quotient $D_{2^{n+1}}/C_{2^n}$. It is easy to see that all these maps restrict to maps between the kernels $J_{\langle 2^n \rangle}$. The main result of the paper is the following analogue of Milnor's conjecture.

\begin{theoremA}\label{thmintro:dRW}
Let $k$ be a field of characteristic $2$.
The maps $R,F,V$ and $d:=1+\sigma$ endow the sequence $J_{\langle 2^\bullet \rangle}^\ast/J_{\langle 2^\bullet \rangle}^{\ast+1}$ with the structure of a $2$-typical Witt complex over $k$, and the unique map of $2$-typical Witt complexes over $k$
\[
(\W_{\langle2^\bullet\rangle}\Omega^\ast_k)/2\longrightarrow J_{\langle 2^\bullet \rangle}^\ast/J_{\langle 2^\bullet \rangle}^{\ast+1}
\]
is a strict isomorphism.
\end{theoremA}

Let us remark on some special cases of this theorem:

\begin{enumerate}
\item For $\ast=0$, the isomorphism of Theorem \ref{thmintro:dRW} identifies with the modulo $2$ reduction of the isomorphism $\W_{\langle2^n\rangle}(k)\cong \pi_0\TR^{n+1}(k;2)$ of \cite[Theorem F]{WittVect}, where $\W_{\langle2^n\rangle}(k)$ is the ring of $(n+1)$-truncated $2$-typical Witt vectors of $k$.
\item For $\bullet=0$, the isomorphism of Theorem \ref{thmintro:dRW} is the isomorphism $\Omega^\ast_k\cong J^\ast/J^{\ast+1}$ discussed above.
\item If $k$ has characteristic different from $2$, then both $(\W_{\langle2^\bullet\rangle}\Omega^\ast_k)/2$ and $\TRR^{n+1}(k;2)^{\phi \Z/2}$ vanish, so Theorem \ref{thmintro:dRW} in fact holds in all characteristics.
\item If $k$ is perfect of characteristic $2$, then $(\W_{\langle2^\bullet\rangle}\Omega^\ast_k)/2=0$ for $\ast>0$, and 
\[(\W_{\langle2^\bullet\rangle}\Omega^0_k)/2=\W_{\langle2^\bullet\rangle}(k)/2\cong k.\]
Similarly, in the case of perfect fields, $\pi_0\TRR^{n+1}(k;2)^{\phi\Z/2}\cong k$ and $J_{\langle 2^\bullet \rangle}=0$ (see \cite[Theorem 4.7]{geomTCR}). Thus, Theorem \ref{thmintro:dRW} has non-trivial content only for non-perfect fields of characteristic $2$.
\end{enumerate}

We prove the theorem by first explicitly calculating the homotopy groups of $\TRR^{n+1}(k;2)^{\phi\Z/2}$ in \S\ref{sec:geomTRR} (even though we really only need $\pi_0$), extending the calculations for perfect fields of \cite[\S4.2]{geomTCR}. We then use our calculation to provide generators for $\pi_0\TRR^{n+1}(k;2)^{\phi\Z/2}$ and $J_{\langle 2^\bullet \rangle}$, analogous to the canonical generators $V^{n-i}\tau_i(a)$ of the Witt vectors (see Propositions \ref{prop:genTRR} and \ref{prop:additivegen}). 
This allows us to define a Witt-complex structure on $J_{\langle 2^\bullet \rangle}^\ast/J_{\langle 2^\bullet \rangle}^{\ast+1}$, and to prove Theorem \ref{thmintro:dRW} by induction on $\bullet$, using the exact sequences of \cite[Lemma 3.5]{Costeanu}, in \S\ref{sec:dRW}.

The description of the homotopy groups of $\TRR^{n+1}(k;2)^{\phi\Z/2}$ is in Theorem \ref{thm:geomTRn}, and it is proved using the  pullbacks of \cite[Theorem 2.7]{geomTCR}. It is somewhat technical, and we won't state it here, but it is completely explicit. There is however a closely related calculation which is more straightforward to state. Let $\TCR(k;2)$ be the $2$-typical real topological cyclic homology spectrum of $k$, which we may define as the equaliser
\[
\TCR(k;2):=eq\big(\xymatrix{\TRR(k;2)\ar@<.5ex>[r]^-\id\ar@<-.5ex>[r]_-F&\TRR(k;2)}\big),
\]
where $\TRR(k;2)$ is the limit of $\TRR^{n+1}(k;2)$ over the maps $R$. Let us point out that, by \cite[Theorem A]{geomTCR}, if $2$ is a unit in $k$, then $\TCR(k;p)^{\phi\Z/2}=0$ for every prime $p$, so that we may assume that  $k$ has characteristic $2$.
Let $C_2$ act on $k\otimes_Sk$ by swapping the two tensor factors, where $S\leq k$ is the subfield of squares. Let us denote by $w$ the generator of $C_2$. The following is proved in Corollary \ref{cor:geomTCRk}.

\begin{theoremA}\label{thmintro:TCR}
Let $k$ be a field of characteristic $2$.
For every integer $l\geq 0$, there is an exact sequence
\[
0\to\pi_{2l} \TCR(k;2)^{\phi\Z/2}\to (k\otimes_S k)^{C_2}\xrightarrow{\pi-\phi} (k\otimes_S k)^{C_2}/Im(1+w)\to\pi_{2l-1} \TCR(k;2)^{\phi\Z/2}\to 0,
\]
where $\pi$ is the quotient map, and $\phi$ is the ring homomorphism defined by $\phi(a\otimes b)=ba^2\otimes b$.
\end{theoremA}
The map $\phi$ in fact determines an isomorphism $\phi\colon k\otimes_S k\to (k\otimes_S k)^{C_2}/Im(1+w)$, and is in a sense a replacement of the Frobenius of $k$ when $k$ is not perfect (see Lemma \ref{Lemma:Frobeniuslift}). It plays a crucial role in the calculations of \S\ref{sec:geomTRR} and in the description of $J_{\langle 2^\bullet \rangle}$.

In \cite[Theorem (1)]{Kato}, Kato exhibits a closely related exact sequence, involving the symmetric and quadratic Witt groups $\W^s(k)$ and $\W^q(k)$. Combined with Theorem \ref{thmintro:TCR}, it gives isomorphisms
\[\pi_{2l} \TCR(k;2)^{\phi\Z/2}\cong \W^s(k) \ \ \ \ \ \ \mbox{and} \ \ \ \ \ \   \pi_{2l-1} \TCR(k;2)^{\phi\Z/2}\cong \W^q(k) \]
for every $l\geq 0$. In fact, this identifies the homotopy groups of $\TCR(k;2)^{\phi\Z/2}$ with the genuine normal L-groups of $k$, as conjectured by Nikolaus, proved in great generality in \cite{HNS}, and verified in \cite{geomTCR} in the case of perfect fields (see Remark \ref{rem:L}). 
%
%
%
%
%

From Theorem \ref{thmintro:TCR}, we can also deduce a version of the Milnor conjecture for $\TC$. Let us choose the respective equaliser and coequaliser
\[
\xymatrix{\nu^\ast_{dRW/2}(k;2)\ar[r]&(\W_{\langle 2^{\infty}\rangle}\Omega^\ast_k)/2\ar@<.5ex>[r]^{\id}\ar@<-.5ex>[r]_{F}&(\W_{\langle 2^{\infty}\rangle}\Omega^\ast_k)/2\ar[r]&\epsilon^\ast_{dRW/2}(k;2)
}
\]
as possible symbolic versions of topological cyclic homology modulo $2$, where  $\W_{\langle 2^{\infty}\rangle}\Omega^\ast_k$ is the limit over the map $R$ of $\W_{\langle 2^{\bullet}\rangle}\Omega^\ast_k$ (and we are intentionally quotienting out $2$ before taking the equaliser). 
Now let
\[
K:=\ker\big(\res^{\Z/2}_e\colon \pi_0\TCR(k;2)^{\phi\Z/2}\longrightarrow (\pi_0\TC(k;2))^{\Z/2}/Im(1+w)\big)
\]
be the kernel of the restriction map, where $w$ is the involution on $\pi_0\TC(k;2)$ induced from the $\Z/2$-action on $\TCR(k;2)$. Let us also denote $T_{-1}:=\pi_{-1}\TCR(k;2)^{\phi\Z/2}$, which we consider as a $\pi_0\TCR(k;2)^{\phi\Z/2}$-module. The following is a $\TC$ analogue of \cite[Theorem (2)]{Kato}.

\begin{corA} For every field $k$ of characteristic $2$, there is an isomorphism of graded rings 
\[\nu_{dRW/2}^\ast(k;2)\cong K^\ast/K^{\ast+1},\]
and an isomorphism of graded $K^\ast/K^{\ast+1}$-modules
\[
\epsilon^\ast_{dRW/2}(k;2)\cong K^{\ast}T_{-1}/K^{\ast+1}T_{-1}.
\]
\end{corA}

We prove this result in \S\ref{sec:TCR}. Our argument is fairly straightforward, but relies on the Milnor conjecture at the prime $2$ and on the identification from \cite[Proposition 2.26]{CMM} of $\nu_{dRW/2}^\ast(k;2)$ and $\epsilon_{dRW/2}^\ast(k;2)$ with the respective equaliser and coequaliser
\[
\xymatrix{\nu^\ast(k)\ar[r]&\Omega^\ast_k\ar@<.5ex>[r]^-{\pi}\ar@<-.5ex>[r]_-{C^{-1}}&\Omega^\ast_k/d(\Omega^{\ast-1}_k)\ar[r]&\epsilon^\ast(k)}
\]
of the projection $\pi$ and the inverse Cartier operator $C^{-1}$. We then use Theorem \ref{thmintro:TCR} to compare $K$ with the augmentation ideal $I$ of the Witt group $\W^s(k)$. In order to carry out this last step, we need to understand the restriction map of $\pi_0\TCR(k;2)^{\phi\Z/2}$. We are unable to do this directly, and we need to employ the existence of a trace map of $\Z/2$-equivariant spectra from the real algebraic K-theory spectrum $\tr\colon\KR(k)\to \TCR(k;2)$, which lifts the K-theoretic trace of \cite{BHM} and which has a certain effect on $\pi_0$. This map will appear in forthcoming work of Harpaz-Nikolaus-Shah \cite{HNS} in the framework of real K-theory of Poincar\'e $\infty$-categories. For completeness, we will give a construction in Appendix \S\ref{sec:trace} for rings with involution $A$, by lifting the trace map of \cite{DO} from $\THR(A)$ to $\TCR(A;p)$.

\begin{theoremA}
Let $A$ be a ring with involution.
For every prime $p$, there is a map of $\Z/2$-spectra $\tr\colon \KR(A)\to \TCR(A;p)$ which forgets to the $K$-theoretic trace map of \cite{BHM}. The composite
\[
\GW^s(A)=\pi_0(\KR(A)^{\Z/2})\xrightarrow{\tr} \pi_0(\TCR(A;2)^{\Z/2})\xrightarrow{R} \pi_0(\THR(A)^{\Z/2})\cong (A^{\Z/2}\otimes A^{\Z/2})/T
\]
sends the element of the Grothendieck-Witt group $\GW^s(A)$ represented by a symmetric form $x$ on the free module $A^{\oplus n}$ to
\[
\tr(x)=\sum_{i=1}^n\big((x^{-1})_{ii}\otimes x_{ii}-(x^{-1})_{ii}x_{ii}\otimes 1\big)+n\otimes 1,
\]
where $x_{ii}$ are the diagonal entries of the matrix of $x$ for the standard basis of $A^{\oplus n}$, and $x^{-1}$ denotes the inverse matrix. Here the isomorphism describing $\pi_0(\THR(A)^{\Z/2})$ is from \cite[Theorem 5.1]{THRmodels}.
\end{theoremA}

\subsection*{Acknowledgements}
I genuinely thank Ib Madsen for encouraging me to look into a possible relation between real THH and the Milnor conjecture, which eventually led to the ideas of this paper. I also thank  Irakli Patchkoria, Thomas Read, and Damiano Testa for helpful conversations regarding some technical aspects of the project.

The author is supported by EPSRC grant EP/W019620/1. 
For the purpose of open access, the author has applied a Creative Commons Attribution (CC-BY) licence to any Author Accepted Manuscript version arising from this submission.

\tableofcontents

\section{Preliminaries on real topological Hochschild homology}\label{prelim}

Here we recall the basic definitions surrounding real topological cyclic homology. In order to streamline this section, we recast the definitions in the special case where the input is a discrete commutative ring $A$ with the trivial involution (which in the next sections of the paper will be a field $k$ of characteristic $2$). We refer the details of these constructions to \cite{THRmodels} and \cite{geomTCR}, and we will freely use the language of stable equivariant homotopy theory.

Let $O(2)$ be the infinite dihedral group, that we identify with the semi-direct product $\Z/2\rtimes S^1$ by choosing the reflection across the real axis as the generator for $\Z/2$.
The real topological Hochschild homology of $A$ is a ring $O(2)$-spectrum $\THR(A)$, whose underlying ring $S^1$-spectrum is the topological Hochschild homology spectrum $\THH(A)$, originally defined in \cite{Bok} (see also \cite{BHM} and \cite{NS}). It can be constructed, as an $O(2)$-equivariant ring orthogonal spectrum, as the geometric realisation of the dihedral bar construction 
\[
\THR(A):=|N^{di}\EM A|=|[n]\mapsto (\EM A)^{\otimes n+1}|,
\]
where $\EM A$ is (a flat model for) the Eilenberg-MacLane ring orthogonal $\Z/2$-spectrum of $A$, and $\otimes$ denotes the smash product of spectra (see \cite{THRmodels}). The action of $O(2)$ is defined from the structure of a dihedral object in the sense of \cite[S 1.5, Example 5]{FLcrossed} and \cite{LodayDihedral}, where the cyclic group $C_{n+1}$ acts in simplicial degree $n$ by rotating the $n+1$ smash factors, and the reflection acts in degree $n$ by reversing the order of the last $n$ smash factors.

Now let $p$ be a prime, $n\geq 0$ an integer, and $D_{p^n}=\Z/2\rtimes C_{p^n}$ the finite dihedral subgroup of $O(2)$ of order $2p^n$. Since the Weyl group of $C_{p^n}$ inside $D_{p^n}$ is $\Z/2$, the (genuine) fixed-points ring spectrum $\THR(A)^{C_{p^n}}$ is canonically a ring $\Z/2$-spectrum.
The inclusion of subgroups $C_{p^{n-1}}\leq C_{p^n}$ induces a restriction map $F$, also called Frobenius, and a transfer map $V$, also called Verschiebung, which are maps of $\Z/2$-spectra
\[
\xymatrix{\THR(A)^{C_{p^n}}\ar@<.5ex>[r]^-F &\THR(A)^{C_{p^{n-1}}}\ar@<.5ex>[l]^-V.}
\]
There is a further map $R$ of $\Z/2$-spectra, sometimes called restriction or truncation
\[
\THR(A)^{C_{p^n}}\stackrel{R}{\longrightarrow} \THR(A)^{C_{p^{n-1}}},
\]
which is defined from the real cyclotomic structure of $\THR(A)$ (see \cite[Definition 3.9]{polynomial}). The maps $R$ and $F$ are moreover maps of ring spectra (see \cite[Remark 3.10]{polynomial}).
On underlying spectra, these are the maps $F,V$ and $R$ of $\THH(A)$, which after applying $\pi_0$ correspond to the operators on the ring of Witt vectors with the same name, see \cite[Theorem 3.3]{WittVect}.

\begin{defn} Let $A$ be a commutative ring, and $p$ a prime. The $p$-typical truncated real topological restriction homology, real topological restriction homology, and real topological cyclic homology of $A$ are the ring $\Z/2$-spectra defined respectively as:
\begin{align*}
&\TRR^{n+1}(A;p):=\THR(A)^{C_{p^n}},
\\
&\TRR(A;p):=\lim\big(
\dots\xrightarrow{R}   \TRR^{n+1}(A;p)\xrightarrow{R}  \TRR^{n}(A;p)\xrightarrow{R} \dots\xrightarrow{R} \TRR^{1}(A;p)=\THR(A)
\big),
\\
&\TCR(A;p):=eq\big(\xymatrix{\TRR(A;p)\ar@<.5ex>[r]^-\id\ar@<-.5ex>[r]_-F&\TRR(A;p)}\big),
\end{align*}
where the map $F$ in the equaliser is induced by the Frobenius maps above, since $R$ and $F$ commute.
\end{defn}

The $\Z/2$-geometric fixed points of these spectra are characterised in \cite{geomTCR}, as we now recall.
These results will be used in \S\ref{sec:geomTRR} below, and we encourage the reader, at least for the purpose of the present paper, to take them as definitions of these objects.

In \cite[\S 1.2]{geomTCR}, we give a canonical equivalence of ring spectra
\[
\THR(A)^{\phi\Z/2}=\TRR^{1}(A;p)^{\phi \Z/2}\simeq  (\EM \underline{A})^{\phi\Z/2}\otimes_{\EM A}(\EM \underline{A})^{\phi\Z/2},
\]
where $\EM \underline{A}$ is the Eilenberg MacLane spectrum of the $\Z/2$ Mackey functor (or Tambara functor) with constant value $A$ and transfer map $2$. Its geometric fixed-points spectrum is then regarded as an $\EM A$-module via the map of ring spectra 
\[
\EM A\simeq (N^{\Z/2}_e\EM A)^{\phi\Z/2}\xrightarrow{\epsilon^{\phi\Z/2}}(\EM \underline{A})^{\phi\Z/2},
\]
where $N^{\Z/2}_e\EM A$ is the Hill-Hopkins-Ravenel norm construction of the ring spectrum $\EM A$ of \cite{HHR} and \cite{Sto}, and $\epsilon$ is the counit of the free-forgetful adjunction between commutative ring $\Z/2$-spectra and commutative ring spectra. We will call this the Frobenius module structure of $(\EM \underline{A})^{\phi\Z/2}$, and refer to \cite[\S 2.5]{THRmodels} for the details of its construction. The Weyl group of $\Z/2$ in $D_{2}=\Z/2\times C_2$ is $C_2$, and therefore $\THR(A)^{\phi\Z/2}$ is canonically a ring $C_2$-spectrum. In \cite[Lemma 1.2]{geomTCR}, we lift the equivalence above to an equivalence of ring $C_2$-spectra
\[
\THR(A)^{\phi\Z/2}\simeq \EM \underline{A}\otimes_{N^{C_2}_e\EM A}N^{C_2}_e((\EM \underline{A})^{\phi\Z/2}),
\]
where the right factor is a module by applying the norm to the map $\EM A\to (\EM \underline{A})^{\phi\Z/2}$, and the left factor is now regarded as a $C_2$-spectrum.

This $C_2$-equivariant homotopy type will help us characterise the $\Z/2$-geometric fixed points of $\TRR^{n+1}(A;p)$, inductively on $n$.
For every $n\geq 1$, the $\Z/2$-geometric fixed points of $\TRR^{n+1}(A;p)$ is equivalent to the product of $(n+1)$-copies of $\THR(A)^{\phi\Z/2}$ if $p$ is odd, see \cite[Theorem 2.1]{geomTCR}. For $p=2$, they are given by a pullback of ring spectra
\[
\xymatrix@C=60pt{
\TRR^{n+1}(A;2)^{\phi \Z/2}\ar[r]^-R\ar[d]_{(c F^{n-1}, c F^{n-1}\sigma_{n+1})}
&
\TRR^{n}(A;2)^{\phi\Z/2}\ar[d]^{(F^{n-1}, \sigma_1F^{n-1}\sigma_{n})}
\\
(\THR(A)^{\phi\Z/2})^{C_2}\times (\THR(A)^{\phi\Z/2})^{C_2}\ar[r]^-{r\times \sigma_{1}r}&\THR(A)^{\phi\Z/2}\times \THR(A)^{\phi\Z/2},
}
\]
see \cite[Theorem 2.7]{geomTCR}. Here $\sigma_{n}$ is the generator of the Weyl group of $\Z/2$ inside the quotient $D_{2^n}/C_{2^{n-1}}$, which is also of order $2$. The map
\[
c\colon (\THR(A)^{C_2})^{\phi\Z/2}\longrightarrow (\THR(A)^{\phi\Z/2})^{C_2}
\]
is a certain canonical map, and $r$ is the canonical map to the $C_2$-geometric fixed points followed by the equivalence given by the cyclotomic structure (see above \cite[Theorem 2.7]{geomTCR} for the definitions).

In \cite[Theorem A]{geomTCR}, we also characterise the real topological cyclic homology of $A$, by providing an equivalence of ring spectra
\[
\TCR(A;2)^{\phi\Z/2}\simeq\big(\xymatrix{(\THR(A)^{\phi\Z/2})^{C_2}\ar@<.5ex>[r]^-r\ar@<-.5ex>[r]_-f&\THR(A)^{\phi\Z/2}}\big),
\]
where $f$ is the forgetful map.

Finally, we will need to briefly use the existence of norm maps on $\THR(A)$ in order calculate a certain restriction map, in Propositions \ref{prop:norm} and \ref{prop:res}. To establish their existence, we simply observe that the dihedral bar construction employed above to define $\THR$ has a canonical symmetric monoidal structure, and therefore $\THR(A)$ is a strictly commutative $O(2)$-equivariant ring spectrum (provided we choose a strictly commutative and flat model for the Eilenberg-MacLane ring $C_2$-spectrum $\EM \underline{A}$, which we can achieve by a cofibrant replacement in the flat model structure of \cite{Sto, BrDuSt}). Thus, we obtain non-additive norm maps
\[
N_{H}^G\colon \pi_0\THR(A)^H\longrightarrow \pi_0\THR(A)^G
\]
for every pair of finite subgroups $H\leq G\leq O(2)$, which, when composed with a restriction map, satisfy the multiplicative double-coset formula.

\section{Real TR and real TC of fields of characteristic $2$}

\subsection{The geometric fixed points of TRR and TCR for fields of characteristic $2$}\label{sec:geomTRR}

Let $k$ be a field of characteristic $2$, and $S\leq k$ the subfield of squares. We regard $k$ as an $S$-vector space, and endow the abelian group $k\otimes_Sk$ with the involution $w$ which flips the two tensor factors.

The homotopy groups of $\TRR(k)^{\phi\Z/2}$ have been computed in \cite[Theorem 4.7, Corollary 4.8]{geomTCR} when the field $k$ is perfect, as a sum of copies of $k$. In this section we give an analogous description of these homotopy groups for a general field of characteristic $2$ (and an analogous proof), where some of the copies of $k$ appearing in the calculation for perfect fields are replaced by expressions involving $k\otimes_S k$ (which is isomorphic to $k$ if $k$ is perfect). This is Theorem \ref{thm:geomTRn} below, and its statement and proof will be the content of \S\ref{sec:geomTRR}.

The key algebraic input for extending the calculation to non-perfect fields lies in the following Lemma, which we will use several times throughout the paper. For every elementary tensor $a\otimes b\in k\otimes_Sk$, let us define
\[
\phi(a\otimes b):=ba^2\otimes b  \in ( k\otimes_S k)^{C_2},
\]
where the $C_2$-invariants on the right are with respect to the involution $w$.
We note that this map does not obviously extend to $k\otimes_Sk$, as it is unclear how to define it on a sum of elementary tensors. It will serve as a replacement of the Frobenius of $k$, and will be related to the cyclotomic structure of $\THR(k)$ by Proposition \ref{prop:lowerFR}, and to the fibre sequence of \cite[Theorem (1)]{Kato} describing the Witt groups of $k$ in Remark \ref{rem:L}.

\begin{lemma}\label{Lemma:Frobeniuslift}
The assignment $\phi$ induces a well-defined additive isomorphism $k\otimes_Sk\xrightarrow{\cong}(k\otimes_Sk)^{C_2}/Im(1+w)$. This isomorphism moreover fits into a commutative diagram
\[
\xymatrix@C=50pt{
k\otimes_Sk\ar[d]^{\mu}\ar[r]_-{\cong}^-{\phi}&(k\otimes_Sk)^{C_2}/Im(1+w)\ar[d]^{\mu}
\\
k\ar[r]^-{(-)^2}&k
}
\]
where the map $\mu$ is the multiplication map, which is an isomorphism if and only if $k$ is perfect.
\end{lemma}

\begin{proof}
It is easy to see that $\phi$ extends to a well-defined additive map after we quotient the image of $1+w$ in the target. To see that it is an isomorphism, choose a basis $k\cong \oplus_XS$ of $k$ as an $S$-vector space. This induces an isomorphism of $C_2$-equivariant abelian groups 
\[
k\otimes_Sk\cong \bigoplus_{X\times X}S
\]
where the involution on the right-hand side sends a basis element $(x,y)$ of $X\times X$ to $(y,x)$. Under this isomorphism, the map $\phi$ corresponds to the map
\[
\bigoplus_{X\times X}S\cong \bigoplus_{X}(\bigoplus_XS)\cong \bigoplus_{X}k\cong(\bigoplus_{X\times X}S)^{C_2}/Im(1+w)
\]
where the second isomorphism is the sum over $X$ of the isomorphism $k\cong \oplus_XS$, and the
last isomorphism sends the summand $x$ to the summand $(x,x)$ via the square map $(-)^2\colon k\xrightarrow{\cong}S$.
\end{proof}

We calculate the homotopy groups of $\TRR(k;2)^{\phi\Z/2}$ using the iterated pullback description of \cite[Theorem 2.7]{geomTCR}, reviewed in \S\ref{prelim}. This description relies on the $C_2$-equivariant homotopy type of $\THR(k)^{{\phi \Z/2}}$, which we calculate in Proposition \ref{prop:geofixgenuine} below using Lemma \ref{Lemma:Frobeniuslift} and the following decomposition of the geometric fixed points $\EM \underline{k}^{\phi\Z/2}$.

\begin{lemma}\label{lemma:decompgeomk}
Let $k$ be a field of characteristic $2$, and let us equip  $\EM \underline{k}^{\phi\Z/2}$ with the Frobenius module structure of \S\ref{prelim}. Then there is a natural splitting of $k$-modules
\[\EM\underline{k}^{{\phi \Z/2}} \simeq \bigoplus_{n \geq 0} \Sigma^{n} \EM(\varphi^\ast k),\]
where $\varphi=(-)^2\colon k\to k$ denotes the Frobenius homomorphism of $k$.
\end{lemma}

\begin{proof}
Since $k$ is a field, the Frobenius module structure on $\underline{k}^{{\phi \Z/2}}$ provides an equivalence of $k$-modules
\[\EM\underline{k}^{{\phi \Z/2}} \simeq \bigoplus_{n \geq 0} \Sigma^{n} \EM(\pi_n(\EM\underline{k}^{{\phi \Z/2}})).\]
Since the Frobenius module structure on $\EM\underline{k}^{{\phi \Z/2}}$ comes from a $k$-algebra $\EM k\to\EM\underline{k}^{{\phi \Z/2}}$, the action of $k$ on $\pi_n(\EM\underline{k}^{{\phi \Z/2}})$ is obtained by restricting, along the ring map $k=\pi_0\EM k\to \pi_0(\EM\underline{k}^{{\phi \Z/2}})$, the action of $\pi_0(\EM\underline{k}^{{\phi \Z/2}})$ on $\pi_n(\EM\underline{k}^{{\phi \Z/2}})$ induced by the ring structure of $\EM\underline{k}^{{\phi \Z/2}}$. The $\pi_0(\EM\underline{k}^{{\phi \Z/2}})$-module  $\pi_n(\EM\underline{k}^{{\phi \Z/2}})$ can be computed from the isotropy separation sequence as follows. The canonical ring homomorphism $\EM k=\EM\underline{k}^{\Z/2}\to\EM\underline{k}^{{\phi \Z/2}}$ induces a long exact sequence of $k$-modules
\[
\dots 
\xrightarrow{\partial}
\pi_{1}\EM k_{h \Z/2}
\to\pi_1 \EM k\to 
\pi_1 \EM\underline{k}^{{\phi \Z/2}} \xrightarrow{\partial} 
\pi_{0}\EM k_{h \Z/2}\to
\pi_{0}\EM k\to
\pi_{0} \EM\underline{k}^{{\phi \Z/2}}
\to 0.
\]
Since $\pi_n\EM k=0$ for $n>0$ and since the transfer map  $k=k_{\Z/2}\cong\pi_{0}\EM k_{h \Z/2}\to
\pi_{0}\EM k=k$ is multiplication by $2$ and hence also zero, all the connecting homomorphisms are isomorphisms of $k$-modules $\pi_n \EM\underline{k}^{{\phi \Z/2}}\cong \pi_{n-1}\EM k_{h \Z/2}$ for $n>0$. The homotopy groups of the homotopy-orbit spectra are equivalent to group-cohomology, and since $k$ is of characteristic $2$ the standard resolution
\[
\dots \xrightarrow{0}k \xrightarrow{2}k \xrightarrow{0}k \xrightarrow{2}k\to 0
\]
gives an isomorphism of $k$-modules $\pi_{n-1}\EM k_{h \Z/2}\cong H^{n-1}(\Z/2;k)\cong k$ for every $n>0$. Moreover, again because the transfer map is zero, the canonical map $k=\pi_0 \EM k\to \pi_0 \EM\underline{k}^{{\phi \Z/2}}$ is an isomorphism of rings. Thus we have completely identified the $\pi_0 \EM\underline{k}^{{\phi \Z/2}}$-module structure of  $\pi_n \EM\underline{k}^{{\phi \Z/2}}$.

It finally remains to show that, under the isomorphism $k\cong\pi_0 \EM\underline{k}^{{\phi \Z/2}}$ above, the ring map $\EM k\to  \EM\underline{k}^{{\phi \Z/2}}$ defining the Frobenius module structure induces the Frobenius $\varphi$ in $\pi_0$.
This follows either from identifying this map with the Tate-valued Frobenius, see \cite[Example IV.1.2.(i)]{NS}, or by the following direct calculation. The counit $\epsilon\colon N^{\Z/2}_e(\EM k)\to \EM\underline{k}$ induces a map on isotropy separation sequences
\[
\xymatrix{
&k\ar[d]^{\tau}\ar[dr]^-{\Delta}_-\cong
\\
(k\otimes k)_{\Z/2}\ar[d]\ar[r]&\pi_0(N_e^{\Z/2}(\EM k))^{\Z/2}\ar[d]^-{\epsilon^{\Z/2}}\ar[r]&\pi_0(N_e^{\phi\Z/2}(\EM k))^{\Z/2}\ar[d]^-{\epsilon^{\phi\Z/2}}
\\
k\ar[r]^0&k\ar[r]^-{\cong}&k
}
\]
where the map $\tau$ is the external norm map. We need to identify the composite $\epsilon^{\Z/2}\tau$ of the two vertical maps in the middle column.
This is the norm of the constant Tambara functor $\underline{k}$ associated to the commutative ring $k$, and it is therefore the Frobenius $\varphi$ (see also \cite[Example 2.18]{DKNP2} for an explicit identification of the target of $\tau$).
\end{proof}

We denote by $\EM(k\otimes_Sk,w)$ the $C_2$-equivariant Eilenberg-MacLane spectrum of the abelian group $k\otimes_Sk$ with $C_2$-action $w$ which switches the tensor factors.

\begin{prop}\label{prop:geofixgenuine} Let $k$ be a field of characteristic $2$. Then there is a natural equivalence of $C_2$-equivariant spectra
\[{\THR(k)}^{{\phi \Z/2}}\simeq \bigoplus_{n \geq 0} \Sigma^{n\rho}  \EM(k\otimes_Sk,w) \oplus \bigoplus_{ \begin{smallmatrix} (n,m) \\ 0 \leq n  < m \end{smallmatrix}} \Sigma^{n+m} {C_2}_{+} \otimes \EM(k\otimes_Sk),\]
where $\rho$ is the regular representation of $C_2$. It follows that there is a natural equivalence of spectra
\[
({\THR(k)}^{{\phi \Z/2}})^{C_2}\simeq (\bigoplus_{n \geq 0} \big((\bigoplus_{0 \leq j < n} \Sigma^{n+j} \EM(k\otimes_S k))\oplus \Sigma^{2n}  \EM(k\otimes_Sk)^{C_2}\big)) 
\oplus (\bigoplus_{\begin{smallmatrix} (n,m) \\ 0 \leq n  < m \end{smallmatrix}}  \Sigma^{n+m} \EM(k\otimes_Sk)).
\]
\end{prop}

\begin{proof} Let $\EM\underline{k}$ be the Eilenberg MacLane $C_2$-spectrum of the ring with trivial involution $k$.
Using the splitting of Lemma \ref{lemma:decompgeomk}, we obtain from \cite[Lemma 4.3]{geomTCR} an equivalence of $C_2$-spectra
\[{\THR(k)}^{{\phi \Z/2}}\simeq \bigoplus_{n \geq 0} \Sigma^{n\rho}  \EM\underline{k}\otimes_{N_{e}^{C_2}\EM k}N_{e}^{C_2}\EM(\varphi^\ast k)
 \oplus \bigoplus_{ \begin{smallmatrix} (n,m) \\ 0 \leq n  < m \end{smallmatrix}} \Sigma^{n+m} {C_2}_{+} \otimes \EM(\varphi^\ast k\otimes_k\varphi^\ast k).\]
This equivalence is moreover natural in $k$ since the decomposition of $\EM\underline{k}^{{\phi \Z/2}}$ of Lemma \ref{lemma:decompgeomk} is natural. Clearly $\varphi^\ast k\otimes_k\varphi^\ast k=k\otimes_S k$, and therefore to obtain the first decomposition of the Proposition it is sufficient to show that the canonical map
\[
\EM\underline{k}\otimes_{N_{e}^{C_2}\EM k}N_{e}^{C_2}\EM(\varphi^\ast k)\longrightarrow \EM(k\otimes_{k\otimes k}(\varphi^\ast k\otimes \varphi^\ast k))\cong \EM(k\otimes_{S}k,w)
\]
is an equivalence, where the middle term is $\pi_0$ of the underlying spectrum of the left term, with the induced involution.

Let us choose a basis of the $k$-vector space $\varphi^\ast k$, that is we write $\varphi^\ast k$ as a direct sum
\[
\varphi^\ast k\cong \bigoplus_{X}k
\] 
over some set $X$. Since the norm commutes with direct sums we obtain an equivalence of $C_2$-spectra
\begin{align*}
\EM\underline{k}\otimes_{N_{e}^{C_2}\EM k}N_{e}^{C_2}\EM(\varphi^\ast k)&\simeq \EM\underline{k}\otimes_{N_{e}^{C_2}\EM k}N_{e}^{C_2}\EM(\bigoplus_{X}k)
\\&\simeq
\bigoplus_{X\times X} \EM\underline{k}\otimes_{N_{e}^{C_2}\EM k}N_{e}^{C_2}\EM k\simeq \bigoplus_{X\times X} \EM\underline{k},
\end{align*}
where the last term is the indexed sum of $\EM\underline{k}$ with the involution on $X\times X$ that swaps the product factors. Under this equivalence, the canonical map above corresponds to the equivalence
\[
\bigoplus_{X\times X} \EM\underline{k}\simeq \EM((\bigoplus_{X}k)\otimes_k(\bigoplus_{X}k),w)\simeq \EM((\varphi^\ast k)\otimes_k(\varphi^\ast k),w)=\EM(k\otimes_Sk,w)
\]
where the middle equivalence is the tensor product of two copies of the choice of basis above.

Now let us identify the $C_2$-fixed points of ${\THR(k)}^{{\phi \Z/2}}$. Notice that $\EM(k\otimes_Sk,w) $ is a module over $\EM\underline{k}$ (via the ring map $k\to k\otimes_Sk$ that sends $a$ to $a^2\otimes 1$) and therefore its $C_2$-fixed-points spectrum is an $\EM k$-module, and therefore it decomposes canonically as a wedge of Eilenberg-MacLane spectra. Its homotopy groups are isomorphic to the Bredon homology groups
\[\pi_i^{C_2}(\Sigma^{n\rho}  \EM(k\otimes_Sk,w)) =H^{C_2}_i(S^{n\rho}; (k\otimes_Sk,w)),\]
which in turn are the homology groups of the chain complex
\[
0\xleftarrow{}(k\otimes_Sk)^{C_2}\xleftarrow{1+w}k\otimes_Sk\xleftarrow{1+w}k\otimes_Sk\xleftarrow{1+w}\dots \xleftarrow{1+w}k\otimes_Sk\xleftarrow{}0
\]
where the first non-zero group on the left is in degree $n$ and the last non-zero group on the right is in degree $2n$ (notice that all the signs on the arrows are $+$ since $k$ has characteristic $2$). It follows that all the groups below $n$ and above $2n$ vanish, that
\[
\pi_{2n}^{C_2}(\Sigma^{n\rho}  \EM(k\otimes_Sk,w))\cong (k\otimes_Sk)^{C_2}
\]
and that
\[
\pi_{i}^{C_2}(\Sigma^{n\rho}\EM(k\otimes_Sk,w))\cong (k\otimes_Sk)^{C_2}/Im(1+w)\xleftarrow{\cong}k\otimes_Sk
\]
for every $n\leq i<2n$, where the left-pointing isomorphism is the map $\phi$ from Lemma \ref{Lemma:Frobeniuslift}.
 \end{proof}

In order to calculate the homotopy groups of $\TRR(k;2)^{\phi\Z/2}$ and $\TCR(k;2)^{\phi\Z/2}$, we also need to determine the maps $r$ and $f$ (see \S\ref{prelim}), under the equivalences of Proposition \ref{prop:geofixgenuine}.
In the following proposition, the summands are arranged exactly as in Proposition \ref{prop:geofixgenuine}. In particular, the summands indexed on $(n,m)$ with $n<m$ in the source, and those indexed on $(n,m)$ with $n \neq m$ in the target, correspond to the induced summands.

\begin{prop} \label{prop:lowerFR} For any field $k$ of characteristic $2$, the maps $r,f \colon {({\THR(k)}^{{\phi \Z/2}})}^{C_2} \to {\THR(k)}^{{\phi \Z/2}}$ induce on $\pi_\ast$ the maps
\[ r,f\colon (\bigoplus_{
\begin{smallmatrix}(n,m)
\\
 n+m=\ast
\\
n>m\geq 0
\end{smallmatrix}}
k\otimes_S k 
)
\oplus
(\bigoplus_{\begin{smallmatrix}
(n,n)\\
2n=\ast
\\
n \geq 0
\end{smallmatrix}}
(k\otimes_S k )^{C_2}
)
\oplus
(\bigoplus_{\begin{smallmatrix}(n,m)
\\
n+m=\ast
\\
0\leq n<m
\end{smallmatrix}}
k\otimes_S k 
 )
 \longrightarrow
  \bigoplus_{\begin{smallmatrix}(n,m)\\
  n+m=\ast\\
n, m \geq 0\end{smallmatrix}} k\otimes_S k,\]
where $r$ kills the $(n,m)$-summands with $n<m$, maps the $(n,m)$-summands with $n>m$ to the $(n,m)$-summand via the identity, and maps the $(n,n)$-summand to the $(n,n)$-summand via the composite
\[
(k\otimes_S k )^{C_2}\stackrel{\pi}{\longrightarrow} (k\otimes_S k )^{C_2}/Im(1+w)\xrightarrow{\phi^{-1}}k\otimes_S k
\]
of the quotient map and the isomorphism of Lemma \ref{Lemma:Frobeniuslift}. The map $f$ kills the $(n,m)$-summands with $n>m$, is the fixed-points inclusion on the summand $(n,n)$, and embeds diagonally the $(n,m)$-summands with $n<m$ into the sum of the summands $(n,m)$ and $(m,n)$.
\end{prop}

\begin{proof}
 By \cite[Lemma 4.3]{geomTCR}, the map $r$ vanishes on the summands $(n,m)$ with $n<m$. By the same Lemma, under the identification of Proposition \ref{prop:geofixgenuine}, it is given on the other summands, for a fixed $n\geq 0$, by the outer composite of the maps in the diagram
\[ \hspace{-1.4cm}
 \xymatrix@C=-5pt@R=15pt{
(\Sigma^{n\rho}  \EM(k\otimes_Sk,w))^{C_2}\ar[d]^-{\simeq}_-{\text{Prop}\ \ref{prop:geofixgenuine}}
 \ar[r]&
 \Sigma^{n}  \EM(k\otimes_Sk,w)^{\phi C_2}\ar[d]^{\simeq}
 &
\Sigma^n( \EM\underline{k}\underset{N_{e}^{C_2}\EM k}{\otimes}N_{e}^{C_2}\EM(\varphi^\ast k))^{\phi C_2}\ar[l]_-{\simeq}
  \\
\displaystyle  (\bigoplus_{0 \leq j < n} \Sigma^{n+j} \EM(k\otimes_Sk))\oplus \Sigma^{2n}  \EM(k\otimes_Sk)^{C_2}\ar[r]
&
\displaystyle  \bigoplus_{0\leq j}\Sigma^{n+j} \EM\big(\frac{(k\otimes_{S}k)^{C_2}}{Im(1+w)}\big)
&
  \Sigma^{n}  (\EM\underline{k})^{\phi C_2}\underset{(N_{e}^{C_2}\EM k)^{\phi C_2}}{\otimes}(N_{e}^{C_2}\EM(\varphi^\ast k))^{\phi C_2}\ar[u]^-{\simeq}
  \\
 &&
 \Sigma^{n}  (\EM\underline{k})^{\phi C_2}\underset{\EM k}{\otimes}\EM(\varphi^\ast k)\ar[d]^{\simeq}\ar[u]^-{\simeq}
  \\
  \displaystyle  \bigoplus_{j\geq 0}\Sigma^{n+j} \EM(k\otimes_{S}k) \ar[r]^-{=}
  &
  \displaystyle  \bigoplus_{j\geq 0}\Sigma^{n+j} (\EM(\varphi^\ast k))\underset{\EM k}{\otimes}(\EM(\varphi^\ast k))\ar[r]^-{\simeq}
  &
   \displaystyle  \Sigma^{n} ( \bigoplus_{j\geq 0}\Sigma^j \EM(\varphi^\ast k))\underset{\EM k}{\otimes}(\EM(\varphi^\ast k)).
 }
 \]
 Here, the left map on the top row is the canonical map, and the right map on the top row is the equivalence of the proof of Proposition \ref{prop:geofixgenuine}. In the right column, the top vertical map is the monoidality of the geometric fixed points, the second map is the diagonal equivalence, and the third one is the splitting induced by the Frobenius module structure. The two bottom horizontal maps are the canonical equivalences.
 
Let us now consider the top left square. Its right vertical equivalence is given by splitting $\EM(k\otimes_Sk,w)^{\phi C_2}$ as the sum of its homotopy groups using the $\EM k$-module induced by the map $k\to k\otimes_Sk$ as we did in Proposition \ref{prop:geofixgenuine} for $\EM(k\otimes_Sk,w)^{C_2}$, and then by identifying these homotopy groups with the homology of the chain complex 
\[
0\xleftarrow{}(k\otimes_Sk)^{C_2}\xleftarrow{1+w}k\otimes_Sk\xleftarrow{1+w}k\otimes_Sk\xleftarrow{1+w}\cdots\ ,
\]
 where the first non-zero group on the left is in degree zero.
 The horizontal map on the second row sends the summand $j<n$ to the summand $j$ via $\phi$, and it maps the last summand to the summand $j=2n$ via the projection map (here $\phi$ appears because we used it to identify the homotopy groups of the source of the map in the proof of Proposition \ref{prop:geofixgenuine}).
 The square commutes by the naturality of the canonical map from fixed points to geometric fixed points.

 Thus, the identification of the map $r$ follows once we prove that the equivalence from the bottom left corner of the diagram to the second entry of the second row is the map $\phi$ on homotopy groups. Let $a,b\colon \mathbb{S}\to \EM k$, so that the suspension of $a\otimes b\colon \mathbb{S}\to \EM(k\otimes_S k)$ is a generator of a homotopy group of  the bottom left entry of the diagram. The composite of the equivalences up to the top right corner of the diagram sends $a\otimes b$ to the element of the homotopy group represented by $a\otimes N_e^{C_2}(b)$. The remaining two equivalences send this to the element represented by $b\cdot a\otimes b$, where the multiplication is with respect to the $k$-module action on $\varphi^\ast k$, and this is precisely $ba^2\otimes b=\phi(a\otimes b)$.

The identification of $f$ is simpler: by \cite[Lemma 4.3]{geomTCR}, it is the diagonal on the summands $(n,m)$ with $n<m$. The identification on the other summands  follows from the fact that the restriction map
\[\res^{C_2}_e \colon H^{C_2}_*(S^{n\rho}; (k\otimes_S k,w)) \to H_*(S^{2n};k\otimes_S k) \]
is the inclusion of fixed points in degree $*=2n$, and zero otherwise.
\end{proof}

\begin{cor}\label{cor:geomTCRk}
For every field $k$ of characteristic $2$, and every integer $l\geq 0$, there is an exact sequence
\[
0\to\pi_{2l} \TCR(k;2)^{\phi\Z/2}\to (k\otimes_S k)^{C_2}\xrightarrow{\pi-\phi}(k\otimes_S k)^{C_2}/Im(1+w)\to\pi_{2l-1} \TCR(k;2)^{\phi\Z/2}\to 0
\]
where $\pi$ quotients the image of $1+w$, and $\phi$ is the isomorphism of Lemma \ref{Lemma:Frobeniuslift} restricted to the fixed points. By Kato's calculation \cite[Theorem (1)]{Kato}, this identifies $\pi_{2l} \TCR(k;2)^{\phi\Z/2}$ with the symmetric Witt group of $k$, and $\pi_{2l-1} \TCR(k;2)^{\phi\Z/2}$ with the quadratic Witt group of $k$.
\end{cor}

\begin{proof}
By Proposition \ref{prop:lowerFR}, the map $r-f$ is an isomorphism in $\pi_\ast$ when restricted and corestricted to the summands $(n,m)$ with $n\neq m$. It is therefore an isomorphism in odd degrees, and its long exact sequence decomposes into exact sequences
\[
0\to\pi_{2l} \TCR(k;2)^{\phi\Z/2}\to  (k\otimes_S k)^{C_2}\oplus\bigoplus_{\begin{smallmatrix}(n, m)\\
n+m=2l
\\
n, m \geq 0\\
n\neq m\end{smallmatrix}}  k\otimes_S k\xrightarrow{r-f} \bigoplus_{\begin{smallmatrix}(n,m)\\
n+m=2l\\ n, m \geq 0\end{smallmatrix}} k\otimes_S k\to \pi_{2l-1} \TCR(k;2)^{\phi\Z/2}\to 0
\] 
for every $l\geq 0$. Again by Proposition \ref{prop:lowerFR}, the kernel and cokernel of $r-f$ are the same as those of 
\[\iota-\phi^{-1}\pi\colon (k\otimes_S k)^{C_2}\longrightarrow k\otimes_S k\]
where $\iota$ is the fixed-points inclusion.
 These are respectively isomorphic to the kernel and cokernel of $\pi-\phi$, by applying the isomorphism $\phi$ of Lemma \ref{Lemma:Frobeniuslift} to the target.

In \cite{Kato}, Kato exhibits an exact sequence 
\[
0\to \W^s(k)\to k\otimes_S k\xrightarrow{\pi-\phi}(k\otimes_S k)/Im(1+w)\to \W^q(k)\to 0
\]
where $\W^s(k)$ and $\W^q(k)$ are respectively the symmetric and quadratic Witt groups of $k$. It is easy to see that the kernel and cokernel of $\pi-\phi$ agree with those above, by restricting and corestricting the maps to the fixed points.
\end{proof}

\begin{rem}\label{rem:L}
Corollary \ref{cor:geomTCRk} in particular shows that the homotopy groups of the spectrum $\TCR(k;2)^{\phi\Z/2}$ agree with the homotopy groups of the cofibre $\Lt^n(k)$ of the canonical map
\[
\Lt^q(k)\longrightarrow\Lt(\Mod^\omega_A,\text{\Qoppa}^{gs}_k)
\]
induced by the symmetrisation map from the quadratic to the genuine Poincar\'e structure, as defined in \cite{9I,9II,9III}. This confirms a conjecture of Nikolaus, proved in \cite{HNS}, in the case of fields.
This is because the even and odd homotopy groups of $\Lt^n(k)$, in degrees greater or equal to $-1$, are respectively the Witt groups of symmetric and quadratic forms of $k$, as explained in \cite[Remark 4.6]{geomTCR}.
\end{rem}

Let us denote by $\pi\colon (k\otimes_Sk)^{C_2}\to (k\otimes_Sk)^{C_2}/Im(1+w)$ the projection map, so that for every $x\in (k\otimes_Sk)^{C_2}$ we can consider the element $(\phi^{-1}\pi)(x)$ of $k\otimes_S k$.
 For every $n\geq 0$, we define a subgroup of $(k\otimes_S k)^{C_2}$ by
\begin{align*}\phi^{n}\big((k\otimes_S k)^{C_2}\big)
:=\{&
x\in (k\otimes_S k)^{C_2}\ | \ \ \ (\phi^{-1}\pi)(x)\in (k\otimes_S k)^{C_2},  (\phi^{-1}\pi)^2(x)\in  (k\otimes_S k)^{C_2},
\\&
  (\phi^{-1}\pi)^3(x)\in  (k\otimes_S k)^{C_2},  \hdots
 , (\phi^{-1}\pi)^n(x) \in  (k\otimes_S k)^{C_2}\}
\end{align*}
where by convention $\phi^{0}((k\otimes_S k)^{C_2})=(k\otimes_S k)^{C_2}$.
 Thus, by construction, there is a well-defined map
\[
\phi^{-1}\pi\colon \phi^{n}\big((k\otimes_S k)^{C_2}\big)\longrightarrow \phi^{n-1}\big((k\otimes_S k)^{C_2}\big),
\]
for every $n\geq 1$, and a map
$
(\phi^{-1}\pi)^{n+1}\colon \phi^{n}\big((k\otimes_S k)^{C_2}\big)\to k\otimes_S k$. Let us consider the pullback
\[
\xymatrix{
\phi^{n}\big((k\otimes_S k)^{C_2}\big)\times_{k\otimes_S k}\phi^{n}\big((k\otimes_S k)^{C_2}\big)\ar[r]\ar[d] &\phi^{n}\big((k\otimes_S k)^{C_2}\big)\ar[d]^-{w(\phi^{-1}\pi)^{n+1}}
 \\
\phi^{n}\big((k\otimes_S k)^{C_2}\big)\ar[r]_-{(\phi^{-1}\pi)^{n+1}}&k\otimes_Sk
}
\]
where we keep in mind that one of the two maps which we pull back is composed with the involution $w$ of $k\otimes_S k$.

\begin{theorem} \label{thm:geomTRn} Let $k$ be a field of characteristic $2$. For any $l \geq 1$, there is an isomorphism
\[\pi_\ast {\TRR^{l+1}(k;2)}^{{\phi \Z/2}} 
\cong
\left\{ 
\begin{array}{ll}\displaystyle
\big(\bigoplus_{\begin{smallmatrix}(n,m)\\
n+m=\ast
\\
n, m \geq 0\\
n\neq m\end{smallmatrix}}
k\otimes_S k\big)\oplus   \big(
\phi^{l-1}\big((k\otimes_S k)^{C_2}\big)\times_{k\otimes_S k}\phi^{l-1}\big((k\otimes_S k)^{C_2}\big)
\big)
& \ast \mbox{even}
\\
\displaystyle
\bigoplus_{\begin{smallmatrix}(n,m)\\
n+m=\ast
\\
n, m \geq 0\\
n\neq m\end{smallmatrix}}k\otimes_S k
& \ast \mbox{odd}
\end{array}
\right.
\]
In particular, in degree zero, we obtain a ring isomorphism
\[
\pi_0 {\TRR^{l+1}(k;2)}^{{\phi \Z/2}} \cong \phi^{l-1}\big((k\otimes_S k)^{C_2}\big)\times_{k\otimes_S k}\phi^{l-1}\big((k\otimes_S k)^{C_2}\big).
\]
The maps $R,F \colon {\TRR^{l+1}(k;2)}^{{\phi \Z/2}}  \to {\TRR^{l}(k;2)}^{{\phi \Z/2}}$ and the Weyl action are described on homotopy groups as follows.

The map $R$ kills the $(n, m)$-summands with $n \neq m$, and in even degrees it sends an element $(x,y)$ of the right-hand pullback to $(\phi^{-1}\pi(x),\phi^{-1}\pi(y))$. 
 
The map $F$ kills the $(n, m)$-summands with $m<n$, embeds the $(n,m)$-summands with $n<m$ diagonally into the sum of the $(n,m)$ and $(m,n)$-summands, and in even degrees it sends an element $(x,y)$ of the right-hand pullback to $(x,x)$.
 
The Weyl action swaps the $(n,m)$-summand and the $(m,n)$-summand for all $n\neq m$, and in even degrees takes an element $(x,y)$ in the pullback to $(y,x)$.
\end{theorem}

\begin{proof} By \cite[Theorem 2.7]{geomTCR} and \S\ref{prelim}, for every $l\geq 1$, there is a pullback square of ring spectra
\[
\xymatrix@C=60pt{
\TRR^{l+1}(k;2)^{\phi \Z/2}\ar[r]^-R\ar[d]_{(c F^{l-1}, c F^{l-1}\sigma_{l+1})}
&
\TRR^{l}(k;2)^{\phi\Z/2}\ar[d]^{(F^{l-1}, \sigma_1F^{l-1}\sigma_{l})}
\\
(\THR(k)^{\phi \Z/2})^{C_2}\times (\THR(k)^{\phi \Z/2})^{C_2}\ar[r]^-{r\times \sigma_{1}r}&\THR(k)^{\phi \Z/2}\times \THR(k)^{\phi \Z/2}
}
\]
where $\sigma_l$ denotes the action of the generator of the Weyl group of $\Z/2$ in $D_{2^{l}}/C_{2^{l-1}}$, which is of order $2$. 
We prove, by induction on $l$, that the connecting homomorphism in the Mayer-Vietoris long exact sequence of this pullback square vanish, and therefore that the square gives a pullback square of homotopy groups. One can see, again by induction, that these pullbacks of homotopy groups indeed match the description of the homotopy groups of the Theorem. However, we will need to prove the vanishing of the connecting maps and the explicit description of the pullback in the same induction step.

For $l=1$, the pullback above describing ${\TRR^2(k;2)}^{{\phi \Z/2}}$ is equivalent to the pullback 
\[(\THR(k)^{\phi\Z/2})^{C_2}{\times_{\THR(k)^{\phi\Z/2}}}
(\THR(k)^{\phi\Z/2})^{C_2}\]
 along the maps $r$ and $\sigma_1 r$ (since the right vertical map in the square above is the diagonal for $l=1$). By the characterisation of $r$ of Proposition \ref{prop:lowerFR}, the map $r-\sigma_1r$ in the corresponding Mayer-Vietoris sequence is surjective in every degree. Therefore, there is a pullback
\[
\pi_\ast {\TRR^2(k;2)}^{{\phi \Z/2}}\cong 
\big(\big(\bigoplus_{\begin{smallmatrix}(n,m)\\
n, m \geq 0,n+m=\ast
\\
(n\neq m)\end{smallmatrix}}
k\otimes_S k\big)\oplus  (k\otimes_S k)^{C_2})\big)
\ {}_{r}\times_{\sigma_1 r}
\big(\big(\bigoplus_{\begin{smallmatrix}(n,m)\\
n, m \geq 0,n+m=\ast
\\
(n\neq m)\end{smallmatrix}}
k\otimes_S k\big)\oplus  (k\otimes_S k)^{C_2})\big)
\]
in even degrees, and an analogous pullback without the summands $(k\otimes_S k)^{C_2}$ in odd degrees. Here the subscripts of the product indicate which maps we are pulling back along.
By the description of $r$ from Proposition \ref{prop:lowerFR}, this is isomorphic to the pullback of the statement of the Theorem. The characterisation of the maps $R,F$ and of the Weyl action follows by the description of the corresponding maps of \cite[Theorem 2.7]{geomTCR}.

Now let $l \geq 2$, and suppose that the decomposition above holds for $\pi_\ast {\TRR^h(k;2)}^{{\phi \Z/2}}$ for all $h \leq l$, and that the maps $R,F\colon {\TRR^{h}(k;2)}^{{\phi \Z/2}}\to {\TRR^{h-1}(k;2)}^{{\phi \Z/2}}$ and $\sigma_h$ are given in homotopy groups by the formulas of the Theorem. We will show that the same holds for $l+1$. The Mayer-Vietoris sequence of the pullback square above is then (we recall that $\sigma_1F =F$)
\[\hspace{-2.8cm}
\xymatrix@C=-5pt@R=25pt{
\dots_{\ }\ar[r]^-{\partial}&
\pi_\ast {\TRR^{l+1}(k;2)}^{{\phi \Z/2}} \ar[d]^-{(c F^{l-1}, c F^{l-1}\sigma_{l+1},R)}
\\
\displaystyle\big[\big(\bigoplus_{\begin{smallmatrix}(n,m)\\
n+m=\ast
\\
n, m \geq 0
\\
n\neq m\end{smallmatrix}}
k\otimes_S k\big)\oplus  (k\otimes_S k)^{C_2}\big]\oplus 
&
\displaystyle\big[\big(\bigoplus_{\begin{smallmatrix}(n,m)\\
n+m=\ast
\\
n, m \geq 0
\\
n\neq m\end{smallmatrix}}
k\otimes_S k\big)\oplus  (k\otimes_S k)^{C_2}\big]\ar[d]^-{r\oplus \sigma_1 r-(F^{l-1},F^{l-1}\sigma_l)}\oplus 
&
\displaystyle
\big[\big(\bigoplus_{\begin{smallmatrix}(n,m)\\
n+m=\ast
\\
n, m \geq 0
\\
n\neq m\end{smallmatrix}}
k\otimes_S k\big)\oplus   \big(\phi^{l-1}((k\otimes_S k)^{C_2})\underset{k\otimes_S k}{\times}\phi^{l-1}((k\otimes_S k)^{C_2}) \big)\big]
\\
&
\displaystyle(\bigoplus_{\begin{smallmatrix}(n,m)\\
n+m=\ast\end{smallmatrix}}k\otimes_S k)
\oplus
(\bigoplus_{\begin{smallmatrix}(n,m)\\
n+m=\ast\end{smallmatrix}}k\otimes_S k)\ar[r]^-\partial
&\dots
}
\]
for $\ast$ even, and a similar expression without the fixed points terms for $\ast$ odd.
An argument completely analogous that of the proof of \cite[Theorem 4.7]{geomTCR} shows that the bottom vertical map is surjective, and identifies its kernel with the formula of the Theorem. The description of the maps $R$ and $F$ also follows by a similar argument. 
\end{proof}

\begin{rem}\label{rem:explicitiso}
From the proof of Theorem \ref{thm:geomTRn} we see that the isomorphism for the $0$-th homotopy group is explicitly given by the map
\[
(F^l,F^l\sigma)\colon \pi_0 {\TRR^{l+1}(k;2)}^{{\phi \Z/2}}\stackrel{\cong}{\longrightarrow} \phi^{l-1}\big((k\otimes_S k)^{C_2}\big)\times_{k\otimes_S k}\phi^{l-1}\big((k\otimes_S k)^{C_2}\big)
\]
where we implicitly identify the target $\pi_0 {\THR(k;2)}^{{\phi \Z/2}}$ of $F^l$ with $k\otimes_Sk$, and $\sigma:=\sigma_l$ denotes the Weyl action on the source of this map. We can see this directly as follows. Let us express the source of this map as the iterated pullback
\[
\pi_0 {\TRR^{l+1}(k;2)}^{{\phi \Z/2}}\stackrel{\cong}{\longrightarrow}(k\otimes_S k)^{C_2}{}_r\!\times_f\dots {}_r\!\times_f (k\otimes_S k)^{C_2}{}_r\!\times_{w r} (k\otimes_S k)^{C_2}{}_{f}\!\times_{wr}\dots{}_{f}\!\times_{wr} (k\otimes_S k)^{C_2}
\]
as in \cite[Remark 2.8]{geomTCR}, where the pullback has $2l$ factors, and the isomorphism is given by the map $(F^l,F^{l-1}R,F^{l-2}R^2,\dots,F R^{l-1}, F\sigma R^{l-1},\dots, F^{l-2}\sigma R^2,  F^{l-1}\sigma R,  F^{l}\sigma)$. We still have a pullback after applying $\pi_0$ because the connecting maps of the Mayer-Vietoris sequences vanish as seen in the proof of Theorem \ref{thm:geomTRn}. By Proposition \ref{prop:lowerFR}, $r=\phi^{-1}\pi$ and $f$ is the fixed points inclusion $(k\otimes_S k)^{C_2}\to k\otimes_S k$. Since $f$ is injective, the projection onto the first and last factors defines an isomorphism between this pullback and  $\phi^{l-1}\big((k\otimes_S k)^{C_2}\big)\times_{k\otimes_S k}\phi^{l-1}\big((k\otimes_S k)^{C_2}\big)$, and the composite map is indeed $(F^l,F^l\sigma)$.
\end{rem}

\subsection{The canonical generators of TRR}

We recall, that for every commutative ring $R$, the ring of $2$-typical $(n+1)$-truncated Witt vectors $\W_{\langle 2^{n}\rangle}(R)$ is the set $R^{\times n+1}$ equipped with the unique functorial ring structure which makes the Witt polynomials into ring homomorphisms (see e.g. \cite[\S 1]{Hbig}). Additively, it is generated by the elements
\[
V^{n-i}\tau_i(a)=(0,\dots,0,a,0,\dots,0),
\]
where the entry $a$ is in the $(n-i+1)$-st component, $a$ ranges through the elements of $R$ and $i=0,\dots,n$.

The goal of this section is to define canonical generators for the pullback of Theorem \ref{thm:geomTRn}, analogous to the generators $V^{n-i}\tau_i(a)$ of the $(n+1)$-truncated Witt vectors, thus providing generators for $\pi_0 {\TRR^{n+1}(k;2)}^{\phi \Z/2}$ analogous to those of $\W_{\langle 2^{n}\rangle}(k)$.

Recall that, for every elementary tensor $a\otimes b\in k\otimes_S k$, we have defined 
\[\phi(a\otimes b):=ba^2\otimes b\in (k\otimes_S k)^{C_2}\]
(see Lemma \ref{Lemma:Frobeniuslift}). Similarly, for any elementary tensor $a\otimes b\in k\otimes_S k$ and $n\geq 0$, let us iterate this construction and define
\[
\phi^n(a\otimes b):=b^{2^{n}-1}a^{2^n}\otimes b\in \phi^{n-1}\big((k\otimes_S k)^{C_2}\big),
\]
as well as $\tau_0(a\otimes b):=a\otimes b\in k\otimes_S k$. We will show in Proposition \ref{prop:genTRR} that $\phi^n(a\otimes b)$ indeed belongs to $\phi^{n-1}\big((k\otimes_S k)^{C_2}\big)$, and as a consequence the pairs defined by
\begin{align*}
\tau_n(a\otimes b)&:=(\phi^n(a\otimes b),\phi^n(b\otimes a)) \\
V^{n-i}\tau_i(a\otimes b)&:=(\phi^i(a\otimes b)+\phi^i(b\otimes a),0)\\
\sigma V^{n-i}\tau_i(a\otimes b)&:=(0,\phi^i(a\otimes b)+\phi^i(b\otimes a))
\end{align*}
for every $0\leq i<n$ belong to the pullback $\phi^{n-1}((k\otimes_S k)^{C_2})\underset{k\otimes_S k}{\times}\phi^{n-1}((k\otimes_S k)^{C_2})$.
Here we recall that the pullback is taken with respect to the maps $(\phi^{-1}\pi)^n$ and $w(\phi^{-1}\pi)^n$, see the diagram above Theorem \ref{thm:geomTRn}, and $\sigma$ is the Weyl action which switches the two pullback components.

\begin{prop}\label{prop:genTRR} Let $k$ be a field of characteristic $2$.
For every $n\geq 0$, the subgroup $\phi^{n}\big((k\otimes_S k)^{C_2}\big)$ of $k\otimes_Sk$ is generated by elements of the form $\phi^{n+1}(a\otimes b)$ and $\phi^{i}(a\otimes b)+\phi^i(b\otimes a)$, for $0\leq i\leq n$ and $a\otimes b\in k\otimes_S k$.

It follows that, for every $n\geq 1$,
\[
 \pi_0\TRR^{n+1}(k;2)^{{\phi \Z/2}}\cong \phi^{n-1}((k\otimes_S k)^{C_2})\underset{k\otimes_S k}{\times}\phi^{n-1}((k\otimes_S k)^{C_2})
\]
is generated by the elements $\tau_n(a\otimes b)$, $V^{n-i}\tau_i(a\otimes b)$ and $\sigma V^{n-i}\tau_i(a\otimes b)$, for $0\leq i\leq n-1$ and $a\otimes b\in k\otimes_Sk$.
\end{prop}

\begin{proof}
Let us first show that the proposed generators belong to $\phi^{n}\big((k\otimes_S k)^{C_2}\big)$. For the first, we see that for every $1\leq j\leq n$ we have that
\[
(\phi^{-1}\pi)^{j}(\phi^{n+1}(a\otimes b))=\phi^{n+1-j}(a\otimes b),
\]
which belongs to $(k\otimes_S k)^{C_2}$ since $n+1-j\geq 1$. On the other hand, for all $0\leq i\leq n-1$, we have that
\[
(\phi^{-1}\pi)^{j}(\phi^{i}(a\otimes b)+\phi^i(b\otimes a))=\phi^{i-j}(a\otimes b)+\phi^{i-j}(b\otimes a)
\]
if $0\leq j\leq i$, which is a fixed point, and for $i<j\leq n$ this is 
\[
(\phi^{-1}\pi)^{j}(\phi^{i}(a\otimes b)+\phi^i(b\otimes a))=(\phi^{-1}\pi)^{j-i}(a\otimes b+b\otimes a)=0
\]
since $\pi$ quotients off the image of $1+w$. 

The proof that these elements generate $\phi^{n}\big((k\otimes_S k)^{C_2}\big)$ is by induction on $n$. For $n=0$, consider the exact sequence
\[
k\otimes_S k\xrightarrow{1+w}(k\otimes_S k)^{C_2}\longrightarrow (k\otimes_S k)^{C_2}/Im(1+w)\to 0.
\]
By Lemma \ref{Lemma:Frobeniuslift}, the right term is generated by the equivalence classes of the elements of the form $\phi(a\otimes b)$, and the image of $1+w$ is generated by the elements of the form $a\otimes b+b\otimes a$, which proves the claim.

Now suppose that the claim holds for $n-1$, and consider the exact sequence
\[
k\otimes_S k\xrightarrow{1+w} \phi^{n}((k\otimes_S k)^{C_2})\longrightarrow \phi^{n}((k\otimes_S k)^{C_2})/Im(1+w)\to 0.
\]
By an argument analogous to the proof of Lemma \ref{Lemma:Frobeniuslift}, $\phi$ defines an isomorphism between $\phi^{n-1}((k\otimes_S k)^{C_2})$ and $\phi^{n}((k\otimes_S k)^{C_2})/Im(1+w)$.  Thus, by the inductive assumption, the classes of $\phi^{n+1}(a\otimes b)$ and $\phi^{i}(a\otimes b)+\phi^i(b\otimes a)$, for $1\leq i\leq n$ and $a\otimes b\in k\otimes_S k$, generate the quotient.  The image of $1+w$ is generated by the elements of the form $a\otimes b+b\otimes a$, which concludes the induction.

The proof for $\pi_0\TRR^{n+1}(k;2)^{\phi\Z/2}$ is completely analogous, by induction on the exact sequences
\[
\xymatrix{
(k\otimes_S k)\oplus (k\otimes_S k)\ar[d]^-{(1+w,0)+(0,1+w)}
\\
\phi^{n-1}((k\otimes_S k)^{C_2})\underset{k\otimes_S k}{\times}\phi^{n-1}((k\otimes_S k)^{C_2})
\ar[d]^-{(\phi^{-1}\pi,\phi^{-1}\pi)}
\\
 \phi^{n-2}((k\otimes_S k)^{C_2})\underset{k\otimes_S k}{\times}\phi^{n-2}((k\otimes_S k)^{C_2})\ar[d]
 \\0
}\]
\end{proof}

Next, we want to understand the effect of the transfer and norm maps of $\TRR(k)$ under the isomorphism of Theorem \ref{thm:geomTRn}, and their relation to the generators of Proposition \ref{prop:genTRR}.  For every $0\leq h<l$, let \[\tran^{D_{2^l}}_{D_{2^{h}}}\colon \pi_0\TRR^{h+1}(k;2)^{\phi \Z/2}\to\pi_0\TRR^{l+1}(k;2)^{\phi \Z/2}\] be the transfer map associated to the subgroup inclusion $D_{2^h}\leq D_{2^{l}}$.

\begin{prop}\label{prop:transfers} For every $0< h<l$, the map $\tran^{D_{2^l}}_{D_{2^{h}}}$ corresponds, under the isomorphism of Theorem \ref{thm:geomTRn}, to the group homomorphism
\[
V^{l-h}\colon 
 \phi^{h-1}((k\otimes_S k)^{C_2})\underset{k\otimes_S k}{\times}\phi^{h-1}((k\otimes_S k)^{C_2})
\longrightarrow 
\phi^{l-1}((k\otimes_S k)^{C_2})\underset{k\otimes_S k}{\times}\phi^{l-1}((k\otimes_S k)^{C_2})
\]
which sends $(x,y)$ to $(x+y,0)$. For $h=0$, it corresponds to
the group homomorphism
\[
V^{l}\colon 
k\otimes_S k
\longrightarrow 
\phi^{l-1}((k\otimes_S k)^{C_2})\underset{k\otimes_S k}{\times}\phi^{l-1}((k\otimes_S k)^{C_2})
\]
which sends $a\otimes b$ to $(a\otimes b+b\otimes a,0)$.
\end{prop}

\begin{proof} Let us first suppose $h>0$.
We need to show that the unique map in the bottom row of the commutative square
\[\xymatrix{
\pi_0\TRR^{h+1}(k;2)^{\phi \Z/2}
\ar[r]^-{\tran^{D_{2^l}}_{D_{2^{h}}}}\ar[d]^-{(F^h,F^h\sigma)}_-\cong
&
\pi_0\TRR^{l+1}(k;2)^{\phi \Z/2}\ar[d]^-{(F^l,F^l\sigma)}_-\cong
\\
\phi^{h-1}((k\otimes_S k)^{C_2})\underset{k\otimes_S k}{\times}\phi^{h-1}((k\otimes_S k)^{C_2})
\ar[r]^-{}&
\phi^{l-1}((k\otimes_S k)^{C_2})\underset{k\otimes_S k}{\times}\phi^{l-1}((k\otimes_S k)^{C_2})
}\]
agrees with $V^{l-h}$, where the vertical maps are the isomorphisms of Theorem \ref{thm:geomTRn} and Remark \ref{rem:explicitiso}. By the double coset formula of the $D_{2^l}$-Mackey functor $\underline{\pi}_0\THR(k)$, the upper composite has first component
\begin{align*}
F^l\tran^{D_{2^l}}_{D_{2^{h}}}&=\res^{D_{2^l}}_{\Z/2}\tran^{D_{2^l}}_{D_{2^{h}}}=\sum_{g\in \Z/2/D_{2^l}/D_{2^h}}\tran^{\Z/2}_{^gD_{2^{h}}\cap \Z/2} c_g\res^{D_{2^{h}}}_{D_{2^{h}}\cap \Z/2^g}.
\end{align*}
The double coset $\Z/2/D_{2^l}/D_{2^h}$ is the quotient of the cyclic group $C_{2^{l-h}}$ by the involution which acts by inversion. It therefore consists of two fixed points (the unit and the rotation $g_0$ of order $2$ in $D_l$) which conjugate $\Z/2$ to itself, and $(2^{l-h}-2)/2$ points whose corresponding intersection $D_{2^{h}}\cap \Z/2^g$ is trivial. Thus
\begin{align*}
F^l\tran^{D_{2^l}}_{D_{2^{h}}}&=\res^{D_{2^{h}}}_{\Z/2}+c_{g_0} \res^{D_{2^{h}}}_{\Z/2}+\sum_{1,g_0\neq g\in \Z/2/D_{2^l}/D_{2^h}}\tran^{\Z/2}_{e} c_g\res^{D_{2^{h}}}_{e}
\\&=\res^{D_{2^{h}}}_{\Z/2}+\res^{D_{2^{h}}}_{\Z/2}c_{g_0} =F^h+ F^h\sigma
\end{align*}
where the transfer $\tran^{\Z/2}_{e}$ is zero since $k$ has characteristic $2$ (see \cite[Theorem 5.1]{THRmodels}), $\sigma$ is the action of the Weyl group of $D_h$ in $D_l$, and the equality is regarded as elements of $k\otimes_S k$. The map $F^h$ is determined in Theorem \ref{thm:geomTRn}: it sends an element in the upper left corner of the square, corresponding to $(x,y)$ in the bottom left corner, to $x$. Thus the unique bottom horizontal map in the square above sends $(x,y)$ to the pair with first component
$x+y$.

Now let $\Z/2'$ be the subgroup of $D_{l}$ generated by a reflection non-conjugate to $\Z/2$. Similarly to the calculation above, the second component of the top composite is
\begin{align*}
F^l\sigma\tran^{D_{2^l}}_{D_{2^{h}}}&=\res^{D_{2^l}}_{\Z/2'}\tran^{D_{2^l}}_{D_{2^{h}}}=\sum_{g\in \Z/2'/D_{2^l}/D_{2^h}}\tran^{\Z/2}_{^gD_{2^{h}}\cap \Z/2} c_g\res^{D_{2^{h}}}_{D_{2^{h}}\cap \Z/2^g}.
\end{align*}
Now the double coset $\Z/2'/D_{2^l}/D_{2^h}$ is the quotient of the cyclic group $C_{2^{l-h}}$ by the free involution, and none of the conjugates of $\Z/2$ is contained in $D_{2^h}$. Thus 
\begin{align*}
F^l\sigma\tran^{D_{2^l}}_{D_{2^{h}}}&=\sum_{g\in \Z/2'/D_{2^l}/D_{2^h}}\tran^{\Z/2}_{e} c_g\res^{D_{2^{h}}}_{e}=0,
\end{align*}
again since $\tran^{\Z/2}_{e}=0$. Thus the second component of the bottom horizontal map is null as claimed.

The proof of the case $h=0$ is similar, by calculating the upper composite of the diagram
\[\xymatrix{
\pi_0\THR(k)^{\phi \Z/2}
\ar[r]^-{\tran^{D_{2^l}}_{\Z/2}}\ar[d]_-\cong
&
\pi_0\TRR^{l+1}(k;2)^{\phi \Z/2}\ar[d]^-{(F^l,F^l\sigma)}_-\cong
\\
k\otimes_S k
\ar[r]^-{}&
\phi^{l-1}((k\otimes_S k)^{C_2})\underset{k\otimes_S k}{\times}\phi^{l-1}((k\otimes_S k)^{C_2})
}\]
where the left vertical map is the isomorphism of \cite[Theorem 5.1]{THRmodels}.
\end{proof}

\begin{rem}
The notation used in Proposition \ref{prop:transfers} for the transfer map is consistent with our notation for the generators of Proposition \ref{prop:genTRR}, since
\[
V^{n-i}\tau_i(a\otimes b)=V^{n-i}(\phi^i(a\otimes b),\phi^i(b\otimes a))=(\phi^i(a\otimes b)+\phi^i(b\otimes a),0).
\]
\end{rem}

The generators $\tau_n(a\otimes b)$ also have a somewhat topological interpretation, as we now explain. As seen at the end of \S\ref{prelim}, the there is a non-additive norm map
\[
N_{\Z/2}^{D_{2^n}}\colon \pi_0\THR(k)^{\Z/2}=\pi_0\TRR^1(k;2)^{\Z/2}\longrightarrow \pi_0\TRR^{n+1}(k;2)^{\Z/2}=\pi_0\THR(k)^{D_{2^n}}.
\]
Moreover, since $k$ is of characteristic $2$, the canonical map $\pi_0\THR(k)^{\Z/2}\to \pi_0\THR(k)^{\phi\Z/2}$ is an isomorphism (since the transfer from the trivial subgroup to $\Z/2$ is zero by \cite[Theorem 5.1]{THRmodels}), and therefore by post-composing with the canonical projection we also obtain a non-additive map
\[
\pi_0\THR(k)^{\phi\Z/2}\cong \pi_0\THR(k)^{\Z/2}\xrightarrow{N_{\Z/2}^{D_{2^n}}}\pi_0\TRR^{n+1}(k;2)^{\Z/2}\longrightarrow \pi_0\TRR^{n+1}(k;2)^{\phi\Z/2}
\]
on geometric fixed points, which we still denote by $N_{\Z/2}^{D_{2^n}}$.

\begin{prop}\label{prop:norm}
Under the isomorphism of Theorem \ref{thm:geomTRn}, the map  $N_{\Z/2}^{D_{2^n}}$ corresponds to the map $N^n\colon k\otimes_S k
\to
\phi^{n-1}((k\otimes_S k)^{C_2})\times_{k\otimes_S k}\phi^{n-1}((k\otimes_S k)^{C_2})$ that sends an elementary tensor $a\otimes b$  to 
\[N^n(a\otimes b)=\tau_n(ab\otimes 1)=(\phi_n(ab\otimes 1),\phi_n(1\otimes ab)).\]
In particular, we find that
\[
\tau_n(a\otimes b)=N^n(a\otimes 1)\cdot \sigma N^n(b\otimes 1).
\]
\end{prop}

\begin{proof}
The identification of $N^n(a\otimes b)$ is similar to the proof of Proposition \ref{prop:transfers}.
It is sufficient to show that the unique map in the bottom row of the commutative square
\[\xymatrix{
\pi_0\TRR^{1}(k;2)^{\phi \Z/2}
\ar[r]^-{N^{D_{2^n}}_{\Z/2}}\ar[d]^-{}_-\cong
&
\pi_0\TRR^{n+1}(k;2)^{\phi \Z/2}\ar[d]^-{(F^n,F^n\sigma)}_-\cong
\\
k\otimes_S k
\ar[r]^-{}&
\phi^{n-1}((k\otimes_S k)^{C_2})\underset{k\otimes_S k}{\times}\phi^{n-1}((k\otimes_S k)^{C_2})
}\]
agrees with $N^n$ on the elementary tensors $a\otimes b$. Indeed, the value on a general element in the tensor product is determined by iterations of the relation
\[
N(x+y)=N(x)+N(y)+V(xy).
\]
By the multiplicative double coset formula of the $D_{2^l}$-Tambara functor $\underline{\pi}_0\THR(k)$, the upper composite has first component
\begin{align*}
F^nN^{D_{2^n}}_{\Z/2}&=\res^{D_{2^n}}_{\Z/2}N^{D_{2^n}}_{\Z/2}=\prod_{g\in \Z/2/D_{2^n}/\Z/2}N^{\Z/2}_{^g\Z/2\cap \Z/2} c_g\res^{\Z/2}_{\Z/2\cap \Z/2^g}.
\end{align*}
The double coset $\Z/2/D_{2^n}/\Z/2$ is the quotient of the cyclic group $C_{2^{n}}$ by the involution which acts by inversion, and consists of two fixed points (the unit and the rotation $g_0$ of order $2$ in $D_n$) which conjugate $\Z/2$ to itself, and $(2^{n}-2)/2$ points whose corresponding intersection $\Z/2\cap \Z/2^g$ is trivial. Moreover, since the cyclic group acts trivially on $\pi_0\THH(k)=k$, the conjugation $c_g$ is trivial except for $g=g_0$. Thus
\begin{align*}
F^nN^{D_{2^n}}_{\Z/2}&=(\id)\cdot (c_{g_0})\cdot(N^{\Z/2}_{e}\res^{\Z/2}_e)^{2^{n-1}-1}.
\end{align*}
Since the restriction map $\res^{\Z/2}_e$ corresponds to the multiplication map $\mu\colon k\otimes_S k\to k$ and $N^{\Z/2}_{e}$ to the map $k\to k\otimes_Sk$ which sends $a$ to $a^2\otimes 1$ (see \cite[Corollary 5.2]{THRmodels}), this sends $a\otimes b$ to
\[
a\otimes b\cdot b\otimes a\cdot ((ab)^2\otimes 1)^{2^{n-1}-1}=(ab)^{2^{n}-1}\otimes ab=\phi_n(1\otimes ab).
\]
Similarly, by letting $\Z/2'$ be the subgroup of $D_{n}$ generated by a reflection non-conjugate to $\Z/2$, the second component of the upper composite in the square above is
\begin{align*}
F^n\sigma N^{D_{2^n}}_{\Z/2}&=\res^{D_{2^n}}_{\Z/2'}N^{D_{2^n}}_{\Z/2}=\prod_{g\in \Z/2'/D_{2^l}/\Z/2}N^{\Z/2'}_{^g\Z/2\cap \Z/2} c_g\res^{\Z/2}_{\Z/2\cap \Z/2'^g}.
\end{align*}
Now the double coset $\Z/2'/D_{2^n}/\Z/2$ is the quotient of the cyclic group $C_{2^{n}}$ by the free involution, and since $\Z/2$ and $\Z/2'$ are not conjugate
\begin{align*}
F^n\sigma N^{D_{2^n}}_{\Z/2}&=(N^{\Z/2'}_{e}\res^{\Z/2}_{e})^{2^{n-1}}.
\end{align*}
Thus the second component of the bottom horizontal map of the square sends $a\otimes b$ to
\[
((ab)^2)^{2^{n-1}}\otimes 1=\phi_n(ab\otimes 1).
\]
This identifies the map $N^n$ as claimed. Finally, observe that
\begin{align*}
N^n(a\otimes 1)\cdot \sigma N^n(b\otimes 1)&=(a^{2^{n}}\otimes 1,a^{2^{n}-1}\otimes a)\cdot (b^{2^{n}-1}\otimes b,b^{2^{n}}\otimes 1)
\\&=(a^{2^{n}}b^{2^{n}-1}\otimes b,a^{2^{n}-1}b^{2^{n}}\otimes a)=\tau_n(a\otimes b).
\end{align*}
\end{proof}

\begin{rem}
For the usual Witt vectors, the elements $\tau_n(a)=(a,0,\dots,0)$ assemble into a (non-additive) multiplicative map $\tau_n\colon R\to W_{\langle 2^n\rangle}(R)$, which is a section for the truncation map $R$. We do not think that this is the case for $\pi_0\TRR^{n+1}(k;2)^{\phi \Z/2}$, since there seem to be no way of extending $\tau_n(a\otimes b)$ to a sum of elementary tensors. Even without a canonical splitting for the truncation map $R\colon \pi_0\TRR^{n+1}(k;2)^{\phi \Z/2}\to \pi_0\TRR^{1}(k;2)^{\phi \Z/2}$ at hand, having a set of generators for $\pi_0\TRR^{n+1}(k;2)^{\phi \Z/2}$ defined from the $\tau_i(a\otimes b)$ will suffice for our purposes.
\end{rem}

\subsection{The fundamental ideal of TRR}\label{sec:fundideal}

The components of the geometric fixed points of any connective $\Z/2$-spectrum $X$ admit a restriction map, defined as the canonical map of cokernels
\[
\xymatrix{
\pi_0(X_{h\Z/2})\ar[d]^{\cong} \ar[rr]^-{\tran_e^{\Z/2}}&&\pi_0(X^{\Z/2})\ar[d]^{\res^{\Z/2}_e}\ar[rr]&&\pi_0(X^{\phi \Z/2})\ar@{-->}[d]\ar[r]&0
\\
(\pi_0X)_{\Z/2} \ar[rr]^-{1+w}&&(\pi_0X)^{\Z/2}\ar[rr]&&(\pi_0X)^{\Z/2}/Im(1+w)\ar[r]&0
}
\]
where $w$ is the action of the generator of $\Z/2$ on $\pi_0X$. This map is moreover a monoidal natural transformation. By applying this construction to the $\Z/2$-spectrum $\TRR^{n+1}(k;2)$, we obtain a ring homomorphism which we denote by
\[
\res^{D_{2^n}}_{C_{2^n}}\colon \pi_0\TRR^{n+1}(k;2)^{\phi \Z/2}\longrightarrow(\pi_0\TR^{n+1}(k;2))^{\Z/2}/(1+w)\cong \W_{\langle 2^{n}\rangle}(k)/2.
\]
Here $\W_{\langle 2^{n}\rangle}(k)$ is the ring of $(n+1)$-truncated $2$-typical Witt vectors of $k$, and the isomorphism is from \cite[Theorem F]{WittVect}. Here we use that the isomorphism of \cite[Theorem F]{WittVect} is $\Z/2$-equivariant, where the $\Z/2$-action on $\W_{\langle 2^{n}\rangle}(k)$ is trivial (see the proof of \cite[Theorem 3.7]{polynomial}, where the first paragraph of page 522 holds also for $p=2$. This can more generally be applied to the case where $k$ has a non-trivial involution, in which case the involution on $\pi_0\TR^{n+1}(k;2)$ corresponds to the map induced on $\W_{\langle 2^{n}\rangle}(k)$ by the involution on $k$ under the functoriality of the Witt vectors).

The goal of this section is to describe explicitly the map $\res^{D_{2^n}}_{C_{2^n}}$ under the isomorphism of Theorem \ref{thm:geomTRn}, and provide generators for its kernel.
We recall that for every commutative ring $R$, as a set, $\W_{\langle 2^{n}\rangle}(R)=R^{\times n+1}$, with the unique functorial ring structure which makes the Witt polynomials into ring homomorphisms. For any $\F_2$-algebra $R$, we moreover have that as a set
\[
\W_{\langle 2^{n}\rangle}(R)=R\times (R/R^{2})^{\times n}.
\]
We denote by $V^{n-i}\tau_i(a)=(0,\dots,0,a,0,\dots,0)$ the canonical additive generators of $\W_{\langle 2^{n}\rangle}(R)$, for $a\in R$.
%
%

\begin{prop}\label{prop:res}
Let $k$ be a field of characteristic $2$, and $n\geq 1$.
Under the isomorphisms of Theorem \ref{thm:geomTRn} and \cite[Theorem F]{WittVect}, the restriction map corresponds to the unique ring homomorphism
\[
\res^{D_{2^n}}_{C_{2^n}}\colon  \phi^{n-1}((k\otimes_S k)^{C_2})\underset{k\otimes_S k}{\times}\phi^{n-1}((k\otimes_S k)^{C_2})
\longrightarrow
\W_{\langle 2^{n}\rangle}(k)/2
\]
which sends the respective generators of Proposition \ref{prop:genTRR} to
\[
\begin{array}{ll}
\res^{D_{2^n}}_{C_{2^n}}\tau_n(a\otimes b)&=(ab,0,\dots,0)=\tau_n(ab)
\\
\res^{D_{2^n}}_{C_{2^n}}V^{n-i}\tau_i(a\otimes b)&=(0,\dots,0,[ab],0,\dots,0)=[V^{n-i}\tau_{i}(ab)]
\\
\res^{D_{2^n}}_{C_{2^n}}\sigma V^{n-i}\tau_i(a\otimes b)&=(0,\dots,0,[ab],0,\dots,0)=[V^{n-i}\tau_{i}(ab)]
\end{array}
\]
for all $0\leq i\leq n-1$, where the mod $k^2$ reduction $[ab]$ of $ab$ sits in the $(n-i+1)$-st component.
\end{prop}

\begin{proof} Since $\underline{\pi_0}\THR(k)$ is a $D_{2^n}$-Tambara functor, the restriction $\res^{D_{2^n}}_{C_{2^n}}$ is a ring homomorphism, and by the double-coset formulas it commutes with norms and transfers, and with the Weyl action. The operators $V^{n-i}$ and $\tau_i$ are described in terms of norms and transfers by Propositions \ref{prop:transfers} and \ref{prop:norm}, and by \cite[Theorem 3.3]{WittVect} for the usual Witt vectors. It therefore follows that
\[
\res^{D_{2^n}}_{C_{2^n}}\sigma V^{n-i}\tau_i(a\otimes b)=[w V^{n-i}\tau_i\res^{\Z/2}_{e}(a\otimes b)]=[V^{n-i}\tau_{i}(ab)],
\]
where $\res^{\Z/2}_{e}$ is the multiplication map of $k\otimes_S k$ by \cite[Theorem 5.1]{THRmodels}. The proof for the other generators is similar.
\end{proof}

\begin{defn}
The fundamental ideal $J_{\langle 2^n\rangle}$ of $\pi_0\TRR^{n+1}(k;2)^{\phi\Z/2}$ is the kernel of the ring homomorphism
\[
J_{\langle 2^n\rangle}:=\ker\big(\res^{D_{2^n}}_{C_{2^n}}\colon  \phi^{n-1}((k\otimes_S k)^{C_2})\underset{k\otimes_S k}{\times}\phi^{n-1}((k\otimes_S k)^{C_2})
\longrightarrow
\W_{\langle 2^{n}\rangle}(k)/2\big)
\]
from Proposition \ref{prop:res}, for $n\geq 1$, and for $n=0$ it is the kernel of the multiplication map
\[
J_{\langle 1\rangle}:=\ker\big(\res^{\Z/2}_{e}\colon  k\otimes_S k
\longrightarrow k=
\W_{\langle 1\rangle}(k)/2\big).
\]
\end{defn}

\begin{prop}\label{prop:additivegen}
For every $n\geq 0$, $J_{\langle 2^n\rangle}$ is the subgroup of  $\pi_0\TRR^{n+1}(k;2)^{\phi\Z/2}$ generated by the elements
\begin{align*}
&\tau_n(a\otimes b)+\tau_n(ab\otimes 1),
\\
 &V^{n-i}\tau_i(a\otimes b)+\sigma V^{n-i}\tau_i(a\otimes b),
 \\
 &V^{n-i}\tau_i(a\otimes b)+V^{n-i}\tau_i(ab\otimes 1),
\end{align*}
for all $0\leq i\leq n-1$ and $a\otimes b\in k\otimes_S k$.
\end{prop}

\begin{proof}
The proof is by induction on $n$. For $n=0$, this is the claim that the kernel of the multiplication map
\[
\mu\colon k\otimes_Sk\longrightarrow k
\]
is generated by $a\otimes b+ab\otimes 1$ for $a\otimes b\in k\otimes_S k$, which is clear.

Now suppose the claim holds for $n$, and consider the commutative diagram with exact rows
\[
\xymatrix@C=15pt{
&K\ar[rrr]^-{V^{n+1}+\sigma V^{n+1}}\ar[d]
&&&J_{\langle 2^{n+1}\rangle}\ar[r]^-{R}\ar[d]
&
J_{\langle 2^{n}\rangle}\ar[d]
\\
&(k\otimes_Sk)\oplus (k\otimes_Sk)\ar[rrr]^-{V^{n+1}+\sigma V^{n+1}}\ar[d]^{\mu+\mu}
&&&\pi_0\TRR^{n+2}(k;2)^{\phi\Z/2}\ar[r]^-{R}\ar[d]^{\res^{D_{2^{n+1}}}_{C_{2^{n+1}}}}
&
\pi_0\TRR^{n+1}(k;2)^{\phi\Z/2}\ar[r]\ar[d]^{\res^{D_{2^{n}}}_{C_{2^{n}}}}
&0
\\
0\ar[r]&
k/k^2\ar[rrr]^-{V^{n+1}}
&&&\W_{\langle 2^{n+1}\rangle}(k)/2\ar[r]^-{R}
&
\W_{\langle 2^{n}\rangle}(k)/2\ar[r]
&0
}
\]
where the vertical maps from the top row to the middle row are kernel inclusions. 
The middle row is exact by \cite[Proof of Theorem 4.9]{geomTCR}, or by the explicit calculation of Theorem \ref{thm:geomTRn} and Proposition \ref{prop:transfers}.
Thus, if we find a set of generators for $K$ and show that the map $R$ in the first row is surjective, then $J_{\langle 2^{n+1}\rangle}$ is generated by the image by $V^{n+1}+\sigma V^{n+1}$ of the generators of $K$, and by a choice of lifts of the generators of  $J_{\langle 2^{n}\rangle}$ given by the inductive assumption. The kernel $K$ consists of those elements $(x,y)$ such that $\mu(x)+\mu(y)$ is a square in $k$. Since every square $c^2$ in $k$ is hit by $c\otimes c$ under $\mu$, $K$ is the subgroup of elements of the form $(x,y)$ where 
\[
x=y+c\otimes c+z
\]
for some $c\in k$ and $z\in \ker (\mu)$. Since the kernel of $\mu$ is generated by elements of the form $a\otimes b+ab\otimes 1$, we conclude that $K$ is generated by elements of the form $(a\otimes b, a\otimes b)$, and elements of the form $(a\otimes b+ab\otimes 1+c\otimes c,0)$. The images of these generators by  $V^{n+1}+\sigma V^{n+1}$ are respectively of the form
\[
V^{n+1}(a\otimes b)+\sigma V^{n+1}(a\otimes b)
\]
and
\[
V^{n+1}(a\otimes b+ab\otimes 1+c\otimes c)+\sigma V^{n+1}(0)=V^{n+1}(a\otimes b)+V^{n+1}(ab\otimes 1)+V^{n+1}(c\otimes c),
\]
where $V^{n+1}(c\otimes c)=0$ by Proposition \ref{prop:transfers}. It therefore remains to show that by applying $R$ to the elements $\tau_{n+1}(a\otimes b)+\tau_{n+1}(ab\otimes 1)$
, $V^{n+1-i}\tau_i(a\otimes b)+\sigma V^{n+1-i}\tau_i(a\otimes b)$ and $V^{n+1-i}\tau_i(a\otimes b)+V^{n+1-i}\tau_i(ab\otimes 1)$, for $1\leq i\leq n$, we hit all the generators of $J_{\langle 2^{n}\rangle}$ given by the inductive assumption. This is the case since  $R\tau_{i+1}=\tau_i$, $RV^{n+1-i}=V^{n-i}R$, and $R\sigma=\sigma R$, by Theorem \ref{thm:geomTRn} and Proposition \ref{prop:transfers}.
\end{proof}

\begin{cor}\label{cor:idealgen}
For every $n\geq 0$,
the  ideal $J_{\langle 2^n\rangle}$ is generated, as a $\pi_0\TRR^{n+1}(k;2)^{\phi\Z/2}$-module, by the elements of the form
\begin{align*}
&\tau_n(1\otimes c)+\tau_n(c\otimes 1)=V^{0}\tau_n(c\otimes 1)+\sigma V^{0}\tau_n(c\otimes 1),
\\
 &V^{n-i}\tau_i(a\otimes b)+\sigma V^{n-i}\tau_i(a\otimes b),
\end{align*}
for all $0\leq i\leq n-1$, $a\otimes b\in k\otimes_S k$ and $c\in k$. In particular, $J_{\langle 2^n\rangle}$ is generated, as a $\pi_0\TRR^{n+1}(k;2)^{\phi\Z/2}$-module, by fixed points for the involution $\sigma$.
\end{cor}

\begin{proof}
By Proposition \ref{prop:additivegen}, the corollary follows from the identities
\begin{align*}
\tau_n(a\otimes b)+\tau_n(ab\otimes 1)&=\tau_{n}(a\otimes 1)\cdot(\tau_n(1\otimes b)+\tau_n(b\otimes 1)),
\\
V^{n-i}\tau_i(a\otimes b)+V^{n-i}\tau_i(ab\otimes 1)&=(V^{n-i}\tau_i(a\otimes 1)+\sigma V^{n-i}\tau_i(b\otimes 1))\cdot (V^{n-i}\tau_i(b\otimes 1)+\sigma V^{n-i}\tau_i(b\otimes 1)),
\end{align*}
which can be easily verified from the definitions.
\end{proof}

\section{Real TR and the de Rham-Witt complex}

\subsection{The Witt complex associated to TRR}

Since the operators $F,V,\sigma$ and $R$ of Theorem \ref{thm:geomTRn} and Proposition \ref{prop:transfers} commute with the restriction map to the Witt vectors modulo $2$, they induce maps on the fundamental ideals
\[
F,R\colon J_{\langle 2^{n+1}\rangle}\to J_{\langle 2^n\rangle} \ \ \ \ \ \ , \ \ \ \ \ \ V\colon J_{\langle 2^{n}\rangle}\to J_{\langle 2^{n+1}\rangle} \ \ \ \ \ \ and  \ \ \ \ \ \ \  \sigma\colon J_{\langle 2^{n}\rangle}\to J_{\langle 2^{n}\rangle}
\]
for all $n\geq 0$.

\begin{prop}\label{prop:operatorspowers}
For every integer $q\geq 2$, the maps $F,R,V,\sigma$ above restrict to maps
\[
F,R\colon J_{\langle 2^{n+1}\rangle}^q\to J_{\langle 2^n\rangle}^q \ \ \ \  \ \ , \ \ \ \ \ \ V\colon J_{\langle 2^{n}\rangle}^q\to J_{\langle 2^{n+1}\rangle}^q \ \ \ \ \ \ and  \ \ \ \ \ \ \  \sigma\colon J_{\langle 2^{n}\rangle}^q\to J_{\langle 2^{n}\rangle}^q.
\]
Moreover, $1+\sigma$ induces a well-defined map
\[
1+\sigma\colon J_{\langle 2^{n}\rangle}^q\to J_{\langle 2^{n}\rangle}^{q+1},
\]
 which satisfies $(1+\sigma)^2=0$.
\end{prop}

\begin{proof}
The claim about $R,F$ and $\sigma$ are clear since these maps are multiplicative. For the map $V$, we employ Corollary \ref{cor:idealgen}. First suppose that $n\geq 1$, so that the power $J_{\langle 2^{n}\rangle}^q$ is additively generated by elements of the form
\[(x,y)\cdot (x_1,x_1)\cdot\dots\cdot (x_q,x_q)=(xx_1\dots x_q,yx_1\dots x_q),\]
where $(x,y)$ is a generator of $\pi_0\TRR^{n+1}(k;2)^{\phi\Z/2}$ from Proposition \ref{prop:genTRR}, and $(x_l,x_l)$ is a generator of $J_{\langle 2^{n}\rangle}$ from Corollary \ref{cor:idealgen}, which is diagonal since they are invariant by the Weyl action $\sigma$. Since $V$ is additive, it is sufficient to show that $V$ sends these elements to $J_{\langle 2^{n+1}\rangle}^q $. Now by Proposition \ref{prop:transfers}, we have that 
\begin{align*}
V(xx_1\dots x_q,yx_1\dots x_q)&=((x+y)x_1\dots x_q,0)
\\&=(x+y,x+y)\cdot (x_1,x_1)\cdot\dots\cdot (x_{q-2},x_{q-2})\cdot (x_{q-1},x_{q-1})\cdot(x_q,0).
\end{align*}
The first factor is
\[
(x+y,x+y)=V(x,y)+\sigma V(x,y), \ \ \ \  \ 
\]
which belongs to $J_{\langle 2^{n+1}\rangle}$ since it is sent to zero by the restriction map by Proposition \ref{prop:res}. By Corollary \ref{cor:idealgen}, each of the factors $(x_1,x_1), \dots,  (x_{q-1},x_{q-1})$ is of the form
\[
 V^{n-i}\tau_i(a\otimes b)+\sigma V^{n-i}\tau_i(a\otimes b)
\]
for some $0\leq i\leq n$, which as an element of $\pi_0\TRR^{n+2}(k;2)^{\phi\Z/2}$ is of the form
\[
V^{n+1-i}\tau_i(a\otimes b)+\sigma V^{n+1-i}\tau_i(a\otimes b)
\]
for $0\leq i\leq n$, and therefore belongs to $J_{\langle 2^{n+1}\rangle}$. Thus, it suffices to show that $(x_{q-1},x_{q-1})\cdot(x_q,0)$ is also in $J_{\langle 2^{n+1}\rangle}$. But since $(x_q,x_q)$ is of the form $V^{n-i}\tau_i(a\otimes b)+\sigma V^{n-i}\tau_i(a\otimes b)$, we have that $(x_q,0)$ is a well-defined element of  $\pi_0\TRR^{n+2}(k;2)^{\phi\Z/2}$, and since $J_{\langle 2^{n+1}\rangle}$ is an ideal, $ (x_{q-1},x_{q-1})\cdot(x_q,0)$ indeed belongs to $J_{\langle 2^{n+1}\rangle}$.

If $n=0$, the ideal $J_{\langle 1\rangle}^q$ of $k\otimes_Sk$ is additively generated by elements of the form $xx_1\dots x_q$ with $x\in k\otimes_S k$ and $x_1,\dots,x_q$ fixed by the involution $w$. Then by Proposition \ref{prop:transfers}
\[
V(xx_1\dots x_q)=(xx_1\dots x_q+w(xx_1\dots x_q),0)=((x+w(x))x_1\dots x_q,0),
\]
and one can repeat the argument used in the case $n\geq 1$.

Finally, let us show that $1+\sigma$ sends $J_{\langle 2^{n}\rangle}^q$ to $J_{\langle 2^{n}\rangle}^{q+1}$. By Corollary \ref{cor:idealgen}, every element of $J_{\langle 2^{n}\rangle}^q $ is a sum of elements of the form $z\cdot g_1\cdot\dots\cdot g_q$ with $z\in \pi_0\TRR^{n+1}(k;2)^{\phi\Z/2}$ and each $g_i\in J_{\langle 2^{n}\rangle}$ fixed by $\sigma$. Since $1+\sigma$ is additive we only need to show that these elements are sent to $J_{\langle 2^{n}\rangle}^{q+1}$. Since $\sigma$ is multiplicative, we have that
\[
(1+\sigma)(z\cdot g_1\cdot\dots\cdot g_q)=z\cdot g_1\cdot\dots\cdot g_q+\sigma(z)\cdot \sigma(g_1)\cdot\dots\cdot \sigma(g_q)=(z+\sigma(z))\cdot g_1\cdot\dots\cdot g_q.
\]
It therefore suffices to show that $z+\sigma(z)$ belongs to $J_{\langle 2^{n}\rangle}$, which is the case since the restriction map to the Witt vectors modulo $2$ is invariant under the action of $\sigma$. Clearly, since $\sigma^2$ is the identity, we have that $(1+\sigma)^2=0$.
\end{proof}

For every $n\geq 0$, let us denote by $J_{\langle 2^{n}\rangle}^\ast/J_{\langle 2^{n}\rangle}^{\ast+1}$ the graded ring defined by the quotients $J_{\langle 2^{n}\rangle}^q/J_{\langle 2^{n}\rangle}^{q+1}$ for $q\geq 0$, and by the multiplication of $J_{\langle 2^{n}\rangle}$. We will show that the sequence of graded rings $J_{\langle 2^{n}\rangle}^\ast/J_{\langle 2^{n}\rangle}^{\ast+1}$ where $n\geq 0$, equipped with the operators $R,F,V$ and $d:=(1+\sigma)$, define the structure of a $2$-typical Witt complex. We recall its definition, from \cite{Costeanu}, in the special case where the base ring has characteristic $2$. In this case item v) simplifies since $d\log[-1]=0$, and the definition agrees to the one for odd primes from \cite{IbLarsDeRhamMixed}.

\begin{defn}[\cite{Costeanu}]\label{def:WittComplex}
A $2$-typical Witt complex over an $\F_2$-algebra $A$ consists
of:
\begin{enumerate}
\item a graded-commutative pro-graded ring $\{E^{\ast}_n, R\colon E^{\ast}_{n+1}\to E^{\ast}_n\}_{n\geq 0}$, 
\item A strict map of pro-rings $\lambda\colon \W_{\langle 2^{\bullet}\rangle}(A)\to E^{0}_\bullet$ from the pro-ring of $2$-typical
Witt vectors of $A$,
\item a strict map of pro-graded rings 
\[F\colon E^{\ast}_{\bullet+1}\longrightarrow E^{\ast}_\bullet\]
 such that $\lambda F=F\lambda$,
\item a strict map of pro-graded $E^{\ast}_\bullet$-modules 
\[V\colon F^{\ast}E^{\ast}_\bullet\longrightarrow E^{\ast}_{\bullet+1}\]
 such that $\lambda V=V\lambda$ and $FV=2$.
The linearity of $V$ means that
$V (x)y = V (xF (y))$ for all $x\in E^{\ast}_n$ and $y\in E^{\ast}_{n+1}$,
\item a strict map of pro-graded abelian groups $d\colon E^{\ast}_\bullet\to E^{\ast+1}_\bullet$, which is a
derivation, in the sense that
\[d(xy)=d(x)y+(-1)^{|x|}xd(y)\]
for all $x,y\in E^{\ast}_n$, and which satisfies the relations
\begin{align*}
FdV&=d
\\
dd&=0
\\
Fd\lambda\tau_n&=(\lambda\tau_{n-1})\cdot (d\lambda\tau_{n-1}),
\end{align*}
where $\tau_n\colon A\to \W_{\langle 2^{n}\rangle}(A)$ is the Teichm\"{u}ller map sending $a$ to $(a,0,\dots,0)$.
\end{enumerate}
\end{defn}

Before showing that the graded ring defined by the ideals $J_{\langle 2^{n}\rangle}$ admits the structure of a Witt-complex, let us point out that since the map $\res^{D_{2^n}}_{C_{2^n}}$ is surjective by Propositions \ref{prop:res}, it induces an isomorphism
\[
J_{\langle 2^{n}\rangle}^0/J_{\langle 2^{n}\rangle}=\pi_0\TRR^{n+1}(k;2)^{\phi\Z/2}/\ker(\res^{D_{2^n}}_{C_{2^n}})\stackrel{\cong}{\longrightarrow} \pi_0\TR^{n+1}(k;2)/2\cong \W_{\langle 2^n\rangle}(k)/2,
\]
where the last isomorphism is from \cite[Theorem F]{WittVect}.

\begin{prop}\label{prop:WittComplex}
The sequence of graded rings $\{J_{\langle 2^{n}\rangle}^\ast/J_{\langle 2^{n}\rangle}^{\ast+1}\}_{n\geq 0}$ equipped with the operators $R,F,V$ and $ d:=(1+\sigma)$ from Proposition \ref{prop:operatorspowers}, and the quotient maps
\[
\lambda\colon \W_{\langle 2^n\rangle}(k)\longrightarrow \W_{\langle 2^n\rangle}(k)/2\cong  J_{\langle 2^{n}\rangle}^0/J_{\langle 2^{n}\rangle},
\]
defines a $2$-typical Witt complex over the field $k$ of characteristic $2$.
\end{prop}

\begin{proof}
First of all, the maps $R,F,V$ and $ d:=(1+\sigma)$ are well-defined on the quotients of the powers of the ideals by Proposition \ref{prop:operatorspowers}.
Axioms i)-iv) of Definition \ref{def:WittComplex} follow immediately from either the fact that $F,V$ and $\res^{D_{2^n}}_{C_{2^n}}$ are induced from the maps of a Mackey functor, or from their explicit formulas from Theorem \ref{thm:geomTRn} and Propositions \ref{prop:transfers} and \ref{prop:res}. This is except from the identity $FV=2$ (which in our case is zero), since by these arguments we only know that  $FV=1+\sigma$. However, for every $x\in J_{\langle 2^{n}\rangle}^q$, we have that
\[
FV(x)=x+\sigma(x)
\]
belongs to $J_{\langle 2^{n}\rangle}^{q+1}$ by Proposition \ref{prop:operatorspowers}, and it is therefore indeed zero in $J_{\langle 2^{n}\rangle}^q/J_{\langle 2^{n}\rangle}^{q+1}$.

Let us show axiom v).
To see that $d$ satisfies the Leibniz rule, let $x\in J_{\langle 2^{n}\rangle}^q/J_{\langle 2^{n}\rangle}^{q+1}$ and $y\in J_{\langle 2^{n}\rangle}^{q'}/J_{\langle 2^{n}\rangle}^{q'+1}$, and let us calculate
\begin{align*}
d(xy)+d(x) y+xd(y)&=xy+\sigma(x)\sigma(y)+(x+\sigma(x)) y+x(y+\sigma(y))\\
&=xy+\sigma(x)\sigma(y)+\sigma(x)y+x\sigma(y)=(x+\sigma(x)) (y+\sigma(y))\\
&=d(x)d(y).
\end{align*}
Since $d(x)$ belongs to $J_{\langle 2^{n}\rangle}^{q+1}$ and $d(y)$ to $J_{\langle 2^{n}\rangle}^{q'+1}$ by Proposition \ref{prop:operatorspowers}, we have that $d(x)d(y)$ belongs to $J_{\langle 2^{n}\rangle}^{q+q'+2}$, and therefore it vanishes in $J_{\langle 2^{n}\rangle}^{q+q'+1}/J_{\langle 2^{n}\rangle}^{q+q'+2}$.

Let us now verify the last three identities involving $d$ in axiom v). For the first one, let $x\in J_{\langle 2^{n}\rangle}^q/J_{\langle 2^{n}\rangle}^{q+1}$. Then
\[FdV(x)=FV(x)+F\sigma V(x)=FV(x)=(1+\sigma)(x)=d(x)\]
in $J_{\langle 2^{n}\rangle}^{q+1}/J_{\langle 2^{n}\rangle}^{q+2}$, where $F\sigma V(x)=0$ by the double coset formula (or by direct calculation). For the second identity, we have that
\[
d^2=(1+\sigma)^2=1+2\sigma+\sigma^2=2+2\sigma =0
\]
since $\sigma$ has order $2$. Finally, for the third one, let $a\in k=\W_{\langle 1\rangle}(k)$. On the one hand, by Proposition \ref{prop:res},
\begin{align*}
Fd\lambda\tau_n(a)&=Fd\tau_n(a\otimes 1)=F(\tau_n(a\otimes 1)+\sigma \tau_n(a\otimes 1))
\\&=\tau_n(a\otimes 1)+\sigma \tau_n(a\otimes 1)=(a^{2^{n}}\otimes 1+a^{2^{n}-1}\otimes a,a^{2^{n}-1}\otimes a+a^{2^{n}}\otimes 1)
\end{align*}
where the third equality holds by the formula for $F$ of Theorem \ref{thm:geomTRn}.
On the other hand
\begin{align*}
&(\lambda\tau_{n-1}(a))\cdot (d\lambda\tau_{n-1}(a))=\tau_{n-1}(a\otimes 1)\cdot (\tau_{n-1}(a\otimes 1)+\sigma\tau_{n-1}(a\otimes 1))
\\&=(a^{2^{n-1}}\otimes 1,a^{2^{n-1}-1}\otimes a)\cdot (a^{2^{n-1}}\otimes 1+a^{2^{n-1}-1}\otimes a,a^{2^{n-1}-1}\otimes a+a^{2^{n-1}}\otimes 1)
\\&=(a^{2^{n}}\otimes 1+a^{2^{n}-1}\otimes a,a^{2^{n}-2}\otimes a^2+a^{2^{n}-1}\otimes a),
\end{align*}
and these are equal since we are tensoring over $S$.
\end{proof}

\subsection{The Milnor conjecture for the de Rham-Witt complex}\label{sec:dRW}

Let us endow the sequence $J_{\langle 2^\bullet \rangle}^\ast/J_{\langle 2^\bullet \rangle}^{\ast+1}$ with the structure of a $2$-typical Witt complex of Proposition \ref{prop:WittComplex}.  We recall that, by definition, the $2$-typical de Rham-Witt complex $\W_{\langle 2^\bullet \rangle}\Omega^\ast_k$ of $k$ is the initial object in the category of  $2$-typical Witt complexes over $k$ (see \cite{Costeanu} and \cite{IbLarsDeRhamMixed}). Thus, there is a unique map of $2$-typical Witt complexes
\[
\W_{\langle 2^\bullet \rangle}\Omega^\ast_k\longrightarrow J_{\langle 2^\bullet \rangle}^\ast/J_{\langle 2^\bullet \rangle}^{\ast+1}.
\]
Let us denote by $\W_{\langle 2^\bullet \rangle}\Omega^\ast_k/2$ the degreewise cokernel of the multiplication by $2$ map. Since all the maps defining the structure of a Witt complex are additive, this is again a Witt-complex, where the map
\[
\W_{\langle 2^\bullet \rangle}(k)\longrightarrow \W_{\langle 2^\bullet \rangle}\Omega^0_k/2=\W_{\langle 2^\bullet \rangle}(k)/2
\]
is the quotient map. Since $2$ vanishes in $J_{\langle 2^\bullet \rangle}^\ast/J_{\langle 2^\bullet \rangle}^{\ast+1}$, the unique map above descends to a unique map of Witt-complexes 
\[
u\colon \W_{\langle 2^\bullet \rangle}\Omega^\ast_k/2\longrightarrow J_{\langle 2^\bullet \rangle}^\ast/J_{\langle 2^\bullet \rangle}^{\ast+1}.
\]

\begin{theorem}\label{thm:dRW}
The unique map of Witt-complexes $u\colon \W_{\langle 2^\bullet \rangle}\Omega^\ast_k/2\to J_{\langle 2^\bullet \rangle}^\ast/J_{\langle 2^\bullet \rangle}^{\ast+1}$
is an isomorphism.
\end{theorem}

\begin{rem}\label{rem:lowcases}
Let us discuss a few special cases of this theorem. For $\ast=0$, the unique map $u$ is by construction the isomorphism
\[
\lambda\colon \W_{\langle 2^\bullet \rangle}\Omega^0_k/2=\W_{\langle 2^\bullet \rangle}(k)/2\stackrel{\cong}{\longrightarrow}\pi_0\TR(k;2)/2\cong \pi_0\TRR^{\bullet+1}(k;2)^{\phi\Z/2}/J_{\langle 2^\bullet \rangle}
\]
where the arrow is the isomorphism of \cite[Theorem F]{WittVect}.

On the other hand, for $\bullet=0$, the map $u$ is the unique map of commutative differential graded algebras
\[
\W_{\langle 1 \rangle}\Omega^\ast_k/2=\Omega^\ast_k\longrightarrow J^\ast_{\langle 1\rangle}/J^{\ast+1}_{\langle 1\rangle}.
\]
This is well-known to be an isomorphism, as claimed in \cite{Kato} (see e.g. \cite{Arason} for a proof, which is also recasted in Lemma \ref{lemma:trunc0} below). In particular, for $\ast=1$, this is equivalent to the fact that, since $k$ has characteristic $2$, a $\Z$-linear derivation out of $k$ is automatically $S$-linear (where we recall that $S\leq k$ is the subfield of squares).
\end{rem}

The rest of the section is dedicated to the proof of Theorem \ref{thm:dRW}.
The proof is by induction on $n$, by means of the exact sequences 
\[
\xymatrix@C=40pt{
\Omega^q_k\oplus\Omega^{q-1}_k\ar[r]^-{V^n+dV^n}
&
\W_{\langle 2^{n}\rangle}\Omega^q_k/2\ar[r]^-{R}
&
\W_{\langle 2^{n-1}\rangle}\Omega^q_k/2\ar[r]
&
0
}
\]
from \cite[Lemma 3.5]{Costeanu}, where $n,q\geq 1$.

The base case for the induction $n=1$ seems to be well-known to the experts, and is used without proof in \cite{Kato}. We recall the argument from \cite[Fact 1]{Arason} for the reader's convenience, and to introduce some notation that we will use in the proof of the induction step.

\begin{lemma}[\cite{Arason}]\label{lemma:trunc0}
The unique map $u\colon \Omega^\ast_k\longrightarrow J_{\langle 1\rangle}^\ast/J_{\langle 1\rangle}^{\ast+1}$  of commutative differential graded algebras is an isomorphism. 
\end{lemma}

\begin{proof}
For every $a\in k$, let us denote $\Delta(a):=1\otimes a+a\otimes 1\in k\otimes_S k$.
We note that the map $u$ is necessarily given by the formula
\[u(ada_1\dots da_q):=a\Delta(a_1)\cdot \dots\cdot \Delta(a_q).\]

In order to show that this is an isomorphism, we choose suitable bases of the source and target as $k$-vector spaces.
Let $\{x_{i}\}_{i\in I}$ be a $2$-basis of $k$. We recall that this is a set of elements of $k$ whose differentials $\{dx_i\}_{i\in I}$ form a basis of the $k$-vector space $\Omega^1_k$ or, equivalently, such that the elements
\[
x^\xi:=\prod_{i\in \xi}x_i
\]
form a basis of $k$ as an $S$-vector space, where $\xi$ ranges through the finite subsets of $I$ (see \cite[Chapter 0, \S21.4]{EGAIV}). Here we use the convention that $x^\emptyset=1$, and we will write $\xi\subset^f \! I$ if $\xi$ is a finite subset of $I$.
It is easy to see that the set $\{1\otimes x^\xi\}_{\xi\subset^f I}$ is a basis of $k\otimes_S k$ as a $k$-vector space, where $k$ acts by multiplication on the left tensor factor. Now let us denote
\[
\Delta(x)^{\xi}:=\prod_{i\in \xi}\Delta(x_i)
\]
for every $\xi\subset^f\! I$ (with the convention that $\Delta(x)^{\emptyset}=1\otimes 1$). These elements satisfy the identities
\begin{align}
\Delta(x)^{\xi}&=\sum_{\nu\subset\xi}x^{\xi\setminus\nu}\cdot(1\otimes x^\nu)\label{form:changebasis1}
\\
1\otimes x^{\xi}&=\sum_{\nu\subset\xi}x^{\xi\setminus\nu}\cdot\Delta(x)^{\nu}\label{form:changebasis2}
\end{align}
for every finite subset $\xi$ of $I$. It follows that $\{\Delta(x)^{\xi}\}_{\xi\subset^f I}$ is also a basis of $k\otimes_S k$. Since the multiplication map sends $\Delta(x)^{\xi}$ to a non-zero element of $k$ if and only if $\xi=\emptyset$, it follows that $\{\Delta(x)^{\xi}\}_{\emptyset\neq \xi\subset^f I}$ is a basis for $J_{\langle 1\rangle}$ as a $k$-vector space with respect to multiplication on the left factor.
It then readily follows that the elements $\Delta(x)^{\xi}$ with $|\xi|\geq q$ form a basis of $J_{\langle 1\rangle}^q$, and that the elements $\Delta(x)^{\xi}$ with $|\xi|= q$ form a basis of $J_{\langle 1\rangle}^q/J_{\langle 1\rangle}^{q+1}$. Since for every $\xi\subset I$ with $|\xi|=q$, the map $u$ sends a basis element $(dx)^\xi:=\prod_{i\in \xi}dx_i$ of $\Omega^q_k$ to $\Delta(x)^{\xi}$, the claim follows.
\end{proof}

The induction step for proving Theorem \ref{thm:dRW} will rely on the following two key technical Lemmas. Let us choose a $2$-basis $\{x_{i}\}_{i\in I}$ of $k$ as in the proof of Lemma \ref{lemma:trunc0}.

\begin{lemma}\label{lemma:key1}
Let $q\geq 1$, and let $\mu_\nu\in k$ for every subset $\nu\subset I$ with $|\nu|=q-1$. Suppose that 
\[\sum_{
\begin{smallmatrix}
\nu\subset I
\\
|\nu|=q-1
\end{smallmatrix}
}\Delta(\mu_\nu)\Delta(x)^{\nu} \in J_{\langle 1 \rangle}^{q+1}.\]
 Then $\sum_{\nu\subset I, |\nu|=q-1}V(\mu_\nu (dx)^\nu)$ is divisible by $2$ in $\W_{\langle 2\rangle}\Omega^{q-1}_k$.
\end{lemma}

\begin{proof}
Let us write $\mu_\nu\in k$ uniquely as a linear combination $\mu_\nu=\sum_{\delta\subset^f I}s_{\nu,\delta}^2x^\delta$ with respect to the basis $\{x^\delta\}_{\delta\subset^f I}$ of $k$ as an $S$-vector space. Let us notice that, since we are tensoring over $S$, the map $\Delta\colon k\to k\otimes_Sk$ is $S$-linear.
Then by applying formula (\ref{form:changebasis2}):
\begin{align*}
\Delta(\mu_\nu)&=\sum_{\delta\subset^f I}s_{\nu,\delta}^2\Delta(x^\delta)=\sum_{\delta\subset^f I}s_{\nu,\delta}^2(1\otimes x^\delta+x^{\delta}\cdot (1\otimes 1))
\\
&=\sum_{\delta\subset^f I}s_{\nu,\delta}^2\sum_{\gamma\subset\delta}x^{\delta\setminus \gamma}\Delta(x)^\gamma+\sum_{\delta\subset^f I}s_{\nu,\delta}^2x^{\delta}\cdot \Delta(x)^{\emptyset}
=\sum_{\emptyset\neq \gamma\subset\delta\subset^f I}s_{\nu,\delta}^2x^{\delta\setminus \gamma}\Delta(x)^\gamma,
\end{align*}
where the last equality holds since the sum $\sum_{\delta\subset^f I}s_{\nu,\delta}^2x^{\delta}\cdot \Delta(x)^{\emptyset}$ is equal to the term $\gamma=\emptyset$ in the previous sum.
It follows that
\begin{align*}
\sum_{
\begin{smallmatrix}
\nu\subset I
\\
|\nu|=q-1
\end{smallmatrix}
}\Delta(\mu_\nu)\Delta(x)^{\nu}
&=\sum_{
\begin{smallmatrix}
\nu\subset I
\\
|\nu|=q-1
\end{smallmatrix}
}
(\sum_{\emptyset\neq \gamma\subset\delta^f\subset I}s_{\nu,\delta}^2x^{\delta\setminus \gamma}\Delta(x)^\gamma)\Delta(x)^{\nu}
=\sum_{
\begin{smallmatrix}
\nu\subset I
\\
|\nu|=q-1
\end{smallmatrix}
}
\sum_{
\begin{smallmatrix}
\emptyset\neq \gamma\subset\delta\subset^f I
\\
\gamma\cap\nu=\emptyset
\end{smallmatrix}
}s_{\nu,\delta}^2x^{\delta\setminus \gamma}\Delta(x)^{\gamma\amalg\nu},
\end{align*}
where in the last sum $\gamma$ and $\nu$ are disjoint, since $\Delta(a)\Delta(a)=0$ in $k\otimes_Sk$ for all $a\in k$.
Since, by assumption, this element belongs to $J_{\langle 1 \rangle}^{q+1}$, and the $\Delta(x)^\xi$ with $|\xi|=q+1$ are a basis for $ J_{\langle 1 \rangle}^{q+1}$, we must have that, for every $\nu\subset I$ with $|\nu|=q-1$ and every $j\in I\setminus \nu$ (corresponding to the terms $\gamma=\{j\}$ in the sum above),  
\[
\sum_{
\begin{smallmatrix}
\nu\subset I
\\
|\nu|=q-1
\end{smallmatrix}
}
\sum_{\delta\subset^f I}\sum_{
j\in \delta\setminus\nu
}
s_{\nu,\delta}^2x^{\delta\setminus j}\Delta(x)^{j\amalg\nu}
=0.
\]
Since the $\Delta(x)^{\xi}$ are linearly independent, we have that, for every subset $\xi\subset I$ with $|\xi|=q$ (corresponding to the terms $\epsilon=j\amalg\nu$), the coefficient of $\Delta(x)^{\xi}$ must vanish, i.e. that
\[
0=\sum_{j\in \xi}
\sum_{
\begin{smallmatrix}
\delta\subset^f I
\\
j\in \delta
\end{smallmatrix}
}
s_{\xi\setminus j,\delta}^2x^{\delta\setminus j}
=
\sum_{\alpha\subset^f I}
\sum_{
j\in \xi\setminus \alpha
}
s_{\xi\setminus j,\alpha\amalg j}^2x^{\alpha}.
\]
Since the $x^{\alpha}$ form a basis of $k$ as an $S$-vector space, we find that for every finite $\alpha\subset I$ and $\xi\subset I$ with $|\xi|=q$ the corresponding coefficient must vanish:
\begin{equation}\label{snudelta}
\sum_{
j\in \xi\setminus \alpha
}
s_{\xi\setminus j,\alpha\amalg j}^2=0.
\end{equation}
We now show that these relations among the coefficients $s_{\nu,\delta}$ imply that $V(\sum_{\nu\subset I, |\nu|=q-1}\mu_\nu (dx)^\nu)$ is divisible by $2$ in $\W_{\langle 2\rangle}\Omega^{q-1}_k$. We do this by showing that the sum $\sum_{\nu\subset I, |\nu|=q-1}\mu_\nu (dx)^\nu$ is in the image of the Frobenius map. The claim will then follow since, by the linearity of $V$ of axiom iv) of Definition \ref{def:WittComplex},
\[
V(F(z))=V(1)\cdot z
\]
for every $z\in \W_{\langle 2\rangle}\Omega^{q-1}_k$,
and $V(1)=2\in \W_{\langle 2\rangle}(k)$ since $k$ has characteristic $2$.

By rearranging the terms and grouping pairs $(\nu,\delta)$ with the same intersection $\beta$, we can write 
\begin{align}\label{groupBeta}
\begin{split}
\sum_{
\begin{smallmatrix}
\nu\subset I
\\
|\nu|=q-1
\end{smallmatrix}
}
\mu_\nu (dx)^\nu
&=
\sum_{
\begin{smallmatrix}
\nu\subset I
\\
|\nu|=q-1
\end{smallmatrix}
}
\sum_{\delta\subset^f I}s^2_{\nu,\delta}x^{\delta} (dx)^\nu
=
\sum_{\beta\subset^f I}
\sum_{
\begin{smallmatrix}
\nu\subset I
\\
|\nu|=q-1
\end{smallmatrix}
}
\sum_{
\begin{smallmatrix}
\delta\subset^f I
\\
\delta\cap\nu=\beta
\end{smallmatrix}
}
s^2_{\nu,\delta}x^{\delta}  (dx)^\nu
\\&=
\sum_{\beta\subset^f I}
x^{\beta}  (dx)^\beta
\sum_{
\begin{smallmatrix}
\nu\subset I
\\
|\nu|=q-1
\end{smallmatrix}
}
\sum_{
\begin{smallmatrix}
\delta\subset^f I
\\
\delta\cap\nu=\beta
\end{smallmatrix}
}
s^2_{\nu,\delta}x^{\delta\setminus\beta}  (dx)^{\nu\setminus\beta}.
\end{split}
\end{align}
Each term $x^{\beta}  (dx)^\beta$ is in the image of the Frobenius, since by axiom v) of Definition \ref{def:WittComplex} (we recall that in the de Rham-Witt complex the map $\lambda$ is the identity)
\begin{align*}
x^\beta(dx)^\beta&=\prod_{i\in\beta}x_idx_i=\prod_{i\in\beta}\tau_0(x_i)d\tau_0(x_i)=\prod_{i\in\beta}F(d\tau_1(x_i))=F(\prod_{i\in\beta}d\tau_1(x_i)).
\end{align*}
It is therefore sufficient to show that for every fixed $\beta\subset I$, the double sum in equation (\ref{groupBeta}) is in the image of the Frobenius.
Let us now group those terms by the union $\lambda$ of $\nu$ and $\delta$,  and write
\begin{align}\label{form:triplesum}
\sum_{
\begin{smallmatrix}
\nu\subset I
\\
|\nu|=q-1
\end{smallmatrix}
}
\sum_{
\begin{smallmatrix}
\delta\subset^f I
\\
\delta\cap\nu=\beta
\end{smallmatrix}
}
s^2_{\nu,\delta}x^{\delta\setminus\beta}  (dx)^{\nu\setminus\beta}
&=
\sum_{\lambda\subset^f I}
\sum_{
\begin{smallmatrix}
\nu\subset I
\\
|\nu|=q-1
\end{smallmatrix}
}
\sum_{
\begin{smallmatrix}
\delta\subset^f I
\\
\delta\cap\nu=\beta
\\
\delta\cup \nu=\lambda
\end{smallmatrix}
}
s^2_{\nu,\delta}x^{\delta\setminus\beta}  (dx)^{\nu\setminus\beta}.
\end{align}
We now show that for every fixed $\beta,\lambda\subset^f I$, the inner double sum in (\ref{form:triplesum}) is in the image of the Frobenius. Notice that, after fixing $\beta$ and $\lambda$, the subset $\delta$ is determined by $\nu$, and let us write $\delta_\nu:=(\lambda\setminus \nu)\amalg \beta$. That is, we show that
\begin{align}\label{form:anothersum}
\sum_{
\begin{smallmatrix}
\nu\subset I
\\
|\nu|=q-1
\\
\beta\subset\nu\subset\lambda
\end{smallmatrix}
}
s^2_{\nu,\delta_\nu}x^{\delta_\nu\setminus\beta}  (dx)^{\nu\setminus\beta}=\sum_{
\begin{smallmatrix}
\nu\subset I
\\
|\nu|=q-1
\\
\beta\subset\nu\subset\lambda
\end{smallmatrix}
}
s^2_{\nu,\delta_\nu}x^{\lambda\setminus \nu}  (dx)^{\nu\setminus\beta}
\end{align}
is in the image of the Frobenius. Let us first treat the case where $\beta=\lambda$ (with $q-1$ elements, otherwise the sum is trivially zero). In this case the sum is just $s^2_{\beta,\beta}$, and every square of $k$ is in the image of the Frobenius since $s^2=F(\tau_1(s))$. Thus, suppose that $\beta$ is a proper subset of $\lambda$, and choose an element $j_0\in \lambda\setminus\beta$. We claim that (\ref{form:anothersum}) is equal to
\[
\sum_{
\begin{smallmatrix}
\nu\subset I
\\
|\nu|=q-1
\\
\beta\subset\nu\subset\lambda
\\
j_0\in\nu
\end{smallmatrix}
}
s^2_{\nu,\delta_\nu}d(x^{(\lambda\setminus \nu)\amalg j_0})(dx)^{(\nu\setminus\beta)\setminus j_0}.
\]
This will conclude the proof, since any square is in the image of the Frobenius by the argument above, and so is each differential by the relation $d=FdV$ of axiom v) (observe that $\lambda\setminus \nu$ and $\nu\setminus\beta$ are disjoint, with union $\lambda\setminus\beta$, so that no summands contains the square of a differential). To see that the last claim holds, let us notice that by iterating the Leibniz rule
\[
d(x^{\xi})=\sum_{j\in\xi}x^{\xi\setminus j}dx_j
\]
for every finite subset $\xi\subset I$,
and therefore
\begin{align*}
&\sum_{
\begin{smallmatrix}
\nu\subset I
\\
|\nu|=q-1
\\
\beta\subset\nu\subset\lambda
\\
j_0\in\nu
\end{smallmatrix}
}
s^2_{\nu,\delta_\nu}d(x^{(\lambda\setminus \nu)\amalg j_0})(dx)^{(\nu\setminus\beta)\setminus j_0}
=
\sum_{
\begin{smallmatrix}
\nu\subset I
\\
|\nu|=q-1
\\
\beta\subset\nu\subset\lambda
\\
j_0\in\nu
\end{smallmatrix}
}
\sum_{j\in (\lambda\setminus \nu)\amalg j_0}
s^2_{\nu,\delta_\nu}x^{((\lambda\setminus \nu)\amalg j_0)\setminus j}(dx_{j})(dx)^{(\nu\setminus\beta)\setminus j_0}
\\
&=
\sum_{
\begin{smallmatrix}
\nu\subset I
\\
|\nu|=q-1
\\
\beta\subset\nu\subset\lambda
\\
j_0\in\nu
\end{smallmatrix}
}
\big(
s^2_{\nu,\delta_\nu}x^{\lambda\setminus \nu}(dx)^{\nu\setminus\beta}
+
\sum_{j\in \lambda\setminus \nu}
s^2_{\nu,\delta_\nu}x^{((\lambda\setminus \nu)\amalg j_0)\setminus j}(dx)^{(\nu\setminus\beta)\setminus j_0\amalg j}
\big)
\\
&=
\big(\sum_{
\begin{smallmatrix}
\nu\subset I
\\
|\nu|=q-1
\\
\beta\subset\nu\subset\lambda
\\
j_0\in\nu
\end{smallmatrix}
}
s^2_{\nu,\delta_\nu}x^{\lambda\setminus \nu}(dx)^{\nu\setminus\beta}
\big)
+
\big(\sum_{
\begin{smallmatrix}
\nu\subset I
\\
|\nu|=q-1
\\
\beta\subset\nu\subset\lambda
\\
j_0\in\nu
\end{smallmatrix}
}
\sum_{j\in \lambda\setminus \nu}
s^2_{\nu,\delta_\nu}x^{((\lambda\setminus \nu)\amalg j_0)\setminus j}(dx)^{(\nu\setminus\beta)\setminus j_0\amalg j}
\big).
\end{align*}
By setting $\zeta=(\nu\setminus j_0)\amalg j$ in the second sum, and observing that $\delta_\nu=(\delta_{\zeta}\setminus j_0)\amalg j$, we find that this is equal to
\begin{align}\label{form:anotherone}
&\big(\sum_{
\begin{smallmatrix}
\nu\subset I
\\
|\nu|=q-1
\\
\beta\subset\nu\subset\lambda
\\
j_0\in\nu
\end{smallmatrix}
}
s^2_{\nu,\delta_\nu}x^{\lambda\setminus \nu}(dx)^{\nu\setminus\beta}
\big)
+
\big(\sum_{
\begin{smallmatrix}
\zeta\subset I
\\
|\zeta|=q-1
\\
\beta\subset\zeta\subset\lambda
\\
j_0\notin\zeta
\end{smallmatrix}
}
\sum_{j\in \zeta\setminus\beta}
s^2_{(\zeta\amalg j_0)\setminus j,(\delta_{\zeta}\setminus j_0)\amalg j}x^{((\lambda\setminus (\zeta\amalg j_0\setminus j))\amalg j_0)\setminus j}(dx)^{((\zeta\amalg j_0\setminus j)\setminus\beta)\setminus j_0\amalg j}
\big)
\\&=
\big(\sum_{
\begin{smallmatrix}
\nu\subset I
\\
|\nu|=q-1
\\
\beta\subset\nu\subset\lambda
\\
j_0\in\nu
\end{smallmatrix}
}
s^2_{\nu,\delta_\nu}x^{\lambda\setminus \nu}(dx)^{\nu\setminus\beta}
\big)
+
\big(\sum_{
\begin{smallmatrix}
\zeta\subset I
\\
|\zeta|=q-1
\\
\beta\subset\nu\subset\lambda
\\
j_0\notin\zeta
\end{smallmatrix}
}
\big(
\sum_{j\in \zeta\setminus\beta}
s^2_{(\zeta\amalg j_0)\setminus j,(\delta_{\zeta}\setminus j_0)\amalg j}\big)x^{\lambda\setminus \zeta}(dx)^{\zeta\setminus\beta}
\big).
\end{align}
Finally, by applying the relation (\ref{snudelta}) for $\xi=\zeta\amalg j_0$ and $\alpha=\delta_{\zeta}\setminus j_0$, so that $\xi\setminus\alpha=(\zeta\setminus\beta)\amalg j_0$, we find that
\[
\sum_{j\in \zeta\setminus\beta}
s^2_{(\zeta\amalg j_0)\setminus j,(\delta_{\zeta}\setminus j_0)\amalg j}=s^2_{\zeta,\delta_\zeta},
\]
which identifies (\ref{form:anotherone}) and (\ref{form:anothersum}).
\end{proof}

Let us remark that, if $(x,y)$ is one of the generators of $J_{\langle 2^{n} \rangle}$ of Proposition \ref{prop:additivegen}, then $x$ and $y$, when regarded as elements of $k\otimes_S k$, belong to $J_{\langle 1 \rangle}$. Thus, if an element $(z,w)$ of $\pi_0\TRR^{n+1}(k;2)^{\phi\Z/2}$ belongs to $J_{\langle 2^{n} \rangle}^{q+1}$ for some $q\geq 0$, then $z$ and $w$ belong to $J_{\langle 1 \rangle}^{q+1}$. The following lemma strengthen this property when the second component $w$ is zero.

\begin{lemma}\label{lemma:key2}
Let $z\in (k\otimes_S k)^{C_2}$ be such that $(z,0)$ belongs to $J_{\langle 2^{n} \rangle}^{q+1}$ for some $q\geq 0$. Then $z$ belongs to $J_{\langle 1 \rangle}^{q+2}$.
\end{lemma}

\begin{proof}
By the characterisation of $J_{\langle 2^{n} \rangle}$ of Proposition \ref{prop:additivegen}, we see that $(x,0)$ can be expressed as a sum of elements $u_1\dots u_{q+1}$, with $u_j\in J_{\langle 2^{n} \rangle}$, of three types:
\begin{enumerate}
\item At least one of the $u_j$ is of the form $V^{n-i}\tau_i(a\otimes b)+V^{n-i}\tau_i(ab\otimes 1)$ for $a,b\in k$ and $0\leq  i\leq n-1$. We can then write this generator in components as
\[
u_1\dots u_{q+1}=(w_1+w_1',0)u_2\dots u_{q+1}
\]
where $w_1$ is the first component of $V^{n-i}\tau_i(a\otimes b)$ and $w_1'$ the first component of $V^{n-i}\tau_i(ab\otimes 1)$.
\item All of the $u_j$ are of the form $V^{n-i}\tau_i(a\otimes b)+\sigma V^{n-i}\tau_i(a\otimes b)$ for $a,b\in k$ and $0\leq i\leq n-1$. Then each factor is diagonal, that is $u_j=(v_j,v_j)$ where $v_j$ is in $J_{\langle 1\rangle}$, and
\[
u_1\dots u_{q+1}=(v_1\dots v_{q+1},v_1\dots v_{q+1}).
\]
\item It is not of the first two types. In this case, at least one of the $u_j$ is of the form $\tau_n(a\otimes b)+\tau_n(ab\otimes 1)$ and the other factors are diagonal. We can then write such a generator as
\begin{align*}
u_1\dots u_{q+1}&=(t_1,t_1')\dots (t_l,t_l')(s_{l+1},s_{l+1})\dots (s_{q+1},s_{q+1})
\\&=(t_1\dots t_ls_{l+1}\dots s_{q+1},t_1'\dots t_l's_{l+1}\dots s_{q+1})
\end{align*}
for some $1\leq l\leq q+1$, where $t_j=\phi_n(a_j\otimes b_j)+\phi_n(a_jb_j\otimes 1)$, $t_j'=\phi_n(b_j\otimes a_j)+\phi_n(1\otimes a_jb_j)$, and $s_j$ is in $J_{\langle 1\rangle}$.
\end{enumerate}
Thus, let us write $(z,0)$ as a sum of these types of generators, where we omit the indexing from the sums to make this expression more digestible:
\[
(z,0)=\sum (w_1+w_1',0)u_2\dots u_{q+1}
+\sum(v_1\dots v_{q+1},v_1\dots v_{q+1})
+\sum (t_1\dots t_ls_{l+1}\dots s_{q+1},t_1'\dots t_l's_{l+1}\dots s_{q+1})
\]
Since the second component of $(z,0)$ is null, we must have that
\[
\sum v_1\dots v_{q+1}=\sum t_1'\dots t_l's_{l+1}\dots s_{q+1}.
\]
By replacing the left-hand side in the first component above we have that $z$ is, by denoting $r_j$ the first component of $u_j$, of the form
\begin{align*}
z&=\sum (w_1+w_1')r_2\dots r_{q+1}+\sum t_1'\dots t_l's_{l+1}\dots s_{q+1}+\sum t_1\dots t_ls_{l+1}\dots s_{q+1}
\\&=\sum(w_1+w_1')r_2\dots r_{q+1}+\sum (t_1\dots t_l+t_1'\dots t_l')s_{l+1}\dots s_{q+1}.
\end{align*}
Thus, since the $r_j$ and $s_j$ belong to $J_{\langle 1 \rangle}$, to conclude the proof it is sufficient to show that $w_1+w_1'$ belongs to $J_{\langle 1 \rangle}^2$ and that $(t_1\dots t_l+t_1'\dots t_l')$ belongs to $J_{\langle 1 \rangle}^{l+1}$. For the first case, we have that
\begin{align*}
w_1+w_1'&=\phi_i(a\otimes b)+\phi_i(b\otimes a)+\phi_i(ab\otimes 1)+\phi_i(1\otimes ab)
\\&=
a^{2^i-1}b^{2^i}\otimes a+b^{2^i-1}a^{2^i}\otimes b+(ab)^{2^i-1}\otimes ab+(ab)^{2^i}\otimes 1
\\&=(ab)^{2^i-1}(b\otimes a+a\otimes b+1\otimes ab+ab\otimes 1)
\\&=(ab)^{2^i-1}(1\otimes a+a\otimes 1)(1\otimes b+b\otimes 1)
\end{align*}
which indeed belongs to $J_{\langle 1 \rangle}^2$. For the second case, we have that
\begin{align*}
(t_1\dots t_l+t_1'\dots t_l')&=\prod_{j=1}^l(\phi_n(a_j\otimes b_j)+\phi_n(a_jb_j\otimes 1))+\prod_{j=1}^l(\phi_n(b_j\otimes a_j)+\phi_n(1\otimes a_jb_j))
\\&=\prod_{j=1}^l(a_j^{2^n-1}b_j^{2^n}\otimes a_j+(a_jb_j)^{2^n-1}\otimes a_jb_j)+\prod_{j=1}^l(b_j^{2^n-1}a_j^{2^n}\otimes b_j+(a_jb_j)^{2^n}\otimes 1)
\\&
=\prod_{j=1}^l\big(((a_jb_j)^{2^n-1}\otimes a_j)\cdot  (b_j\otimes 1+1\otimes b_j)\big)
+
\prod_{j=1}^l\big(
((a_j)^{2^n}(b_j)^{2^n-1}\otimes 1)
\cdot  (b_j\otimes 1+1\otimes b_j)
\big)
\\&=
\big(\prod_{j=1}^l((a_jb_j)^{2^n-1}\otimes a_j)+\prod_{j=1}^l((a_j)^{2^n}(b_j)^{2^n-1}\otimes 1)\big)
\cdot
\prod_{j=1}^l  (b_j\otimes 1+1\otimes b_j)
\\&=
\big(\prod_{j=1}^l(a_jb_j)^{2^n-1}(1 \otimes a_j)+\prod_{j=1}^l(a_jb_j)^{2^n-1}(a_j\otimes 1)\big)
\cdot
\prod_{j=1}^l  (b_j\otimes 1+1\otimes b_j)
\\&=
(\prod_{j=1}^l(a_jb_j)^{2^n-1})\big(1 \otimes (\prod_{j=1}^la_j)+(\prod_{j=1}^la_j)\otimes 1\big)
\cdot
\prod_{j=1}^l  (b_j\otimes 1+1\otimes b_j)
\end{align*}
which belongs to $J_{\langle 1 \rangle}^{l+1}$.
\end{proof}

\begin{proof}[Proof of Theorem \ref{thm:dRW}]
By Remark \ref{rem:lowcases}, in degree $q=0$ the map $u_n$ is an isomorphism for all $n\geq 0$. Thus, let $q\geq 1$.
By Lemma \ref{lemma:trunc0}, the map $u_0\colon \Omega^q_k\longrightarrow J_{\langle 1\rangle}^q/J_{\langle 1\rangle}^{q+1}$ is an isomorphism for every $q\geq 1$. Thus, let $n\geq 1$, assume that $u_{n-1}$ is an isomorphism, and let us show that $u_{n}$ is an isomorphism.
Since $u_{n}$ is a map of Witt complexes, it clearly hits all the generators of $J_{\langle 2^{n} \rangle}^q$ from Proposition \ref{prop:additivegen}, and it is therefore surjective. To see that it is injective, consider the commutative diagram with exact rows
\[
\xymatrix@C=50pt{
\Omega^q_k\oplus\Omega^{q-1}_k\ar[r]^-{V^{n}+dV^n}
&
(\W_{\langle 2^{n}\rangle}\Omega^q_k)/2\ar[r]^-{R}\ar[d]^-{u_{n}}
&
(\W_{\langle 2^{n-1}\rangle}\Omega^q_k)/2\ar[r]\ar[d]^{u_{n-1}}_{\cong}
&
0
\\
&
J_{\langle 2^{n} \rangle}^q/J_{\langle 2^{n} \rangle}^{q+1}
\ar[r]^-R&
J_{\langle 2^{n-1} \rangle}^q/J_{\langle 2^{n-1} \rangle}^{q+1}\ar[r]&0
}
\]
where the top row is exact by \cite[Lemma 3.5]{Costeanu}. It then suffices to show that $u_{n}$ is injective when restricted to the image of the top left horizontal map $V^n+dV^n$. Thus, let $a\in  \Omega^q_k$ and $b\in  \Omega^{q-1}_k$, and suppose that $u_{n}(V^n(a)+dV^n(b))$ can be represented by an element of $J_{\langle 2^{n} \rangle}^{q+1}$. We need to show that $c:=V^n(a)+dV^n(b)$ is divisible by $2$ in $\W_{\langle 2^{n}\rangle}\Omega^q_k$. Let us explicitly calculate $u_{n}(V^n(a)+dV^n(b))$. Given a $2$-basis $\{x_{i}\}_{i\in I}$ of $k$, let us write
\[
a=\sum_{
\begin{smallmatrix}
\xi\subset I
\\
|\xi|=q
\end{smallmatrix}
}\lambda_\xi(dx)^\xi \ \ \ \ \ \ \ \mbox{and}\ \ \ \ \ \ \ \ 
b=\sum_{
\begin{smallmatrix}
\nu\subset I
\\
|\nu|=q-1
\end{smallmatrix}
}\mu_\nu(dx)^\nu
\]
where $\lambda_\xi,\mu_\nu\in k$ and $(dx)^\xi=\prod_{i\in \xi}dx_i$. Since $u$ is a map of Witt complexes, we must have that
\[
u_{n}(V^n(a)+dV^n(b))=V^n(u_0(a))+dV^n(u_0(b))=
\sum_{
\begin{smallmatrix}
\xi\subset I
\\
|\xi|=q
\end{smallmatrix}
}V^n(\lambda_\xi\Delta(x)^\xi) +
\sum_{
\begin{smallmatrix}
\nu\subset I
\\
|\nu|=q-1
\end{smallmatrix}
} dV^n(\mu_\nu \Delta(x)^\nu).
\]
Recall from Proposition \ref{prop:transfers} that for every $x\otimes y\in k\otimes_S k$ we have that $V^n(x\otimes y)$ is represented by $(\tran(x\otimes y),0)$ in $\pi_0\TRR^{n+1}(k;2)^{\phi\Z/2}$, where $\tran(x\otimes y)=x\otimes y+y\otimes x$. Thus since $d=1+\sigma$ we have
\begin{align*}
u_{n}(V^n(a)+dV^n(b))&=
\sum_{
\begin{smallmatrix}
\xi\subset I
\\
|\xi|=q
\end{smallmatrix}
}(\tran(\lambda_\xi\Delta(x)^\xi),0) +
\sum_{
\begin{smallmatrix}
\nu\subset I
\\
|\nu|=q-1
\end{smallmatrix}
} (\tran(\mu_\nu \Delta(x)^\nu),\tran(\mu_\nu \Delta(x)^\nu)).
\end{align*}
By choosing $i\in \xi$, we can write
\begin{align*}
\tran(\lambda_\xi\Delta(x)^\xi)&=\tran(\lambda_\xi\Delta(x_i)\Delta(x)^{\xi\setminus i})=\tran(\lambda_\xi\Delta(x_i))\Delta(x)^{\xi\setminus i}=\tran(\lambda_\xi\otimes x_i+\lambda_\xi x_i\otimes 1)\Delta(x)^{\xi\setminus i}
\\&=
(\lambda_\xi\otimes x_i+\lambda_\xi x_i\otimes 1+x_i\otimes \lambda_\xi+1\otimes \lambda_\xi x_i)\Delta(x)^{\xi\setminus i}=\Delta(\lambda_\xi)\Delta(x_i)\Delta(x)^{\xi\setminus i}
\\&=\Delta(\lambda_\xi)\Delta(x)^{\xi}
\end{align*}
where the second equality holds since $\Delta(x)^{\xi\setminus i}$ is fixed by the involution.
Similarly $\tran(\mu_\nu\Delta(x)^\nu)=\Delta(\mu_\nu)\Delta(x)^{\nu}$ (which is obvious in the case where $q=1$ and $\nu=\emptyset$). Thus we find that
\begin{align*}
u_{n}(V^n(a)+dV^n(b))&=
\sum_{
\begin{smallmatrix}
\xi\subset I
\\
|\xi|=q
\end{smallmatrix}
}(\Delta(\lambda_\xi)\Delta(x)^{\xi},0) +
\sum_{
\begin{smallmatrix}
\nu\subset I
\\
|\nu|=q-1
\end{smallmatrix}
} (\Delta(\mu_\nu)\Delta(x)^{\nu},\Delta(\mu_\nu)\Delta(x)^{\nu})
\\&=
\big(
\sum_{
\begin{smallmatrix}
\xi\subset I
\\
|\xi|=q
\end{smallmatrix}
}\Delta(\lambda_\xi)\Delta(x)^{\xi}
+
\sum_{
\begin{smallmatrix}
\nu\subset I
\\
|\nu|=q-1
\end{smallmatrix}
}\Delta(\mu_\nu)\Delta(x)^{\nu}
,
\sum_{
\begin{smallmatrix}
\nu\subset I
\\
|\nu|=q-1
\end{smallmatrix}
}\Delta(\mu_\nu)\Delta(x)^{\nu}
\big),
\end{align*}
and this element is by assumption in $J_{\langle 2^{n} \rangle}^{q+1}$. Let us analyse the two components separately, starting from the second one. As observed above Lemma \ref{lemma:key2}, these components must in fact belong to $J_{\langle 1 \rangle}^{q+1}$, and therefore
\[
\sum_{
\begin{smallmatrix}
\nu\subset I
\\
|\nu|=q-1
\end{smallmatrix}
}\Delta(\mu_\nu)\Delta(x)^{\nu} \in J_{\langle 1 \rangle}^{q+1}.
\]
By Lemma \ref{lemma:key1}, $V(\mu_\nu (dx)^\nu)$ vanishes in the de Rham Witt complex modulo $2$, and therefore so does
\[
dV^n(b)=dV^{n-1}(\sum_{
\begin{smallmatrix}
\nu\subset I
\\
|\nu|=q-1
\end{smallmatrix}
}V(\mu_\nu (dx)^\nu)).
\]
Our original element $c=V^n(a)+dV^n(b)$ is then equal to $V^n(a)$ in the de Rham-Witt complex modulo $2$, and the map $u_n$ sends this element to
\[
u_n(c)=V^n(a)=\big(
\sum_{
\begin{smallmatrix}
\xi\subset I
\\
|\xi|=q
\end{smallmatrix}
}\Delta(\lambda_\xi)\Delta(x)^{\xi}
,0
\big).
\]
Moreover, this element is by assumption in $J_{\langle 2^{n} \rangle}^{q+1}$. By applying Lemma \ref{lemma:key2}, the first component $\sum_{
\begin{smallmatrix}
\xi\subset I
\\
|\xi|=q
\end{smallmatrix}
}\Delta(\lambda_\xi)\Delta(x)^{\xi}$ in fact belongs to $J_{\langle 1 \rangle}^{q+2}$. Again by Lemma \ref{lemma:key1}, we find that 
\[
c=V^n(a)=V^{n-1}V(\sum_{
\begin{smallmatrix}
\xi\subset I
\\
|\xi|=q
\end{smallmatrix}
}\lambda_\xi(dx)^{\xi})=0\]
in $(\W_{\langle 2^{n}\rangle}\Omega^q_k)/2$, proving that $u_n$ is injective.
\end{proof}

\subsection{The Milnor conjecture and TCR}\label{sec:TCR}

Let us recall that the $2$-typical topological cyclic homology spectrum $\TC(k;2)$ of $k$ can be defined as the equaliser
\[
\xymatrix{\TC(k;2)\ar[r]&\TR(k;2)\ar@<.5ex>[r]^{\id}\ar@<-.5ex>[r]_{F}&\TR(k;2)}
\]
of the identity and the Frobenius map of $\TR(k;2)$. Let us denote $\W_{\langle 2^{\infty}\rangle}\Omega^\ast_k$ the limit of $\W_{\langle 2^{n}\rangle}\Omega^\ast_k$ over the map $R$, and define $\nu_\ast^{dRW/2}(k;2)$ and $\epsilon_\ast^{dRW/2}(k;2)$ respectively as the equaliser and coequaliser of the parallel group homomorphisms
\[
\xymatrix{\nu^\ast_{dRW/2}(k;2)\ar[r]&(\W_{\langle 2^{\infty}\rangle}\Omega^\ast_k)/2\ar@<.5ex>[r]^{\id}\ar@<-.5ex>[r]_{F}&(\W_{\langle 2^{\infty}\rangle}\Omega^\ast_k)/2\ar[r]&\epsilon^\ast_{dRW/2}(k;2)
}.
\]
Then, since the parallel arrows are ring homomorphisms, $\nu^\ast_{dRW/2}(k;2)$ is a graded ring, and $\epsilon^\ast_{dRW/2}(k;2)$ is a graded $\nu^\ast_{dRW/2}(k;2)$-module (where the action on the latter is defined via either of the maps $\id$ or $F$).
We now prove a version of the Milnor conjecture for $\nu^\ast_{dRW/2}(k;2)$ and $\epsilon^\ast_{dRW/2}(k;2)$, which describes these in terms of the graded ring associated to the kernel of the restriction map
\[
K:=\ker\big(\res^{\Z/2}_e\colon \pi_0\TCR(k;2)^{\phi\Z/2}\longrightarrow (\pi_0\TC(k;2))^{\Z/2}/Im(1+w)\big),
\]
where the involution on $\pi_0\TC(k;2)$ is induced by the involution underlying the $\Z/2$-spectrum $\TCR(k;2)$.
The restriction map is defined as we did at the beginning of \S\ref{sec:fundideal}. Let us also define the $\pi_{0}\TCR(k;2)^{\phi\Z/2}$-module
\[
T_{-1}:=\pi_{-1}\TCR(k;2)^{\phi\Z/2},
\]
so that the quotients $K^{\ast}T_{-1}/K^{\ast+1}T_{-1}$ form a graded $K^{\ast}/K^{\ast+1}$-module.

\begin{theorem}\label{thm:MilnorTCR} Let $k$ be a field of characteristic $2$.
There is an isomorphism of graded rings 
\[\nu_{dRW/2}^\ast(k;2)\cong K^\ast/K^{\ast+1},\]
and an isomorphism of graded $K^\ast/K^{\ast+1}$-modules
\[
\epsilon^\ast_{dRW/2}(k;2)\cong K^{\ast}T_{-1}/K^{\ast+1}T_{-1}.
\]
\end{theorem}

\begin{proof}
Let $\nu^\ast(k)$ and  be $\epsilon^\ast(k)$ be respectively the equaliser and coequaliser of the projection map and the inverse Cartier operator
\[
\xymatrix{\nu^\ast(k)\ar[r]&\Omega^\ast_k\ar@<.5ex>[r]^-{\pi}\ar@<-.5ex>[r]_-{C^{-1}}&\Omega^\ast_k/d(\Omega^{\ast-1}_k)\ar[r]&\epsilon^\ast(k)
}.
\]
The map $R$ of the de Rham-Witt complex induces multiplicative maps $\nu_{dRW/2}^\ast(k;2)\to \nu^\ast(k)$ and $\epsilon_{dRW/2}^\ast(k;2)\to \epsilon^\ast(k)$, which are isomorphisms by the proof of \cite[Proposition 2.26]{CMM}. Moreover, by \cite[Theorem (2)]{Kato}, the graded ring associated to the fundamental ideal $I$ of the symmetric Witt group is isomorphic to $\nu^\ast(k)$, and the graded module $I^\ast \W^q(k)/I^{\ast+1} \W^q(k)$ to $\epsilon^\ast(k)$. Thus, since  \cite[Theorem (1)]{Kato} and Corollary \ref{cor:geomTCRk} identify $\W^s(k)$ and $\pi_0\TCR(k;2)^{\phi\Z/2}$, as rings, with the equaliser of 
\[
\xymatrix{
 (k\otimes_S k)^{C_2}
\ar@<.5ex>[r]^-\pi\ar@<-.5ex>[r]_-\phi&
(k\otimes_S k)^{C_2}/Im(1+w)
},
\]
and $\W^q(k)$ and $T_{-1}$, as modules, with their coequaliser, it suffices to show that $I$ and $K$ correspond to the same ideal under these identifications. 

For the symmetric Witt group, the isomorphism with the equaliser is given by the unique additive map that sends the rank $1$ form $\langle a\rangle$, with $a\in k^\times$, to $a^{-1}\otimes a$. For $\pi_0\TCR(k;2)^{\phi\Z/2}$, it is induced by the map
\[
\pi_0\TCR(k;2)^{\phi\Z/2}\stackrel{R}{\longrightarrow}\pi_0\TRR^{2}(k;2)^{\phi\Z/2}\stackrel{c}{\longrightarrow} \pi_0(\THR(k)^{\phi\Z/2})^{C_2},
\]
followed by the identification of the target with $(k\otimes_S k)^{C_2}$ from Proposition \ref{prop:geofixgenuine}. Let us note that, after including the fixed points into $k\otimes_Sk$, this is the map
\[
\pi_0\TCR(k;2)^{\phi\Z/2}\stackrel{R}{\longrightarrow} \pi_0\THR(k)^{\phi\Z/2}\cong k\otimes_Sk,
\]
where the isomorphism is from \cite[Theorem 5.1]{THRmodels} (there it is stated for the fixed points, but since the transfer map of $k$ is zero the isomorphism descends to the geometric fixed points). 
Thus, in order to compare $I$ and $K$, it suffices to show that, under the isomorphism of Corollary \ref{cor:geomTCRk}, the restriction map from $\pi_0\TCR(k;2)^{\phi\Z/2}$ to $(\pi_0\TC(k;2))^{\Z/2}/Im(1+w)$ sends $a^{-1}\otimes a$ to $1$. As we do not have a good handle of $\pi_0\TC(k;2)$ for a general field $k$, we found ourselves unable to prove this by direct calculation. 

Instead, we can employ the existence of a trace map of $\Z/2$-equivariant spectra $\tr\colon \KR(k)\to \TCR(k;2)$ which lifts the trace map $\K(k)\to \TC(k;2)$ from \cite{BHM}. This trace map is constructed in the forthcoming paper \cite{HNS} in the setting of Poincar\'e $\infty$-categories.
For the purpose of our paper, we content ourselves with giving a point-set construction of this trace map in the case of rings with involution, as carried out in Proposition \ref{trace} below.
In fact, we need very little from this trace map: since $\tr\colon \KR(k)\to \TCR(k;2)$ is a map of $\Z/2$-equivariant spectra and $\W^{s}(k)\cong \pi_0\KR(k)^{\phi\Z/2}$, it induces a commutative square
\[
\xymatrix{
\W^{s}(k)\ar[d]_{rk}\ar@{-->}[r]^-{\tr}&\pi_0\TCR(k;2)^{\phi\Z/2}\ar[d]^-{\res^{\Z/2}_e}\ar[r]^-R&\pi_0\THR(k)^{\phi\Z/2}\cong k\otimes_Sk
\\
\Z/2\ar[r]^-{\tr}&(\pi_0\TC(k;2))^{\Z/2}/Im(1+w)
}
\]
where the bottom map is induced by the usual trace map from \cite{BHM}. The composite on the top row sends the rank $1$ form $\langle a\rangle$ to $a^{-1}\otimes a$, as proved in  Proposition \ref{trace}. It follows that the top trace map must be an isomorphism, and since the bottom map is a ring homomorphism and therefore injective, the respective vertical kernels $I$ and $K$ are then isomorphic.
\end{proof}

\begin{rem}
In the proof of Theorem \ref{thm:MilnorTCR}, we are using the Milnor conjecture twice: once in order to identify $\nu^\ast(k)$ with the graded ring of $I$, and then in order to identify $I$ with the kernel of $\pi-\phi$. It seems plausible that one could find a proof of the theorem which does not use Kato's Theorems. Define $W$ and $T$ respectively as the equaliser and coequaliser of 
\[
\xymatrix{W\ar[r]&k\otimes_S k\ar@<.5ex>[r]^-{\pi}\ar@<-.5ex>[r]_-{\phi}&(k\otimes_S k)/Im(1+w)\ar[r]&T
},
\]
and $K$ as the kernel of the multiplication map $\mu\colon W\to \Z/2$. One can then try to directly show that the induced sequence
\[
0\to K^{n}/K^{n+1}\to J^{n}/J^{n+1}\xrightarrow{\pi-\phi} J^{n}/((J^{n+1}+B)\cap J^n)\to K^{n}T/K^{n+1}T\to 0
\]
remains exact, where $B$ is the subgroup of $k\otimes_Sk$ generated by the elements $a\otimes b+b\otimes a$. This would then prove Theorem \ref{thm:MilnorTCR} because, under the isomorphism $\Omega^\ast_k\cong J^\ast/J^{\ast+1}$, the kernel of the middle map corresponds to $\nu^\ast(k)$, and the cokernel to $\epsilon^\ast(k)$ (see \cite[Fact 6]{Arason}).
\end{rem}

\appendix
\section{The trace map for rings with involution}\label{sec:trace}

Let us finish the paper with a construction of a $\Z/2$-equivariant lift of the trace map $\tr\colon \K\to \TC(-;p)$, for every prime $p$. 
The trace was first constructed by B\"okstedt-Hsiang-Madsen in \cite{BHM} as a natural transformation $\K\to \TC$ on the category of ring spectra. This construction has been extended to various settings, most notably as a natural transformation of functors from stable infinity categories, see e.g. \cite{BGTmult}.

A $\Z/2$-equivariant extension of this map as a natural transformation $\KR\to \TCR$ of functors from Poincar\'e categories will appear in forthcoming work of Harpaz-Nikolaus-Shah \cite{HNS}.
For the purpose of this article, it will be more than sufficient to define the $\Z/2$-equivariant trace map on the category of discrete rings with involution. We will give a construction in line with the construction of the Dennis trace as carried out in \cite[\S2.6]{IbSurvey}. After restricting down to $\THR$ this construction agrees with the one from \cite{DO}, which was defined for ring spectra with involution.

Let $A$ be a ring with involution $w\colon A^{op}\to A$, and $\GL_n(A)$ the group of invertible $n\times n$-matrices with the involution that sends $M$ to $M^\ast:=w(M)^T$, where $(-)^T$ denotes the matrix transposition and $w$ is applied to $M$ entrywise. The set of fixed points of $\GL_n(A)$ is the set of symmetric matrices
\[\GL_n(A)^{\Z/2}=\{M\in \GL_n(k)\ | \ M^\ast=M\}\]
and $\GL_n(A)$ acts on it by $g\cdot M=gMg^\ast$. We let $B^\sigma \GL_n(A)$ be the classifying space $B\GL_n(A)$ with the involution of \cite[Proposition 1.1.3]{BF}. Its $\Z/2$-fixed-points space is the bar construction of the action of $\GL_n(A)$ on $\GL_n(A)^{\Z/2}$ above (see also \cite[\S 2.1]{DO} for the details). For the purpose of this paper, we define $\KR(A)$ to be the $\Z/2$-equivariant group-completion of the $\Z/2$-equivariant $E_{\infty}$-monoid with involution 
\[
\coprod_{n\geq 0}B^\sigma \GL_n(A),
\]
where the monoid operation is induced by the direct sum of matrices. This is in fact the classifying space of the symmetric monoidal category with duality of finite dimensional free $A$-modules, and therefore indeed a  $\Z/2$-equivariant $E_{\infty}$-monoid.
By construction, $\pi_0(\KR(A)^{\Z/2})$ is the group-completion of the commutative monoid
\[
\coprod_{n\geq 0}\pi_0((B^\sigma \GL_n(A))^{\Z/2})\cong \coprod_{n\geq 0}\GL_n(A)^{\Z/2}/\GL_n(A),
\]
which is the Grothendieck-Witt group $\GW^s(A)$ of symmetric forms of free $A$-modules. The transfer map is induced by the functor that sends a free module of rank $n$ to the hyperbolic matrix of size $2n$, and therefore $\pi_0(\KR(A)^{\phi\Z/2})$ is the symmetric Witt group $W^{s}(A)$ (again of free $A$-modules).

Let us also recall from \cite[Theorem 5.1]{THRmodels} that there is an isomorphism of abelian groups
\[\pi_0(\THR(A)^{\Z/2})\cong (A^{\Z/2}\otimes A^{\Z/2})/T\]
where $A^{\Z/2}$ is the subgroup of fixed points of the involution $w$, and the quotient is by the subgroup $T$ generated by the elements of the form (i) and (ii) from \cite[Theorem 5.1]{THRmodels}. In particular, for $A=k$ a ring of characteristic $2$ with trivial involution, this is $k\otimes_Sk$, and since the transfer map $(a+w(a))\otimes 1$ of \cite[Theorem 5.1]{THRmodels} is in this case zero, we have as well that $\pi_0(\THR(k)^{\phi \Z/2})\cong k\otimes_Sk$.

\begin{prop}\label{trace} Let $A$ be a ring with involution.
For every prime $p$, there is a map of $\Z/2$-spectra $\tr\colon \KR(A)\to \TCR(A;p)$ which forgets to the $K$-theoretic trace map of \cite{BHM}. The composite
\[
\GW^s(A)=\pi_0(\KR(A)^{\Z/2})\xrightarrow{\tr} \pi_0(\TCR(A;2)^{\Z/2})\xrightarrow{R} \pi_0(\THR(A)^{\Z/2})\cong (A^{\Z/2}\otimes A^{\Z/2})/T
\]
sends the element of $\GW^s(A)$ represented by a symmetric form $x$ on $A^{\oplus n}$ to
\[
\tr(x)=\sum_{i=1}^n\big((x^{-1})_{ii}\otimes x_{ii}-(x^{-1})_{ii}x_{ii}\otimes 1\big)+n\otimes 1,
\]
where $x_{ii}$ are the entries of the matrix of $x$ for the standard basis of $A^{\oplus n}$, and $x^{-1}$ denotes the inverse matrix.
\end{prop}

\begin{proof}
We construct the trace by employing a construction completely analogous to the one from Dennis and B\"okstedt-Hsiang-Madsen, as explained in \cite[\S2.6]{IbSurvey}.
Let $B^{di}\GL_n(A)$ be the dihedral bar construction of $\GL_n(A)$, defined as the geometric realisation of the dihedral nerve $N^{di}\GL_n(A)$, which is the cyclic nerve of $\GL_n(A)$ with the involution analogous to the one of $\THR(k)$ from \S \ref{prelim}. 
Its  $\Z/2$-fixed-points space is the two-sided bar construction of the action of $\GL_n(A)$ on $\GL_n(A)^{\Z/2}$, and we refer to \cite[\S 2.1]{DO} for the details.

We define the trace map from maps of $\Z/2$-spaces
\begin{equation}\label{formulatrace}
B^\sigma \GL_n(A)\stackrel{s}{\to} B^{di}\GL_n(A)\to B^{di}\M_n(A)\to \Omega^{\infty}(\THR(\M_n(A))^{C_r})\stackrel{m}{\simeq} \Omega^{\infty}(\THR(A)^{C_r}),
\end{equation}
for every $r\geq 1$, by taking the disjoint union over $n\geq 0$ and group-completing the source with respect to direct sums.
The maps in this composite are defined as follows.
The map $s$ is the canonical section, which is defined on an $n$-simplex $(g_1,\dots,g_n)$ by
\[
s(g_1,\dots,g_n)=((g_1\dots g_n)^{-1},g_1,\dots,g_n).
\]
The second map of (\ref{formulatrace}) is the inclusion of invertible matrices into the monoid of all $(n\times n)$-matrices $\M_n(A)$, again with the transposition of matrices and entrywise $w$ as involution. For the third map, we use that $B^{di}\M_n(A)$ is the geometric realisation of the dihedral nerve $N^{di}\M_n(A)$, and therefore has an action of the dihedral group $D_r$ of order $2r$, for every integer $r\geq 1$.
The realisation of the $r$-subdivision $\sd_r$ of \cite[\S1]{BHM} applied to the dihedral nerve $N^{di}\M_n(A)$ has a $\Z/2$-action, and its geometric realisation is $D_r$-equivariantly isomorphic to $B^{di}\M_n(A)$ (see \cite[\S1.2]{geomTCR} for a detailed discussion about subdivisions of dihedral objects). Thus, we obtain $\Z/2$-equivariant isomorphisms
\[
\xymatrix{
B^{di}\M_n(A)\ar[r]^-{\Delta_r}_-{\cong}
&
 |(\sd_r N^{di}\M_n(A))^{C_r}|\ar[r]^-{E_r}_-{\cong}
 & (B^{di}\M_n(A))^{C_r}
 }
\]
where the first map is induced by the diagonal map degreewise, and the second map is the canonical isomorphism $E_r\colon |\sd_rX|\to |X|$ for a dihedral set $X$, from \cite[Lemma 1.1]{BHM} (which is denoted by $D_r$ there). 
By denoting $\Delta^k$ the standard $k$-simplex space
\[
\Delta^k:=\{(t_0,\dots,t_k)\in\mathbb{R}^{k+1}\ | \ t_0+t_1+\dots+t_k=1,\ t_i\geq 0\ \mbox{for all $0\leq i\leq k$} \},
\]
the map $E_r$ sends the equivalence class of $(x;t)$, with $x\in (\sd_rX)_k=X_{r(k+1)-1}$ and $t\in \Delta^{k}$, to the class of $(x;\delta_r(t))$, where $\delta_r\colon \Delta^{k}\to  \Delta^{r(k+1)-1}$ sends $t$ to $(t,\dots,t)/r$ (with $r$-many components).
The third map of (\ref{formulatrace}) is then defined to be the adjoint of the composite of the maps of $\Z/2$-spectra
\[
\xymatrix@R=10pt@C=17pt{
\Sigma^{\infty}_+B^{di}\M_n(A)\ar[rr]^-{\Sigma^\infty_+(E_r\Delta_r)}&&\Sigma^{\infty}_+(B^{di}\M_n(A))^{C_r}\ar[r]& (\Sigma^{\infty}_+B^{di}\M_n(A))^{C_r}\ar[r]^-{\cong}&
 (B^{di}\Sigma^{\infty}_+\M_n(A))^{C_r}
 \ar[d]
 \\&&&
 \THR(\M_n(A))^{C_r}
 &
\ar[l]_-{=} (B^{di}\EM \M_n(A))^{C_r}
}
\]
which are respectively $\Sigma^\infty_+(E_r\Delta_r)$, the tom Dieck splitting, the monoidality of the equivariant suspension spectrum functor, and the Hurewicz map.
Finally, the last map of (\ref{formulatrace}) is induced by the inclusion of $A$ into $\M_n(A)$ as $(1,1)$-entry, and it is a $\Z/2$-equivalence by \cite[Theorem 4.9]{THRmodels} (which visibly restricts to $C_r$-fixed points).

A direct verification shows that the map of (\ref{formulatrace}) is compatible with the direct sum of matrices and with its symmetry isomorphism, and therefore by setting $r=p^{m-1}$ we obtain a map of $\Z/2$-spectra
\[
\tr^m\colon \KR(A)\longrightarrow \THR(A)^{C_{p^{m-1}}}=\TRR^{m}(A;p)
\]
for every integer $m\geq 1$ and prime $p$. To obtain a map to $\TCR(A;p)$ we need to show that the maps $\tr^m$ are compatible with the restriction and Frobenius $R,F\colon \TRR^{m+1}(A;p)\to \TRR^{m}(A;p)$. Unravelling the definitions, it is sufficient to provide $\Z/2$-equivariant homotopies between the composites from the bottom left to the bottom right $\Z/2$-spaces of the diagram
\[
\xymatrix@C=50pt{
&&(B^{di}\M_n(A))^{C_{p^{m}}}\ar@<-1ex>[d]_-{R}\ar@<1ex>[d]^-{F}
\\
B^{\sigma}\GL_n(A)\ar[r]_-s&B^{di}\M_n(A)\ar[ur]^-{E_{p^m}\Delta_{p^{m}}}\ar[r]_-{E_{p^{m-1}}\Delta_{p^{m-1}}}&(B^{di}\M_n(A))^{C_{p^{m-1}}}
}
\]
The vertical map $R$ is defined by identifying the $C_p$-fixed points of the $p^m$-fold subdivision with the $p^{m-1}$-fold subdivision, and then taking $C_{p^{m-1}}$-fixed points. It follows from \cite[(1.12)]{BHM} that the inner most triangle commutes strictly. The vertical map $F$ is the inclusion of fixed points. To see how the outer triangle commutes when restricted along $s$, we decompose the diagram as
\[
\xymatrix@R=12pt@C=17pt{
&&&|(\sd_{p^m}N^{di}\M_n(A))^{C_{p^{m}}}|\ar[rr]^-{E_{p^m}}\ar[d]&&(B^{di}\M_n(A))^{C_{p^{m}}}\ar[dd]^-{F}
\\
&&&|(\sd_{p^m}N^{di}\M_n(A))^{C_{p^{m-1}}}|\ar[d]^-{E_p}\ar[drr]^-{E_{p^m}}
\\
B^{\sigma}\GL_n(A)\ar[r]_-s&B^{di}\M_n(A)\ar@<.3ex>[uurr]^-{\Delta_{p^{m}}}\ar[rr]_-{\Delta_{p^{m-1}}}
&&|(\sd_{p^{m-1}}N^{di}\M_n(A))^{C_{p^{m-1}}}|\ar[rr]_-{E_{p^{m-1}}}&&(B^{di}\M_n(A))^{C_{p^{m-1}}}
}
\]
where the unlabelled map is the inclusion of fixed points. The lower right triangle commutes by \cite[(1.12)]{BHM}, and the square above it by naturality of the inclusion of fixed points. We then need to define a $\Z/2$-equivariant homotopy that makes the the triangle on the left commute when restricted along $s$. 
By factoring $\Delta_{p^{m}}=\Delta_{p^{m-1}}\circ\Delta_p$, it is sufficient to treat the case where $m=1$.
The homotopy provided in \cite[Proposition 2.5]{BHM} is not quite $\Z/2$-equivariant, but we can use a small variation of it. Let us define $h_k\colon \Delta^k\times [0,1]\to \Delta^{p(k+1)-1}$ for every $k\geq 0$, by
\[
h(t,s)=(st/p+(1-s)t,\dots,st/p,st/p)
\]
where the right-hand side has $p$ components.
If we apply the subdivision $\sd_e$ as in \cite[\S 1.2]{geomTCR} to make the $\Z/2$-actions on the spaces of the diagram simplicial,
the upper composite sends the equivalence class of $(g_1,\dots,g_{2k+1};t)$, with $(g_1,\dots,g_{2k+1})$ a $k$-simplex of $\sd_eN^{\sigma}\GL_n(A)$ and $t\in \Delta^k$, to the equivalence class of
\[
(\Delta_{p}((g_1\dots g_{2k+1})^{-1},g_1,\dots,g_{2k+1});\delta_p(t)).
\]
The lower composite is simply the functor $\sd_e$ applied to the section $s$.
Thus by sending the same equivalence class to the class of $(\Delta_{p}((g_1\dots g_{2k+1})^{-1},g_1,\dots,g_{2k+1});h_k(s,t))$, we obtain a $\Z/2$-equivariant homotopy from the upper composite to
\begin{align*}
[\Delta_{p}((g_1\dots g_{2k+1})^{-1},g_1,\dots,g_{2k+1});t,0,\dots,0]&=[d_l^{(k+1)(p-1)}\Delta_{p}((g_1\dots g_{2k+1})^{-1},g_1,\dots,g_{2k+1});t],
\end{align*}
where each $0$ on the left is the zero vertex of $\Delta^k$, and $d_l$ is the last face map of $\sd_eN^{di}\M_n(A)$.
Since this last face map multiplies the central three components $a_q,a_{q+1}$ and $a_{q+2}$ of a $q$-simplex $(a_0,\dots,a_{2q+1})$ of $\sd_eN^{di}\M_n(A)$, we find that, by denoting $g_0:=(g_1\dots g_{2k+1})^{-1}$,
\begin{align*}
d_l^{(k+1)(p-1)}\Delta_{p}(g_0,\dots,g_{2k+1})&=(g_0,
\dots,g_{k},(g_{k+1}\dots g_{2k+1}(g_0\dots g_{2k+1})^{p-2}g_0\dots g_{k+1}),
g_{k+2},\dots,g_{2k+1}).
\end{align*}
The middle entry is equal to $g_{k+1}$ since $g_0=(g_1\dots g_{2k+1})^{-1}$, and it follows that the end of the homotopy is indeed the subdivision of $s$. We can therefore lift $\tr^m$ along $R$ and $F$ to obtain a map of $\Z/2$-spectra $\tr\colon \KR(A)\to\TCR(A;p)$, for every prime $p$.

Let us now identify the effect of the trace in $\pi_0$ of the fixed points. By construction, if we compose $\tr$ with the map $R$ all the way to $\THR(A)$ we recover the map $\tr^0\colon \KR(A)\to \THR(A)$. Thus for this calculation, we need to describe the map (\ref{formulatrace}) for $r=1$. On fixed points, the map $s$ is the map of bar constructions
\[
B(\GL_{n}(A);\GL_n(A)^{\Z/2})
\stackrel{s}{\longrightarrow}
B(\GL_n(A)^{\Z/2};\GL_{n}(A);\GL_n(A)^{\Z/2})
\]
which sends a $k$-simplex $(g_1,\dots,g_k,x)$, with $g_i\in \GL_n(A)$ and $x\in \GL_n(A)^{\Z/2}$, to
\[
s(g_1,\dots,g_k,x)=((g_1\dots g_kxg_k^\ast\dots g^\ast_1)^{-1},g_1,\dots,g_k,x),
\]
where $(-)^\ast$ denotes the involution on $\GL_n(A)$. Thus, after applying $\pi_0$ and identifying the components of $\THR$ using \cite[Theorem 5.1]{THRmodels}, the map of (\ref{formulatrace}) becomes a map
\[
\GL_n(A)^{\Z/2}/_\sim\stackrel{s}{\to} (\GL_n(A)^{\Z/2}\times\GL_n(A)^{\Z/2})/_\sim\to (\M_n(A)^{\Z/2}\otimes \M_n(A)^{\Z/2})/T\stackrel{m}{\cong} (A^{\Z/2}\otimes A^{\Z/2})/T,
\]
where the quotients on the two sets on the left are for the respective actions of $\GL_n(A)$.
By the calculation of $s$ above, this map sends the isomorphism class of a form of rank $n$, represented by a matrix $x\in \GL_n(A)^{\Z/2}$, to $m(x^{-1}\otimes x)$.
For the proof of Theorem \ref{thm:MilnorTCR} we were only interested in rank $1$ forms (since these generate the Witt group of a field), and since for $n=1$ the map $m$ is the identity, we immediately find that the class of a symmetric form determined by a unit $a$ of $A$ fixed by the involution is sent to $a^{-1}\otimes a$.

For larger values of $n$ we need to determine the isomorphism
\[
m\colon (\M_n(A)^{\Z/2}\otimes \M_n(A)^{\Z/2})/T\stackrel{\cong}{\longrightarrow} (A^{\Z/2}\otimes A^{\Z/2})/T.
\]
Let us decompose a symmetric matrix $M\in \M_n(A)^{\Z/2}$ as $M=\sum_{i=1}^nM_{ii}e_{ii}+\sum_{1\leq i<j\leq n}(M_{ij}e_{ij}+w(M_{ij})e_{ji})$ where $e_{ij}$ is the canonical basis element with $1$ in the entry $(i,j)$ and with all the other entries equal to zero. By regarding the abelian group with involution $\M_n(A)$ as a Mackey functor, we can then write the fixed point $M$ as
\[
M=\sum_{i=1}^nM_{ii}e_{ii}+\sum_{1\leq i<j\leq n}\tran(M_{ij}e_{ij})
\]
where $\tran$ denotes the transfer map of the Mackey functor, which sends a matrix $N$ to $N+N^\ast$. By applying the same decomposition to a second fixed point $M'\in \M_n(A)^{\Z/2}$ we find that
\begin{align*}
M'\otimes M&=\sum_{l,i=1}^nM'_{ll}e_{ll}\otimes M_{ii}e_{ii}+\sum_{l=1}^n\sum_{1\leq i<j\leq n}M'_{ll}e_{ll}\otimes\tran(M_{ij}e_{ij})
\\&+\sum_{1\leq l<k\leq n}\sum_{i=1}^n\tran(M'_{lk}e_{lk})\otimes M_{ii}e_{ii}+\sum_{1\leq l<k\leq n}\sum_{1\leq i<j\leq n}^n\tran(M'_{lk}e_{lk})\otimes\tran(M_{ij}e_{ij}).
\end{align*}
By the relation (ii) of \cite[Theorem 5.1]{THRmodels} defining the subgroup $T$, this is equivalent to
\begin{align*}
M'\otimes M&=\sum_{l,i=1}^nM'_{ll}e_{ll}\otimes M_{ii}e_{ii}+\sum_{l=1}^n\sum_{1\leq i<j\leq n}\tran(M'_{ll}e_{ll} M_{ij}e_{ij})
\otimes 1
\\
&+\sum_{1\leq l<k\leq n}\sum_{i=1}^n1\otimes \tran(M'_{lk}e_{lk}M_{ii}e_{ii})+\sum_{1\leq l<k\leq n}\sum_{1\leq i<j\leq n}\tran(M'_{lk}e_{lk}(M_{ij}e_{ij}+w(M_{ij})e_{ji}))\otimes 1
\\
&=\sum_{l,i=1}^nM'_{ll}e_{ll}\otimes M_{ii}e_{ii}+\sum_{1\leq i<j\leq n}\tran(M'_{ii}M_{ij}e_{ij})\otimes 1+\sum_{1\leq l<k\leq n}1\otimes \tran(M'_{lk}M_{kk}e_{lk})
\\
&+\sum_{1\leq l<k<j\leq n}\tran(M'_{lk}M_{kj}e_{lj})\otimes 1
+\sum_{1\leq l,i<k\leq n}\tran(M'_{lk}w(M_{ik})e_{li})\otimes 1
\end{align*}
where the last equality follows from carrying out the matrix multiplication on the canonical basis. Again by \cite[Theorem 5.1]{THRmodels}, the transfer map of the Mackey functor $\underline{\pi}_0 \THR(\M_n(A))$ sends the equivalence class of a matrix $M$ in $\pi_0\THH(\M_n(A))\cong \M_n(A)/[\M_n(A),\M_n(A)]$, to $1\otimes \tran(M)=\tran(M)\otimes 1$ in $(\M_n(A)^{\Z/2}\otimes \M_n(A)^{\Z/2})/T$. 
The map $m$ from \cite[Theorem 4.9]{THRmodels} is a map of $\Z/2$-spectra, and therefore it commutes with the transfer. Moreover, since the underlying map of spectra is the trace map of \cite{BHM}, in $\pi_0$ it sends a matrix to its trace, and therefore the terms involving $e_{ij}$ vanish for $i\neq j$. We then find that
\begin{align*}
m(M'\otimes M)
&=\sum_{l,i=1}^nm(M'_{ll}e_{ll}\otimes M_{ii}e_{ii})+\sum_{1\leq l<k\leq n}\tran(M'_{lk}w(M_{lk}))\otimes 1
\end{align*}
in $(A^{\Z/2}\otimes A^{\Z/2})/T$. By the definition of $m$ of \cite[Proof of Theorem 4.9]{THRmodels}, it sends $e_{ij}\otimes e_{lk}$ to $1$ if $j=l$ and $k=i$, and to zero otherwise. Thus
\begin{align*}
m(M'\otimes M)
&=\sum_{i=1}^nM'_{ii}\otimes M_{ii}+\sum_{1\leq l<k\leq n}\tran(M'_{lk}w(M_{lk}))\otimes 1.
\end{align*}
Since $M$ is symmetric, i.e. $w(M_{lk})=M_{kl}$, we may write this expression as
\begin{align*}
m(M'\otimes M)
&=\sum_{i=1}^nM'_{ii}\otimes M_{ii}+\tr(M'M)\otimes 1-\sum_{i=1}^nM'_{ii}M_{ii}\otimes 1,
\end{align*}
where $\tr$ denotes the usual trace of a matrix. When $M'\otimes M=x^{-1}\otimes x$ for some $x\in \GL_n(A)^{\Z/2}$, this gives the formula we wanted.
\end{proof}

\bibliographystyle{amsalpha}
\bibliography{bib}

\newcommand{\etalchar}[1]{$^{#1}$}
\providecommand{\bysame}{\leavevmode\hbox to3em{\hrulefill}\thinspace}
\providecommand{\MR}{\relax\ifhmode\unskip\space\fi MR }
\providecommand{\MRhref}[2]{%
  \href{http://www.ams.org/mathscinet-getitem?mr=#1}{#2}
}
\providecommand{\href}[2]{#2}
\begin{thebibliography}{CDH{\etalchar{+}}20b}

\bibitem[AN21]{AN}
Benjamin Antieau and Thomas Nikolaus, \emph{Cartier modules and cyclotomic
  spectra}, J. Amer. Math. Soc. \textbf{34} (2021), no.~1, 1--78. \MR{4188814}

\bibitem[Ara20]{Arason}
J\'{o}n~Kr. Arason, \emph{On differential forms over fields of characteristic
  2}, Online Note, 2020.

\bibitem[BDS16]{BrDuSt}
Morten Brun, Bj{\o}rn~Ian Dundas, and Martin Stolz, \emph{Equivariant structure
  on smash powers}, arXiv:1604.05939, 2016.

\bibitem[BF84]{BF}
D.~Burghelea and Z.~Fiedorowicz, \emph{Hermitian algebraic {$K$}-theory of
  topological spaces}, Algebraic {$K$}-theory, number theory, geometry and
  analysis ({B}ielefeld, 1982), Lecture Notes in Math., vol. 1046, Springer,
  Berlin, 1984, pp.~32--46. \MR{750675}

\bibitem[BGT14]{BGTmult}
Andrew~J. Blumberg, David Gepner, and Gon\c~calo Tabuada, \emph{Uniqueness of
  the multiplicative cyclotomic trace}, Adv. Math. \textbf{260} (2014),
  191--232. \MR{3209352}

\bibitem[BHM93]{BHM}
M.~B{\"o}kstedt, W.~C. Hsiang, and I.~Madsen, \emph{The cyclotomic trace and
  algebraic {$K$}-theory of spaces}, Invent. Math. \textbf{111} (1993), no.~3,
  465--539. \MR{1202133 (94g:55011)}

\bibitem[B{\"o}k86]{Bok}
Marcel B{\"o}kstedt, \emph{Topological {H}ochschild homology}, preprint, 1986.

\bibitem[CDH{\etalchar{+}}20a]{9II}
Baptiste Calm\`es, Emanuele Dotto, Yonatan Harpaz, Fabian Hebestreit, Markus
  Land, Kristian Moi, Denis Nardin, Thomas Nikolaus, and Wolfgang Steimle,
  \emph{Hermitian {K}-theory for stable {$\infty$}-categories {II}: Cobordism
  categories and additivity}, arXiv:2009.07224, 2020.

\bibitem[CDH{\etalchar{+}}20b]{9III}
\bysame, \emph{Hermitian {K}-theory for stable {$\infty$}-categories {III}:
  Grothendieck-{W}itt groups of rings}, arXiv:2009.07225, 2020.

\bibitem[CDH{\etalchar{+}}23]{9I}
\bysame, \emph{Hermitian {K}-theory for stable {$\infty$}-categories {I}:
  foundations}, Selecta Math. (N.S.) \textbf{29} (2023), no.~1, Paper No. 10,
  269. \MR{4514986}

\bibitem[CMM21]{CMM}
Dustin Clausen, Akhil Mathew, and Matthew Morrow, \emph{{$K$}-theory and
  topological cyclic homology of henselian pairs}, J. Amer. Math. Soc.
  \textbf{34} (2021), no.~2, 411--473. \MR{4280864}

\bibitem[Cos08]{Costeanu}
Viorel Costeanu, \emph{On the 2-typical de {R}ham-{W}itt complex}, Doc. Math.
  \textbf{13} (2008), 413--452. \MR{2520474}

\bibitem[DKNP23]{DKNP2}
Emanuele Dotto, Achim Krause, Thomas Nikolaus, and Irakli Patchkoria,
  \emph{Witt vectors with coefficients and {TR}}, Arxiv: 2312.12971, 2023.

\bibitem[DMP24]{geomTCR}
Emanuele Dotto, Kristian Moi, and Irakli Patchkoria, \emph{On the geometric
  fixed points of real topological cyclic homology}, J. Lond. Math. Soc. (2)
  \textbf{109} (2024), no.~2, Paper No. e12862, 68. \MR{4704156}

\bibitem[DMPR21]{THRmodels}
Emanuele Dotto, Kristian Moi, Irakli Patchkoria, and Sune~Precht Reeh,
  \emph{Real topological {H}ochschild homology}, J. Eur. Math. Soc. (JEMS)
  \textbf{23} (2021), no.~1, 63--152. \MR{4186464}

\bibitem[DO19]{DO}
Emanuele Dotto and Crichton Ogle, \emph{{$K$}-theory of {H}ermitian {M}ackey
  functors, real traces, and assembly}, Ann. K-Theory \textbf{4} (2019), no.~2,
  243--316. \MR{3990786}

\bibitem[DPM22]{polynomial}
Emanuele Dotto, Irakli Patchkoria, and Kristian~Jonsson Moi, \emph{Witt
  vectors, polynomial maps, and real topological {H}ochschild homology}, Ann.
  Sci. \'{E}c. Norm. Sup\'{e}r. (4) \textbf{55} (2022), no.~2, 473--535.
  \MR{4411877}

\bibitem[FL91]{FLcrossed}
Zbigniew Fiedorowicz and Jean-Louis Loday, \emph{Crossed simplicial groups and
  their associated homology}, Trans. Amer. Math. Soc. \textbf{326} (1991),
  no.~1, 57--87. \MR{998125}

\bibitem[GH99]{GeisserHesselholt}
Thomas Geisser and Lars Hesselholt, \emph{Topological cyclic homology of
  schemes}, Algebraic {$K$}-theory ({S}eattle, {WA}, 1997), Proc. Sympos. Pure
  Math., vol.~67, Amer. Math. Soc., Providence, RI, 1999, pp.~41--87.
  \MR{1743237}

\bibitem[Gro67]{EGAIV}
A.~Grothendieck, \emph{\'{E}l\'{e}ments de g\'{e}om\'{e}trie alg\'{e}brique.
  {IV}. \'{E}tude locale des sch\'{e}mas et des morphismes de sch\'{e}mas
  {IV}}, Inst. Hautes \'{E}tudes Sci. Publ. Math. (1967), no.~32, 361.
  \MR{238860}

\bibitem[Hes04]{HdRWZp}
Lars Hesselholt, \emph{Topological {H}ochschild homology and the de
  {R}ham-{W}itt complex for {$\Bbb Z_{(p)}$}-algebras}, Homotopy theory:
  relations with algebraic geometry, group cohomology, and algebraic
  {$K$}-theory, Contemp. Math., vol. 346, Amer. Math. Soc., Providence, RI,
  2004, pp.~253--259. \MR{2066502}

\bibitem[Hes15]{Hbig}
\bysame, \emph{The big de {R}ham-{W}itt complex}, Acta Math. \textbf{214}
  (2015), no.~1, 135--207. \MR{3316757}

\bibitem[HHR16]{HHR}
M.~A. Hill, M.~J. Hopkins, and D.~C. Ravenel, \emph{On the nonexistence of
  elements of {K}ervaire invariant one}, Ann. of Math. (2) \textbf{184} (2016),
  no.~1, 1--262. \MR{3505179}

\bibitem[HM97]{WittVect}
Lars Hesselholt and Ib~Madsen, \emph{On the {$K$}-theory of finite algebras
  over {W}itt vectors of perfect fields}, Topology \textbf{36} (1997), no.~1,
  29--101. \MR{1410465 (97i:19002)}

\bibitem[HM04]{IbLarsDeRhamMixed}
\bysame, \emph{On the {D}e {R}ham-{W}itt complex in mixed characteristic}, Ann.
  Sci. \'Ecole Norm. Sup. (4) \textbf{37} (2004), no.~1, 1--43. \MR{2050204}

\bibitem[HNS21]{HNS}
Yonatan Harpaz, Thomas Nikolaus, and Jay Shah, \emph{Real topological cyclic
  homology and normal {L}-theory}, To appear, 2021.

\bibitem[H{\o}g16]{Amalie}
Amalie H{\o}genhaven, \emph{Real topological cyclic homology of spherical group
  rings}, arXiv: 1611.01204, 2016.

\bibitem[Ill79]{Ill}
Luc Illusie, \emph{Complexe de de~{R}ham-{W}itt et cohomologie cristalline},
  Ann. Sci. \'{E}cole Norm. Sup. (4) \textbf{12} (1979), no.~4, 501--661.
  \MR{565469}

\bibitem[Kat82]{Kato}
Kazuya Kato, \emph{Symmetric bilinear forms, quadratic forms and {M}ilnor
  {$K$}-theory in characteristic two}, Invent. Math. \textbf{66} (1982), no.~3,
  493--510. \MR{662605}

\bibitem[Lod87]{LodayDihedral}
Jean-Louis Loday, \emph{Homologies di\'edrale et quaternionique}, Adv. in Math.
  \textbf{66} (1987), no.~2, 119--148. \MR{917736}

\bibitem[Mad94]{IbSurvey}
Ib~Madsen, \emph{Algebraic {$K$}-theory and traces}, Current developments in
  mathematics, 1995 ({C}ambridge, {MA}), Int. Press, Cambridge, MA, 1994,
  pp.~191--321. \MR{1474979}

\bibitem[Mil71]{Milnor}
John Milnor, \emph{Symmetric inner products in characteristic {$2$}}, Prospects
  in mathematics ({P}roc. {S}ympos., {P}rinceton {U}niv., {P}rinceton,
  {N}.{J}., 1970), Ann. of Math. Stud., vol. No. 70, Princeton Univ. Press,
  Princeton, NJ, 1971, pp.~59--75. \MR{347866}

\bibitem[Mil70]{Milnor1}
\bysame, \emph{Algebraic {$K$}-theory and quadratic forms}, Invent. Math.
  \textbf{9} (1969/70), 318--344. \MR{260844}

\bibitem[Mor05]{Morel}
Fabien Morel, \emph{Milnor's conjecture on quadratic forms and mod 2 motivic
  complexes}, Rend. Sem. Mat. Univ. Padova \textbf{114} (2005), 63--101.
  \MR{2207862}

\bibitem[NS18]{NS}
Thomas Nikolaus and Peter Scholze, \emph{On topological cyclic homology}, Acta
  Math. \textbf{221} (2018), no.~2, 203--409. \MR{3904731}

\bibitem[OVV07]{OVV}
D.~Orlov, A.~Vishik, and V.~Voevodsky, \emph{An exact sequence for
  {$K^M_\ast/2$} with applications to quadratic forms}, Ann. of Math. (2)
  \textbf{165} (2007), no.~1, 1--13. \MR{2276765}

\bibitem[Sto11]{Sto}
Martin Stolz, \emph{Equivariant structure on smash powers of commutative ring
  spectra}, Ph.D. thesis, University of Bergen, 2011.

\bibitem[Voe03]{VoeMilnor}
Vladimir Voevodsky, \emph{Motivic cohomology with {${\bf Z}/2$}-coefficients},
  Publ. Math. Inst. Hautes \'Etudes Sci. (2003), no.~98, 59--104. \MR{2031199}

\end{thebibliography}

\end{document}